\documentclass[final,leqno]{siamltex}

\usepackage[leqno]{amsmath}
\usepackage{graphicx}
\usepackage{graphics}
\usepackage{epstopdf}
\usepackage{subfigure}
\usepackage{threeparttable}

\usepackage{bm}
\usepackage{amssymb}
\usepackage{leftidx}
\usepackage{empheq}
\usepackage{color}
\usepackage{makecell}
\usepackage{enumerate}
\allowdisplaybreaks
\renewcommand{\theequation}{\arabic{section}.\arabic{equation}}
\numberwithin{equation}{section}
\newtheorem{remark}{Remark}[section]
\newtheorem{example}{Example}[section]
\title{A novel energy-optimal scalar auxiliary variable (EOP-SAV) approach for gradient flows.
        \thanks{
We would like to acknowledge the assistance of volunteers in putting together this example manuscript and supplement. This work is supported by National Natural Science Foundation of China (Grant Nos: 12001336, 12271302, 12131014).}}

      \author{Zhengguang Liu
             \thanks{School of Mathematics and Statistics, Shandong Normal University, Jinan, China. Email: liuzhg@sdnu.edu.cn.}
             \and
         	Yanrong Zhang
         	\thanks{Department of Applied Mathematics, The Hong Kong Polytechnic University, Hung Hom, Hong Kong. Email: yanrongzhang\_math@163.com.}
             \and
             Xiaoli Li
             \thanks{Corresponding author. School of Mathematics, Shandong University, Jinan, Shandong, 250100, China. Email: xiaomath@sdu.edu.cn.}}
\begin{document}
\maketitle

\begin{abstract}
In recent years, the scalar auxiliary variable (SAV) approach has become very popular and hot in the design of linear, high-order and unconditional energy stable schemes of gradient flow models. However, the nature of SAV-based numerical schemes preserving modified energy dissipation limits its wider application. A relaxation technique to correct the modified energy for the baseline SAV method (RSAV) was proposed by Zhao et al. in \cite{jiang2022improving} and Shen et al. in \cite{zhang2022generalized}. The RSAV approach is unconditionally energy stable with respect to a modified energy that is closer to the original free energy, and provides a much improved accuracy when compared with the SAV approach. In this paper, inspired by the RSAV approach, we propose a novel technique to correct the modified energy of the SAV approach, which can be proved to be an optimal energy approximation.
We construct new high-order implicit-explicit schemes based on the proposed energy-optimal SAV (EOP-SAV) approach. The constructed EOP-SAV schemes not only provide an improved accuracy but also simplify calculation, and can be viewed as the optimal relaxation. We also prove that the numerical schemes based on the EOP-SAV approach are unconditionally energy stable. Compared with the RSAV approach, the proposed EOP-SAV approach does not need introduce any relaxed factors and can share the similar procedure for error estimates.  Several interesting numerical examples have been presented to demonstrate the accuracy and effectiveness of the proposed methods.
\end{abstract}

\begin{keywords}
Scalar auxiliary variable, Gradient flow, Relaxation, Optimal, Error analysis.
\end{keywords}

    \begin{AMS}
         65M12; 35K20; 35K35; 35K55; 65Z05
    \end{AMS}

\pagestyle{myheadings}
\thispagestyle{plain}
\markboth{ZHENGGUANG LIU, YANRONG ZHANG AND XIAOLI LI} {EOP-SAV APPROACH FOR GRADIENT FLOWS}
  \section{Introduction}
Gradient flow models are generally derived from the functional variation of free energy. In general, the free energy $E(\phi)$ contains the sum of an integral phase of a nonlinear functional and a quadratic term:
\begin{equation}\label{intro-e1}
E(\phi)=\frac12(\phi,\mathcal{L}\phi)+E_1(\phi)=\frac12(\phi,\mathcal{L}\phi)+\int_\Omega F(\phi)d\textbf{x},
\end{equation}
where $\mathcal{L}$ is a symmetric non-negative linear operator, and $E_1(\phi)=\int_\Omega F(\phi)d\textbf{x}$ is nonlinear free energy. $F(\phi)$ is the energy density function. The gradient flow from the energetic variation of the above energy functional $E(\phi)$ in \eqref{intro-e1} can be obtained as follows:
\begin{equation}\label{intro-e2}
\displaystyle\frac{\partial \phi}{\partial t}=-\mathcal{G}\mu,\quad\mu=\displaystyle\mathcal{L}\phi+F'(\phi),
\end{equation}
where $\mu=\frac{\delta E}{\delta \phi}$ is the chemical potential. $\mathcal{G}$ is a positive operator. For example, $\mathcal{G}=I$ for the $L^2$ gradient flow and $\mathcal{G}=-\Delta$ for the $H^{-1}$ gradient flow.

It is not difficult to find that the above phase field system satisfies the following energy dissipation law:
\begin{equation*}
\frac{d}{dt}E=(\frac{\delta E}{\delta \phi},\frac{\partial\phi}{\partial t})=-(\mathcal{G}\mu,\mu)\leq0,
\end{equation*}
which is a very important property for gradient flows in physics and mathematics.

Many experts and scholars have considered a series of effective numerical calculation methods to maintain the energy stability of the scheme for different gradient flows. In general, a fully explicit format does not preserve the original structure of the system. Completely implicit methods can guarantee the structure of the model, but such methods may require harsh time step limit to ensure the unique solver, and need to solve nonlinear equations at each step, so they are not efficient in practice. In recent years, the widely used methods mainly include convex splitting method \cite{baskaran2013convergence,eyre1998unconditionally}, stabilized method \cite{chen1998applications,shen2010numerical,xu2006stability}, exponential time difference (ETD) method \cite{du2019maximum,du2021maximum,ju2018energy}, invariant energy quadratization (IEQ) method \cite{yang2018linear,yang2017numerical,zhao2017numerical}, Lagrange multiplier method \cite{cheng2018multiple} and
scalar auxiliary variable (SAV) method \cite{cheng2018multiple,shen2018scalar,shen2019new}, including relaxed SAV method (RSAV) \cite{jiang2022improving,zhang2022generalized} et al.. The SAV approach is a new linear algorithm proposed by Shen et al. \cite{shen2018scalar} to construct the unconditional energy stable schemes of gradient flow models. Nowadays, it caused a great deal of heat in the numerical simulation of nonlinear systems. Many experts and scholars have applied the SAV method to various gradient flow problems and obtained satisfactory simulation results. Later, with the in-depth study of scholars, this method was successfully applied to solve various complex nonlinear problems, such as Navier-Stokes equation \cite{lin2019numerical}, Schr$\ddot{o}$dinger equation \cite{antoine2021scalar}, magneto-hydrodynamics (MHD) model \cite{li2022stability} and so on.

Recently, a series of improved SAV algorithms have been proposed on the basis of baseline SAV methods, known as SAV-type methods or SAV-based methods, by changing the definition of auxiliary variables, adding relaxation factors, introducing Lagrange multipliers and so on. These SAV-type algorithms optimize and enrich the SAV method from different angles and have been successfully applied to solve various gradient flow models. For example, Hou and Xu proposed an extended SAV method in \cite{hou2021robust}, which extends the restriction that the free energy has a lower bound by changing the definition of introducing variable $r(t)$. Yang and Dong \cite{yang2020roadmap} proposed a class of generalized constant positive auxiliary variable method, which completely eliminated the hypothetical condition that free energy has a lower bound and expanded the choice of auxiliary variables. We considered an exponential scalar auxiliary variable method (E-SAV) \cite{liu2020exponential} by taking advantage of the non-negative feature of the exponential function. Qiao et al. \cite{ju2022stabilized} proposed a stabilized E-SAV (sESAV) method to simultaneously preserve the energy dissipation law and maximum bound principle (MBP) in discrete settings. Shen et al. \cite{cheng2018multiple} constructed a Lagrange multiplier method to keep the original energy dissipative law. Jiang et al. \cite{jiang2022improving} present a relaxation technique to construct a relaxed SAV (RSAV) approach to improve the accuracy and consistency noticeably.

In this paper, inspired by the RSAV approach in \cite{jiang2022improving} and R-GSAV approach in \cite{zhang2022generalized}, we propose a novel technique to correct the modified energy of the SAV approach, which can be proved to be an optimal energy approximation. Based on this novel technique, we construct second-order Crank-Nicolson and high-order BDF$k$ unconditionally energy stable EOP-SAV numerical schemes. The constructed EOP-SAV schemes can improve the accuracy while maintaining less computational complexity, and can be viewed as the optimal relaxation. We also prove that the numerical schemes based on the EOP-SAV approach are unconditionally energy stable. Compared with the RSAV approach, the proposed EOP-SAV approach does not need introduce any relaxed factors and can share exactly the similar procedure for error estimates.  Several interesting numerical examples have been presented to demonstrate the accuracy and effectiveness of the proposed methods.

The paper is organized as follows. In Sect.2, we first review the SAV-type method including the traditional SAV and relaxed SAV formulations. In Sect.3, we consider a novel EOP-SAV approach based on second-order Crank-Nicloson scheme. We also consider the high-order BDF$k$ unconditionally energy stable EOP-GSAV numerical schemes based on general SAV scheme in Sect.4.  Finally, in Sect.5, some numerical simulations are demonstrated to verify the accuracy and efficiency of our proposed schemes.
\section{A brief review of the SAV-type approach}

In this section, to give our energy-optimal scalar auxiliary variable (EOP-SAV) approach, we first review the traditional SAV method which considered by Shen et.al. in \cite{shen2018scalar} and the relaxed SAV approach proposed in \cite{jiang2022improving} by Jiang et.al..
\subsection{The traditional SAV approach}
Firstly, we need to assume that the nonlinear free energy $E_1(\phi)$ is bounded from below which means $E_1(\phi)=(F(\phi),1)>-C$ for a positive constant $C$. Introduce a scalar variable $R(t)=\sqrt{E_1(\phi)+C}$ and rewrite the gradient flows \eqref{intro-e2} as the following equivalent system:
\begin{equation}\label{SAV-e1}
   \begin{array}{l}
\displaystyle\frac{\partial \phi}{\partial t}=-\mathcal{G}\mu,\\
\displaystyle\mu=\mathcal{L}\phi+\frac{R(t)}{\sqrt{E_1(\phi)+C}}F'(\phi),\\
\displaystyle\frac{dR}{dt}=\frac{1}{2\sqrt{E_1(\phi)+C}}({F'}(\phi),\phi_t).
   \end{array}
\end{equation}
It is not difficult to obtain the following modified energy dissipative law for above equivalent system:
\begin{equation*}
\aligned
\frac{d}{dt}\widetilde{E}(\phi)=-(\mathcal{G}\mu,\mu)\leq0,
\endaligned
\end{equation*}
where the energy $\widetilde{E}(\phi)=\frac12(\phi,\mathcal{L}\phi)+R^2-C$.

Before giving a semi-discrete formulation, we let $N>0$ be a positive integer and set
\begin{equation*}
\Delta t=T/N,\quad t^n=n\Delta t,\quad \text{for}\quad n\leq N.
\end{equation*}
We discretisize the state variable $\phi$ and the introducing variable $R$ implicitly and discretisize the
energy density function $F'(\phi)$ explicitly to obtain the following $k$th-order implicit-explicit (IMEX) schemes:
\begin{equation}\label{SAV-e2}
   \begin{array}{l}
\displaystyle\frac{\alpha_k\phi^{n+1}-A_k(\phi^n)}{\Delta t}=-\mathcal{G}\mu^{n+1},\\
\displaystyle\mu^{n+1}=\mathcal{L}\phi^{n+1}+\frac{R^{n+1}}{\sqrt{E_1(\widehat{\phi}^{n+1})+C}}F'(\widehat{\phi}^{n+1}),\\
\displaystyle\frac{\alpha_kR^{n+1}-A_k(R^n)}{\Delta t}=\frac{1}{2\sqrt{E_1(\widehat{\phi}^{n+1})+C}}({F'}(\widehat{\phi}^{n+1}),\frac{\alpha_k\phi^{n+1}-A_k(\phi^n)}{\Delta t}).
   \end{array}
\end{equation}
Here $\alpha_k$, $A_k$ and $\widehat{\phi}^{n+1}$ are different for $k$th-order schemes. For example, they can be defined as follows:

First-order:
\begin{equation*}
\aligned
\alpha_k=1,\quad A_k(\phi^n)=\phi^n,\quad \widehat{\phi}^{n+1}=\phi^n,
\endaligned
\end{equation*}

Second-order:
\begin{equation*}
\aligned
\alpha_k=\frac32,\quad A_k(\phi^n)=2\phi^n-\frac12\phi^{n-1},\quad \widehat{\phi}^{n+1}=2\phi^n-\phi^{n-1}.
\endaligned
\end{equation*}
For more details, please see \cite{zhang2022generalized}.

The above IMEX numerical schemes \eqref{SAV-e2} is unconditional energy stable with a modified energy $\mathcal{E}^{n}=\frac12(\mathcal{L}\phi^{n},\phi^{n})+|R^n|^2$ for $k=1$ and $\mathcal{E}^{n}=\frac12\left[(\mathcal{L}\phi^{n},\phi^{n})+(2\mathcal{L}\phi^{n}-\mathcal{L}\phi^{n-1},\phi^{n}-\phi^{n-1})\right]+|2R^{n}-R^{n-1}|^2
+|R^{n}|^2$ for $k=2$.
\subsection{The relaxed SAV (RSAV) approach}
The numerical schemes based on the traditional SAV approach preserve a modified energy dissipative law according to the auxiliary variables instead of the original variables. To overcome this issue, Jiang et.al. \cite{jiang2022improving} consider a relaxed technique to update the numerical variable $R^{n+1}$. The numerical schemes resulting from the RSAV method preserve a quite close original energy dissipative law.

Now, we consider the following second-order Crank-Nicolson scheme based on the relaxed SAV (RSAV) approach (RSAV-CN) proposed in \cite{jiang2022improving}: set $R^0=\sqrt{E_1(\phi_0)+C}$, and compute $\phi^{n+1}$, $R^{n+1}$ via the following two steps:

\textbf{Step I}: Compute $\phi^{n+1}$ and $\widetilde{R}^{n+1}$ by the following semi-implicit Crank-Nicolson scheme:
\begin{equation}\label{RSAV-CN-e1}
   \begin{array}{l}
\displaystyle\frac{\phi^{n+1}-\phi^n}{\Delta t}=-\mathcal{G}\mu^{n+\frac12},\\
\displaystyle\mu^{n+\frac12}=\frac12\mathcal{L}\phi^{n+1}+\frac12\mathcal{L}\phi^{n}+\frac{\widetilde{R}^{n+1}+R^n}{2\sqrt{E_1(\widehat{\phi}^{n+1})+C}}F'(\widehat{\phi}^{n+\frac12}),\\
\displaystyle \frac{\widetilde{R}^{n+1}-R^n}{\Delta t}=\displaystyle\frac{1}{2\sqrt{E_1(\widehat{\phi}^{n+1})+C}}\left(F'(\widehat{\phi}^{n+\frac12}),\frac{\phi^{n+1}-\phi^n}{\Delta t}\right).
   \end{array}
\end{equation}
where $\widehat{\phi}^{n+\frac12}=\frac32\phi^n-\frac12\phi^{n-1}$.

\textbf{Step II}: Update the scalar auxiliary variable $R^{n+1}$ via a relaxation step as
\begin{equation}\label{RSAV-CN-e2}
R^{n+1}=\lambda_0\widetilde{R}^{n+1}+(1-\lambda_0)\sqrt{E_1(\phi^{n+1})+C},\quad \lambda_0\in\mathcal{V}.
\end{equation}
Here is $\mathcal{V}$ a set defined by
\begin{equation}\label{RSAV-CN-e3}
\mathcal{V}=\left\{\lambda|\lambda\in[0,1]~s.t.~|R^{n+1}|^2-|\widetilde{R}^{n+1}|^2\leq\Delta t\eta\left(\mathcal{G}\mu^{n+\frac12},\mu^{n+\frac12}\right),\ R^{n+1}=\lambda\widetilde{R}^{n+1}+(1-\lambda)\sqrt{E_1(\phi^{n+1})+C}\right\}.
\end{equation}

We give the following remark to elaborate the optimal choice for the relaxation parameter $\lambda_0$:
\begin{remark}\label{RSAV-rm1}
The optimal choice for the relaxation parameter $\lambda_0$ can be chosen as the solution of the following optimization problem
\begin{equation}\label{RSAV-CN-e4}
\lambda_0=\min\limits_{\lambda\in[0,1]}\lambda~s.t.\quad a\lambda^2+b\lambda+c\leq0,
\end{equation}
where the coefficients are
\begin{eqnarray*}
 a &=&\left(\widetilde{R}^{n+1}-\sqrt{E_1(\phi^{n+1})+C}\right)^2,\\
 b &=&2\left(\widetilde{R}^{n+1}-\sqrt{E_1(\phi^{n+1})+C}\right)\sqrt{E_1(\phi^{n+1})+C},\\
 c &=&E_1(\phi^{n+1})+C-(\widetilde{R}^{n+1})^2-\Delta t\eta\left(\mathcal{G}\mu^{n+\frac12},\mu^{n+\frac12}\right).
\end{eqnarray*}
If $a=0$, then we set $\lambda_0=0$ to make $R^{n+1}=\widetilde{R}^{n+1}=\sqrt{E_1(\phi^{n+1})+C}$. If $a\neq0$, the solution to \eqref{RSAV-CN-e4} is given as
\begin{equation*}
\lambda_0=\max\{0,\frac{-b-\sqrt{b^2-4ac}}{2a}\}.
\end{equation*}
\end{remark}
\section{The energy-optimal SAV (EOP-SAV) approach}
In this section, we will consider a novel modified SAV method, named EOP-SAV approach which is unconditionally energy stable with respect to a modified energy that is closer to the original free energy than the baseline SAV and RSAV approaches, and provides an improved accuracy. We can also prove the considered method is an optimal technique to modify the dissipative law of the SAV method. The core idea of the RSAV approach is to find a relaxed technique to modify $R^{n+1}$ by the weighted sum of $\widetilde{R}^{n+1}$ and $\sqrt{E_1(\phi^{n+1})+C}$. Actually, in order to make the modified energy of the SAV method as close as possible to the original energy, we only need to modify $(R^{n+1})^2$ to be as close to $\sqrt{E_1(\phi^{n+1})+C}$ as possible. Meanwhile, we also need the modified $(R^{n+1})^2$ to satisfy the dissipative law.
\subsection{The second-order EOP-SAV/CN scheme}
Firstly, we consider the following second-order Crank-Nicolson scheme based on the new described EOP-SAV approach: set $R^0=\sqrt{E_1(\phi_0)+C}$, $\mathcal{E}_1(\phi^n)=E_1(\phi^{n})+C$ and compute $\phi^{n+1}$, $R^{n+1}$ via the following two steps:

\textbf{Step I}: Compute $\phi^{n+1}$ and $\widetilde{R}^{n+1}$ by the following semi-implicit Crank-Nicolson scheme:
\begin{equation}\label{emSAV-CN-e1}
   \begin{array}{l}
\displaystyle\frac{\phi^{n+1}-\phi^n}{\Delta t}=-\mathcal{G}\mu^{n+\frac12},\\
\displaystyle\mu^{n+\frac12}=\frac12\mathcal{L}\phi^{n+1}+\frac12\mathcal{L}\phi^{n}+\frac{\widetilde{R}^{n+1}+R^n}{2\sqrt{\mathcal{E}_1(\widehat{\phi}^{n+\frac12})}}F'(\widehat{\phi}^{n+\frac12}),\\
\displaystyle \frac{\widetilde{R}^{n+1}-R^n}{\Delta t}=\displaystyle\frac{1}{2\sqrt{\mathcal{E}_1(\widehat{\phi}^{n+\frac12})}}\left(F'(\widehat{\phi}^{n+\frac12}),\frac{\phi^{n+1}-\phi^n}{\Delta t}\right).
   \end{array}
\end{equation}
where $\widehat{\phi}^{n+\frac12}=\frac32\phi^n-\frac12\phi^{n-1}$.

\textbf{Step II}: Update the scalar auxiliary variable $R^{n+1}$ as
\begin{equation}\label{emSAV-CN-e2}
R^{n+1}=\min\left\{s^{n+1},\sqrt{\mathcal{E}_1(\phi^{n+1})}\right\},
\end{equation}
where $s^{n+1}=\sqrt{\frac12(\mathcal{L}\phi^{n},\phi^{n})-\frac12(\mathcal{L}\phi^{n+1},\phi^{n+1})+|R^{n}|^2}$ can be obtained easily.

\begin{theorem}\label{emsav-th1}
The update technique of $R^{n+1}$ in \textbf{Step II} \eqref{emSAV-CN-e2} is the optimal choice to  modify the energy dissipative law of the SAV method.
\end{theorem}
\begin{proof}
Taking the inner products of \eqref{emSAV-CN-e1} with $\mu^{n+\frac12}$, $-\frac{\phi^{n+1}-\phi^n}{\Delta t}$ and $\frac{\widetilde{R}^{n+1}+R^n}{2}$ respectively,
we obtain immediately
\begin{equation}\label{emSAV-CN-e3}
\aligned
\displaystyle \left[\frac12(\mathcal{L}\phi^{n+1},\phi^{n+1})+|\widetilde{R}^{n+1}|^2\right]-\left[\frac12(\mathcal{L}\phi^{n},\phi^{n})+|R^{n}|^2\right]=-\Delta t(\mathcal{G}\mu^{n+\frac12},\mu^{n+\frac12})\leq0.
\endaligned
\end{equation}
Obviously, if we modify $R^{n+1}$, the following inequality will be hold:
\begin{equation}\label{emSAV-CN-e4}
\aligned
\displaystyle \left[\frac12(\mathcal{L}\phi^{n+1},\phi^{n+1})+|R^{n+1}|^2\right]-\left[\frac12(\mathcal{L}\phi^{n},\phi^{n})+|R^{n}|^2\right]\leq0.
\endaligned
\end{equation}
We immediately obtain the following inequality
\begin{equation}\label{emSAV-CN-e5}
\aligned
0\leq|R^{n+1}|^2\leq\frac12(\mathcal{L}\phi^{n},\phi^{n})-\frac12(\mathcal{L}\phi^{n+1},\phi^{n+1})+|R^{n}|^2,
\endaligned
\end{equation}
It means that $\frac12(\mathcal{L}\phi^{n},\phi^{n})-\frac12(\mathcal{L}\phi^{n+1},\phi^{n+1})+|R^{n}|^2$ is the least upper bound on $|R^{n+1}|^2$ to satisfy the energy dissipative law.

From \eqref{emSAV-CN-e5}, we obtain $\frac12(\mathcal{L}\phi^{n},\phi^{n})-\frac12(\mathcal{L}\phi^{n+1},\phi^{n+1})+|R^{n}|^2\geq0$. Setting $$s^{n+1}=\sqrt{\frac12(\mathcal{L}\phi^{n},\phi^{n})-\frac12(\mathcal{L}\phi^{n+1},\phi^{n+1})+|R^{n}|^2},$$ we are easy to obtain :

(1). If $\displaystyle s^{n+1}\geq\sqrt{\mathcal{E}_1(\phi^{n+1})}$, we update $R^{n+1}$ by $R^{n+1}=\sqrt{\mathcal{E}_1(\phi^{n+1})}$. It means the modified energy is totally equal to the original energy. Obviously $R^{n+1}=\sqrt{\mathcal{E}_1(\phi^{n+1})}$ is an optimal choice.

(2). If $\displaystyle s^{n+1}<\sqrt{\mathcal{E}_1(\phi^{n+1})}$, we update $R^{n+1}$ by $\displaystyle R^{n+1}=s^{n+1}$.
Noting that $\frac12(\mathcal{L}\phi^{n},\phi^{n})-\frac12(\mathcal{L}\phi^{n+1},\phi^{n+1})+|R^{n}|^2=(s^{n+1})^2$ is the least upper bound on $|R^{n+1}|^2$, thus it is the closest real number to $\mathcal{E}_1(\phi^{n+1})$ which means it is an optimal choice to update $R^{n+1}$.
\end{proof}
\begin{theorem}\label{emSAV-CN-th1}
The second-order EOP-SAV/CN scheme \eqref{emSAV-CN-e1}-\eqref{emSAV-CN-e2} is unconditionally energy stable in the sense that
\begin{equation}\label{emSAV-CN-e6}
\mathcal{\widetilde{E}}(\phi^{n+1})-\mathcal{\widetilde{E}}(\phi^{n})\leq0,
\end{equation}
where $\mathcal{\widetilde{E}}(\phi^{n+1})=\frac12(\mathcal{L}\phi^{n+1},\phi^{n+1})+|R^{n+1}|^2-C$ and we further have the following original dissipative law:
\begin{equation*}
\aligned
\mathcal{E}(\phi^{n+1})\leq\mathcal{E}(\phi^{n}),
\endaligned
\end{equation*}
under the condition of $s^{n+1}\geq\sqrt{\mathcal{E}_1(\phi^{n+1})}$. Here $\mathcal{E}(\phi^{n+1})=\frac12(\mathcal{L}\phi^{n+1},\phi^{n+1})+E_1(\phi^{n+1})$ is the original energy.
\end{theorem}
\begin{proof}
From the equation in \textbf{Step II} of the EOP-SAV/CN scheme \eqref{emSAV-CN-e2}, we immediately obtain
\begin{equation}\label{emSAV-CN-e7}
\aligned
|R^{n+1}|^2\leq (s^{n+1})^2=\frac12(\mathcal{L}\phi^{n},\phi^{n})-\frac12(\mathcal{L}\phi^{n+1},\phi^{n+1})+|R^{n}|^2,
\endaligned
\end{equation}
which means $\mathcal{\widetilde{E}}(\phi^{n+1})-\mathcal{\widetilde{E}}(\phi^{n})\leq0$.

From equation \eqref{emSAV-CN-e2}, we also obtain $|R^{n+1}|^2\leq\mathcal{E}_1(\phi^{n+1})$, then the following inequality is satisfied:
\begin{equation}\label{emSAV-CN-e8}
|R^{n+1}|^2-C\leq E_1(\phi^{n+1}),
\end{equation}
which means $\mathcal{\widetilde{E}}(\phi^{n+1})\leq\mathcal{E}(\phi^{n+1}).$ Noticing that $\mathcal{\widetilde{E}}(\phi^{0})=\mathcal{E}(\phi^{0})$ and from above inequality, we can immediately obtain that
\begin{equation*}
\aligned
\mathcal{\widetilde{E}}(\phi^{n})\leq\mathcal{E}(\phi^{n}),\quad \forall n\geq0.
\endaligned
\end{equation*}
Specially, if $s^{n+1}\geq\sqrt{\mathcal{E}_1(\phi^{n+1})}$, we get $|R^{n+1}|^2=\mathcal{E}_1(\phi^{n+1})$. Then the following equation will hold:
\begin{equation*}
\aligned
\mathcal{E}(\phi^{n+1})=\mathcal{\widetilde{E}}(\phi^{n+1}).
\endaligned
\end{equation*}
Thus, we could have the following original dissipative law:
\begin{equation}\label{emSAV-CN-e9}
\aligned
\mathcal{E}(\phi^{n+1})=\mathcal{\widetilde{E}}(\phi^{n+1})\leq\mathcal{\widetilde{E}}(\phi^{n})\leq\mathcal{E}(\phi^{n}), \quad \text{for}\quad s^{n+1}\geq\sqrt{\mathcal{E}_1(\phi^{n+1})},
\endaligned
\end{equation}
which completes the proof.
\end{proof}
\begin{theorem}
The modified energy $\mathcal{\widetilde{E}}(\phi^{n+1})$ in the proposed EOP-SAV/CN scheme \eqref{emSAV-CN-e1}-\eqref{emSAV-CN-e2} is optimal from a relaxation point of view.
\end{theorem}
\begin{proof}
If $\displaystyle\sqrt{\mathcal{E}_1(\phi^{n+1})}\leq s^{n+1}$,
we have
$$R^{n+1}=\sqrt{\mathcal{E}_1(\phi^{n+1})},$$
which means there is a relaxed factor $\lambda=0$ to let
$$R^{n+1}=\lambda s^{n+1}+(1-\lambda)\sqrt{\mathcal{E}_1(\phi^{n+1})}=\lambda \widetilde{R}^{n+1}+(1-\lambda)\sqrt{\mathcal{E}_1(\phi^{n+1})}.$$ In this case, we obtain
$$\mathcal{\widetilde{E}}(\phi^{n+1})=\frac12(\mathcal{L}\phi^{n+1},\phi^{n+1})+|R^{n+1}|^2-C=\frac12(\mathcal{L}\phi^{n+1},\phi^{n+1})+\mathcal{E}_1(\phi^{n+1})-C=\mathcal{E}(\phi^{n+1}).$$
If $\displaystyle\sqrt{\mathcal{E}_1(\phi^{n+1})}>s^{n+1}$,
we have
$$R^{n+1}=\min\left\{s^{n+1},\sqrt{\mathcal{E}_1(\phi^{n+1})}\right\}=s^{n+1}.$$
Noting that $s^{n+1}\geq\widetilde{R}^{n+1}$ which means
$$\widetilde{R}^{n+1}\leq R^{n+1}<\sqrt{\mathcal{E}_1(\phi^{n+1})},$$
then we immediately obtain that there is a constant $\lambda\in(0,1]$ to satisfy
$$R^{n+1}=\lambda\widetilde{R}^{n+1}+(1-\lambda)\sqrt{\mathcal{E}_1(\phi^{n+1})}.$$ In this case, we have
$$\mathcal{\widetilde{E}}(\phi^{n+1})=\frac12(\mathcal{L}\phi^{n+1},\phi^{n+1})+|R^{n+1}|^2-C<\mathcal{E}(\phi^{n+1}).$$
Notice that $|R^{n+1}|^2=(s^{n+1})^2=\frac12(\mathcal{L}\phi^{n},\phi^{n})-\frac12(\mathcal{L}\phi^{n+1},\phi^{n+1})+|R^{n}|^2$ is the maximum that satisfies the energy dissipative law, thus $\mathcal{\widetilde{E}}(\phi^{n+1})$ is the closest value to $\mathcal{E}(\phi^{n+1})$.
\end{proof}
\begin{theorem}\label{emsav-th2}
The proposed EOP-SAV/CN scheme \eqref{emSAV-CN-e1}-\eqref{emSAV-CN-e2} is second-order accurate in time. Specially, $R^{n+1}=\sqrt{\mathcal{E}_1(\phi(\textbf{x},t^{n+1}))}+O(\Delta t^2)$, it means the modified step in \eqref{emSAV-CN-e2} does not affect the order of accuracy in time.
\end{theorem}
\begin{proof}
If $\displaystyle\sqrt{\mathcal{E}_1(\phi^{n+1})}\leq s^{n+1}$,
we have
$$R^{n+1}=\min\left\{s^{n+1},\sqrt{\mathcal{E}_1(\phi^{n+1})}\right\}=\sqrt{\mathcal{E}_1(\phi^{n+1})}.$$
If $\displaystyle\sqrt{\mathcal{E}_1(\phi^{n+1})}>s^{n+1}$,
we have $R^{n+1}=\min\left\{s^{n+1},\sqrt{\mathcal{E}_1(\phi^{n+1})}\right\}=s^{n+1}$. It means that
$$\widetilde{R}^{n+1}\leq s^{n+1}=R^{n+1}<\sqrt{\mathcal{E}_1(\phi^{n+1})}.$$
Notice that $\widetilde{R}^{n+1}=\sqrt{\mathcal{E}_1(\phi(\textbf{x},t^{n+1}))}+O(\Delta t^2)$ and $\sqrt{\mathcal{E}_1(\phi^{n+1})}=\sqrt{\mathcal{E}_1(\phi(\textbf{x},t^{n+1}))}+O(\Delta t^2)$ we immediate obtain
$$R^{n+1}=\sqrt{\mathcal{E}_1(\phi(\textbf{x},t^{n+1}))}+O(\Delta t^2).$$
\end{proof}
\subsection{The energy-optimal generalized SAV (EOP-GSAV) scheme}
Inspired by the above EOP-SAV/CN scheme, we, in this subsection, construct an energy-optimal generalized SAV (EOP-GSAV) scheme, which not only inherits all the advantages of the GSAV approach, but can also modify the energy as close as the original energy. The detailed high-order EOP-GSAV/BDF$k$ scheme for the system \eqref{intro-e1} can be described as follows: given $R^0=\mathcal{E}(\phi_0)=E(\phi_0)+C>0$, $R^{n-1}$, $R^n$ $\phi^{n-1}$, $\phi^n$, we can update $\phi^{n+1}$ via the following two steps:

\textbf{Step I}: Compute $\phi^{n+1}$ and $\widetilde{R}^{n+1}$ by the following GSAV/BDF$k$ scheme:
\begin{equation}\label{emGSAV-e1}
   \begin{array}{l}
\displaystyle\frac{\alpha_k\overline{\phi}^{n+1}-A_k(\phi^n)}{\Delta t}=-\mathcal{G}\mu^{n+1},\\
\displaystyle\mu^{n+1}=\mathcal{L}\overline{\phi}^{n+1}+F'(\widehat{\phi}^{n+1}),\\
\displaystyle\frac{\widetilde{R}^{n+1}-R^n}{\Delta t}=\displaystyle-\frac{\widetilde{R}^{n+1}}{\mathcal{E}(\overline{\phi}^{n+1})}(\mathcal{G}\mu^{n+1},\mu^{n+1}),\\
\displaystyle\xi^{n+1}=\frac{\widetilde{R}^{n+1}}{\mathcal{E}(\overline{\phi}^{n+1})}\\
\phi^{n+1}=\left[1-(1-\xi^{n+1})^{k+1}\right]\overline{\phi}^{n+1}.
   \end{array}
\end{equation}
where $\alpha_k$, $\widehat{\phi}^{n+1}$ and the operator $A_k$ can be chosen as follows:

$k=1$:
\begin{equation*}
\aligned
\alpha_k=1,\quad A_k(\phi^n)=\phi^n,\quad \widehat{\phi}^{n+1}=\phi^n,
\endaligned
\end{equation*}

$k=2$:
\begin{equation*}
\aligned
\alpha_k=\frac32,\quad A_k(\phi^n)=2\phi^n-\frac12\phi^{n-1},\quad \widehat{\phi}^{n+1}=2\phi^n-\phi^{n-1}.
\endaligned
\end{equation*}

$k=3$:
\begin{equation*}
\aligned
\alpha_k=\frac{11}{6},\quad A_k(\phi^n)=3\phi^n-\frac32\phi^{n-1}+\frac13\phi^{n-2},\quad \widehat{\phi}^{n+1}=3\phi^n-3\phi^{n-1}+\phi^{n-2}.
\endaligned
\end{equation*}
For more details, please see [25].

\textbf{Step II}: Update the scalar auxiliary variable $R^{n+1}$ as
\begin{equation}\label{emGSAV-e2}
R^{n+1}=\min\left\{R^n,\mathcal{E}(\phi^{n+1})\right\}.
\end{equation}

\begin{theorem}
The update technique of $R^{n+1}$ in \textbf{Step II} \eqref{emGSAV-e2} is the optimal choice to  modify the energy dissipative law of the SAV method.
\end{theorem}
\begin{proof}
From the \textbf{Step I} in the EOP-GSAV/BDF$k$ scheme \eqref{emGSAV-e1}, we are easy to obtain the following energy inequality
\begin{equation}\label{emGSAV-e3}
\aligned
\widetilde{R}^{n+1}-R^n\leq\displaystyle -\frac{\Delta t(\mathcal{G}\mu^{n+1},\mu^{n+1})R^n}{\mathcal{E}(\overline{\phi}^{n+1})+\Delta t(\mathcal{G}\mu^{n+1},\mu^{n+1})}\leq0,
\endaligned
\end{equation}
It means that $R^{n}$ is the least upper bound on $\widetilde{R}^{n+1}$ to keep the energy dissipative law.

We explain below why $R^{n+1}$ in \textbf{Step II} \eqref{emGSAV-e2} is the optimal choice to modify the energy dissipative law.

(1). If $R^{n}\geq\mathcal{E}(\phi^{n+1})$, we update $R^{n+1}$ by $R^{n+1}=\mathcal{E}(\phi^{n+1})$. It means the modified energy is totally equal to the original energy. Obviously $R^{n+1}=\mathcal{E}(\phi^{n+1})$ is an optimal choice.

(2). If $R^{n}<\mathcal{E}(\phi^{n+1})$, we update $R^{n+1}$ by $R^{n+1}=R^{n}$.
Noting that $R^{n}$ is the least upper bound on $R^{n+1}$ and meanwhile the energy dissipative $R^{n+1}\leq R^n$ holds, thus it is the closest real number to $\mathcal{E}(\phi^{n+1})$ which means it is an optimal choice to update $R^{n+1}$.
\end{proof}
\begin{theorem}\label{emGSAV-th1}
The high-order EOP-GSAV/BDF$k$ scheme \eqref{emGSAV-e1}-\eqref{emGSAV-e2} is unconditionally energy stable in the sense that
\begin{equation}\label{emGSAV-e6}
\mathcal{\widetilde{E}}(\phi^{n+1})-\mathcal{\widetilde{E}}(\phi^{n})\leq0,
\end{equation}
where $\mathcal{\widetilde{E}}(\phi^{n+1})=R^{n+1}-C$ is a modified energy. We further have the following original dissipative law:
\begin{equation*}
\aligned
E(\phi^{n+1})\leq E(\phi^{n}),
\endaligned
\end{equation*}
under the condition of $\mathcal{E}(\phi^{n+1})\leq R^n$. Here $E(\phi^{n})$ is the original energy.
\end{theorem}
\begin{proof}
The equation \eqref{emGSAV-e2} indicates that $R^{n+1}\leq R^n$, then we immediately obtain
\begin{equation}\label{emGSAV-e7}
\aligned
\mathcal{\widetilde{E}}(\phi^{n+1})-\mathcal{\widetilde{E}}(\phi^{n})\leq0,
\endaligned
\end{equation}

From equation \eqref{emGSAV-e2}, we can also obtain $R^{n+1}\leq\mathcal{E}(\phi^{n+1})=E(\phi^{n+1})+C$, then the following inequality is satisfied:
\begin{equation}\label{emGSAV-e8}
R^{n+1}-C\leq E(\phi^{n+1}),
\end{equation}
which means $\mathcal{\widetilde{E}}(\phi^{n+1})\leq E(\phi^{n+1})$ and $\mathcal{\widetilde{E}}(\phi^{n})\leq E(\phi^{n})$, $\forall n\geq0$.

If $\mathcal{E}(\phi^{n+1})\leq R^n$, we get $R^{n+1}=\mathcal{E}(\phi^{n+1})=E(\phi^{n+1})+C$, the following equation will hold:
\begin{equation*}
\aligned
\mathcal{\widetilde{E}}(\phi^{n+1})=E(\phi^{n+1}).
\endaligned
\end{equation*}
Obviously, the following original dissipative law under the condition $\mathcal{E}(\phi^{n+1})\leq R^n$  will satisfy:
\begin{equation}\label{emGSAV-e9}
\aligned
E(\phi^{n+1})=\mathcal{\widetilde{E}}(\phi^{n+1})\leq\mathcal{\widetilde{E}}(\phi^{n})\leq E(\phi^{n}),
\endaligned
\end{equation}
which completes the proof.
\end{proof}
\begin{theorem}
The \textbf{Step II} in the EOP-GSAV/BDF$k$ scheme \eqref{emGSAV-e1}-\eqref{emGSAV-e2} can be rewritten as an equation with a relaxed factor $R^{n+1}=\lambda\widetilde{R}^{n+1}+(1-\lambda)\mathcal{E}(\phi^{n+1})$. Furthermore, it can be viewed as the optimal choice for the relaxation. Naturally, we can derive the following error estimate:
\begin{equation}
\|\phi^{n+1}-\phi(\textbf{x},t^{n+1})\|_{H^2}\leq C(\Delta t)^k, \quad n+1\leq \frac{T}{\Delta t}.
\end{equation}
\end{theorem}
\begin{proof}
Firstly, from the \textbf{Step II} \eqref{emGSAV-e2}, if $R^n\geq\mathcal{\widetilde{E}}(\phi^{n+1})$, we have $$R^{n+1}=\min\left\{R^n,\mathcal{E}(\phi^{n+1})\right\}=\mathcal{E}(\phi^{n+1}),$$
which means there is a relaxed factor $\lambda=0$ to make
$$R^{n+1}=\lambda\widetilde{R}^{n+1}+(1-\lambda)\mathcal{E}(\phi^{n+1}).$$
Secondly, if $R^n<\mathcal{\widetilde{E}}(\phi^{n+1})$, we obtain
$$R^{n+1}=\min\left\{R^n,\mathcal{E}(\phi^{n+1})\right\}=R^n.$$
Noting that $\widetilde{R}^{n+1}\leq R^n$, then we have $R^{n+1}\in[\widetilde{R}^{n+1},\mathcal{\widetilde{E}}(\phi^{n+1})).$ We immediately obtain there is a constant $\lambda\in(0,1]$ to satisfy
$$R^{n+1}=\lambda\widetilde{R}^{n+1}+(1-\lambda)\mathcal{E}(\phi^{n+1}).$$
Then, we will explain the modified energy $\mathcal{\widetilde{E}}(\phi^{n+1})$ is optimal from a relaxation point of view. If $R^n\geq\mathcal{\widetilde{E}}(\phi^{n+1})$, we have $R^{n+1}=\mathcal{E}(\phi^{n+1})$. It means the modified energy $\mathcal{\widetilde{E}}(\phi^{n+1})$ is equal to the original energy. If $R^n<\mathcal{\widetilde{E}}(\phi^{n+1})$, we have $R^{n+1}=R^n<\mathcal{\widetilde{E}}(\phi^{n+1})$. Considering that $R^n$ is the maximum that satisfies the energy dissipative law, thus the modified energy $\mathcal{\widetilde{E}}(\phi^{n+1})=R^{n+1}$ is the closest value to the original energy.

Next, we will give the error estimate of the EOP-GSAV/BDF$k$ scheme \eqref{emGSAV-e1}-\eqref{emGSAV-e2}. By using the same procedure as the error estimate in \cite{zhang2022generalized}, especially for (4.27)-(4.40), we can obtain a similar result
\begin{equation}
\|\phi^{n+1}-\phi(\textbf{x},t^{n+1})\|_{H^2}\leq C(\Delta t)^k, \quad n+1\leq \frac{T}{\Delta t}.
\end{equation}
\end{proof}
\section{Examples and discussion}
In this section, we consider some numerical examples to illustrate the simplicity and efficiency of our proposed method. In all considered examples, we consider the periodic boundary conditions and use a Fourier spectral method in space.

\begin{example}\label{ex:AC}
\rm
The following Allen-Cahn equation is under our consideration,
\begin{equation}\label{eq:Allen-Cahn}
	\frac{\partial \phi}{\partial t}=M\left(\alpha_{0} \Delta \phi+\left(1-\phi^{2}\right) \phi\right),
\end{equation}
subject to periodic boundary conditions.

{\em Case A.} We give the exact solution
\begin{equation} \label{eq:AC-CH-exact-solution-example}
	\phi(x, y, t)=\exp (\sin (\pi x) \sin (\pi y)) \sin (t),
\end{equation}
by introducing an external force $f$ into \eqref{eq:Allen-Cahn} in the domain $\Omega=(0, 2)^{2}$.
We set the values of the parameters $M$ and $\alpha_{0}$ to $1$ and $0.01^2$, respectively.
To ensure that the spatial discretization error is much smaller than the time discretization error, we adopt $64^2$ Fourier modes for space discretization.

In Fig.\,\ref{Fig:AC-order-test-CN-BDF}, we present the $L^2$-norm error convergence rates at $T = 0.5$ obtained using the Crank-Nicolson (CN) scheme and BDF$k$ ($k=1,2,3,4$) schemes. Our observations are as follows:
(\romannumeral1) The expected convergence rates are achieved for all cases;
(\romannumeral2) For BDF$1$ and BDF$2$ schemes, the errors of EOP-GSAV schemes are substantially smaller than those of GSAV schemes;
(\romannumeral3) For BDF$3$ and BDF$4$ schemes, the margin of improvement is not as significant as for lower-order schemes.

Meanwhile, Fig.\,\ref{Fig:AC-zeta-E1} displays the relaxation parameter $\lambda_0$ evolution using the R-SAV/CN scheme, as well as the evolution of the difference between the original energy of the nonlinear part and $s^{n+1} = \sqrt{\frac12(\mathcal{L}\phi^{n},\phi^{n})-\frac12(\mathcal{L}\phi^{n+1},\phi^{n+1})+|R^{n}|^2}$ obtained using the EOP-SAV/CN scheme with a time step of $\Delta t=0.01$.
It is observed that the value of $\mathcal{E}_1(\phi^{n+1})-s^{n+1}$ is always negative, indicating that $R^{n+1}=\mathcal{E}_1(\phi^{n+1})$ at each time step. Furthermore, the modified results obtained using the EOP-SAV/CN scheme are identical to those obtained using the R-SAV/CN scheme, with $\lambda_0$ remaining constantly equal to $0$.
Furthermore, the errors of EOP-SAV/CN scheme are marginally higher than that of SAV/CN scheme, which means that the closer the modified energy is to the original energy may not necessarily result in a smaller error in the solution.

Fig.\ref{Fig:AC-zeta-E} illustrates the evolution of the relaxation parameter $\lambda_0^{n+1}$ using the R-GSAV/BDF$2$ scheme and the difference between the original energy and the modified energy using the EOP-GSAV/BDF$2$ scheme with a time step of $\Delta t=0.01$.
We observe that, except for an initial time interval, $\lambda_0^{n+1}$ remains zero, and the value of $\mathcal{E}(\phi^{n+1})-R^{n}$ is consistently negative. This indicates that the modified energy of the EOP-GSAV/BDF$2$ scheme is closer to that of the R-GSAV/BDF$2$ scheme in this case. The third energy contrast diagram in Fig.\ref{Fig:AC-zeta-E} further confirms this point.
Furthermore, we report the $L^2$-norm errors for different schemes: GSAV: $5.5896E-05$, R-GSAV: $4.3074E-05$, and EOP-GSAV: $4.3071E-05$. These results show that the EOP-GSAV scheme achieves a slightly lower error compared to the GSAV and R-GSAV schemes.

\begin{figure}[htbp]
	\centering
\subfigure[]
{
\includegraphics[width=7cm,height=6cm]{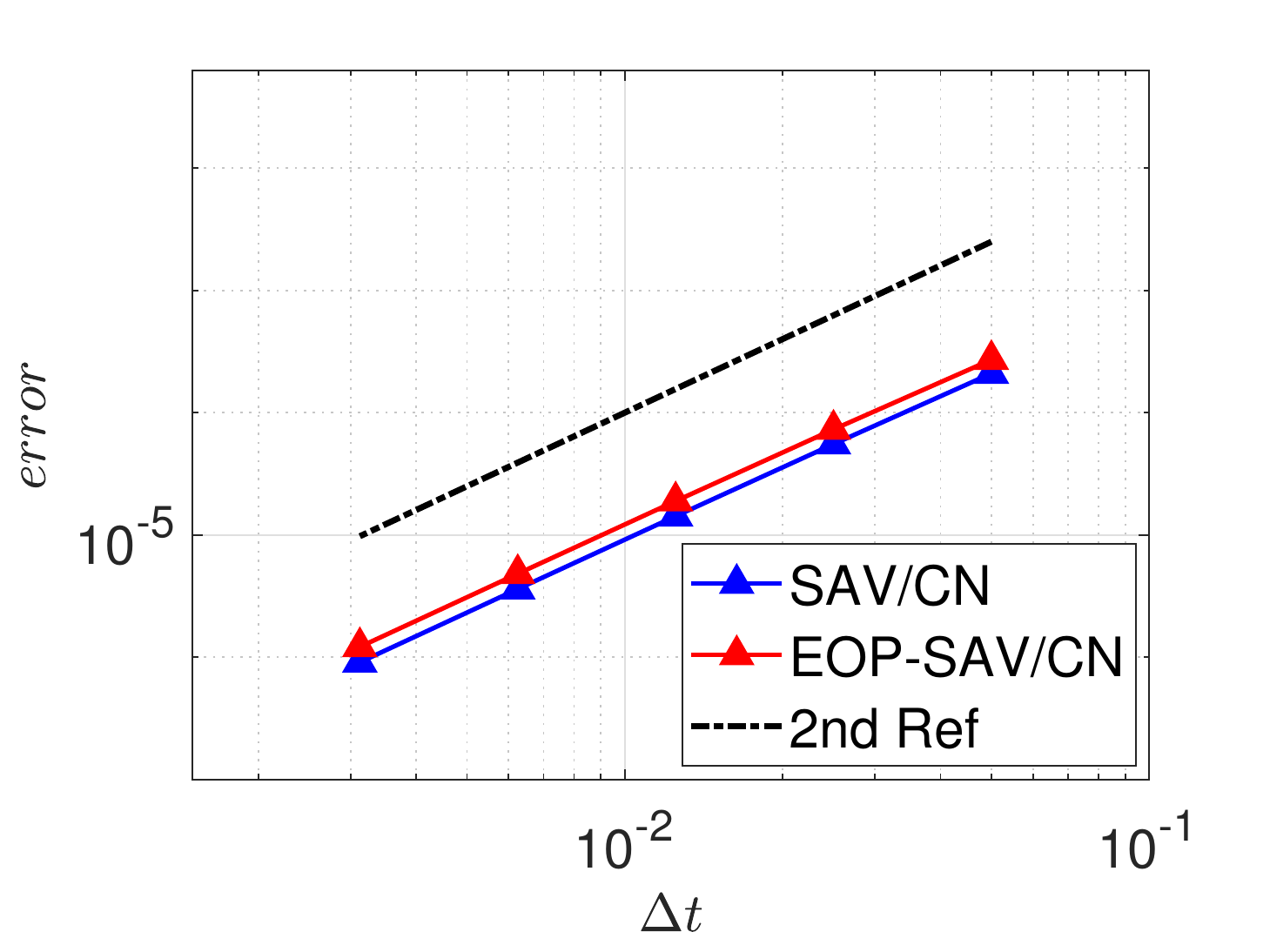}
}
\subfigure[]
{
\includegraphics[width=7cm,height=6cm]{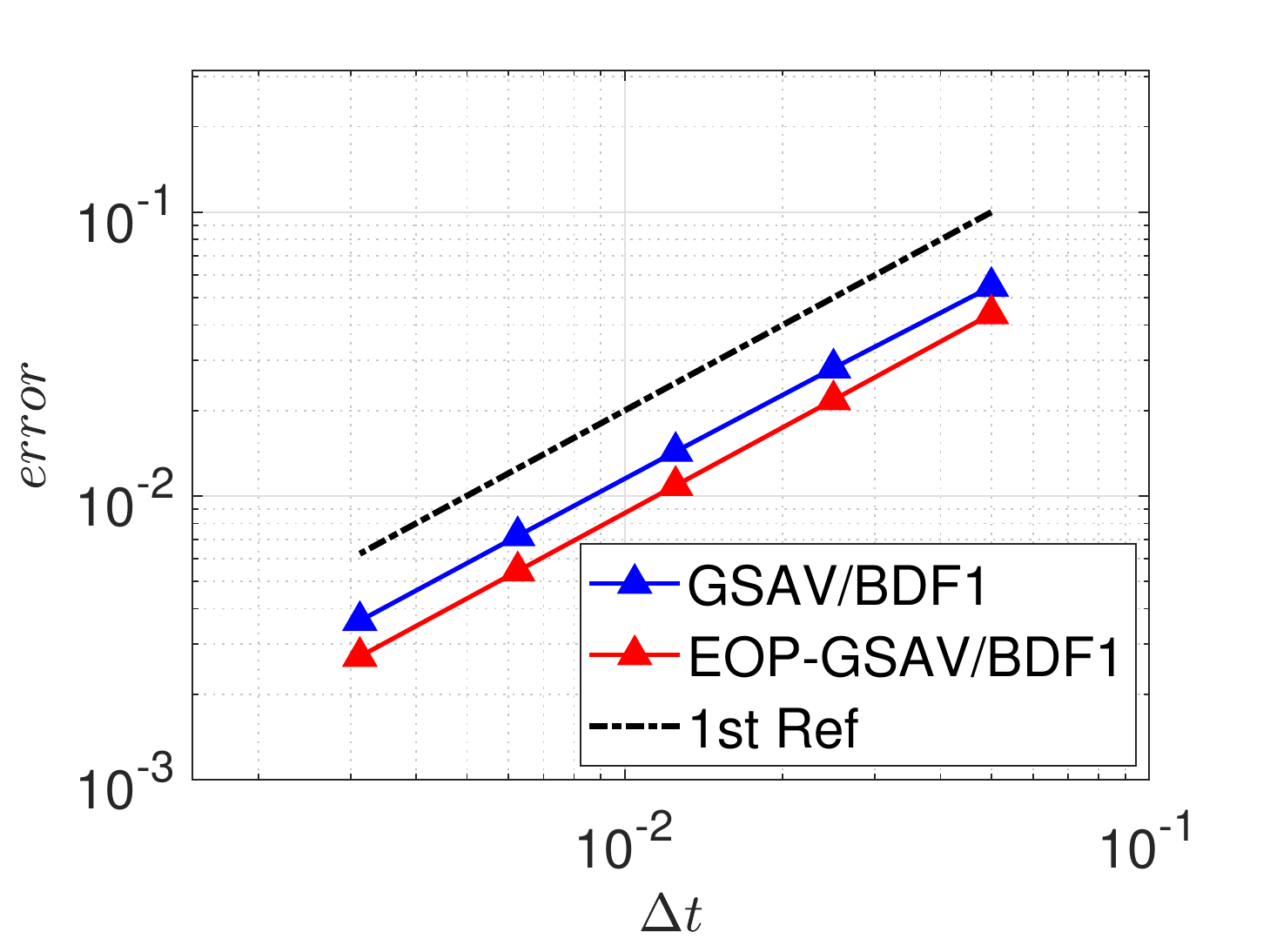}
}
\quad
\subfigure[]
{
\includegraphics[width=7cm,height=6cm]{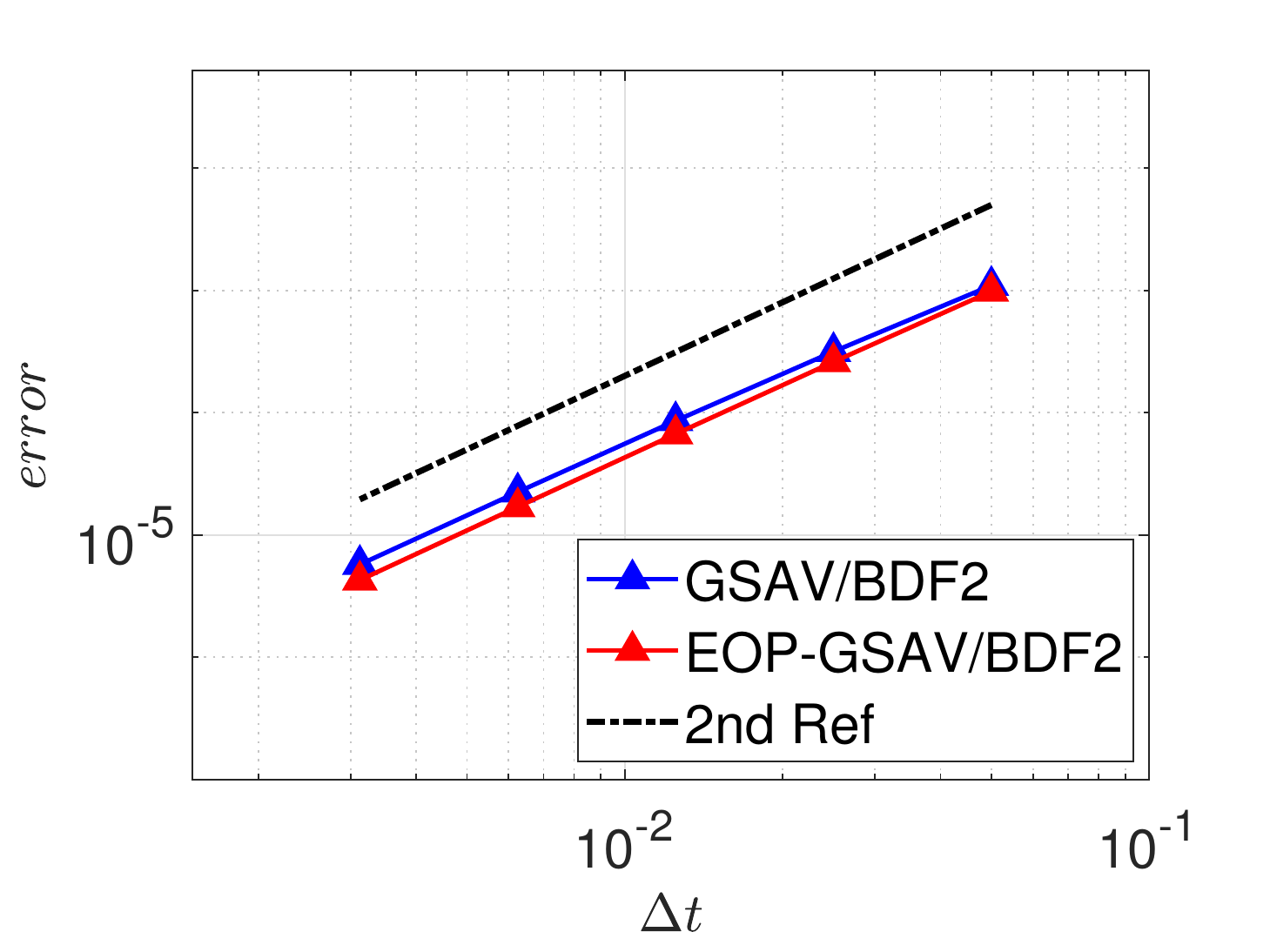}
}
\subfigure[]{
\includegraphics[width=7cm,height=6cm]{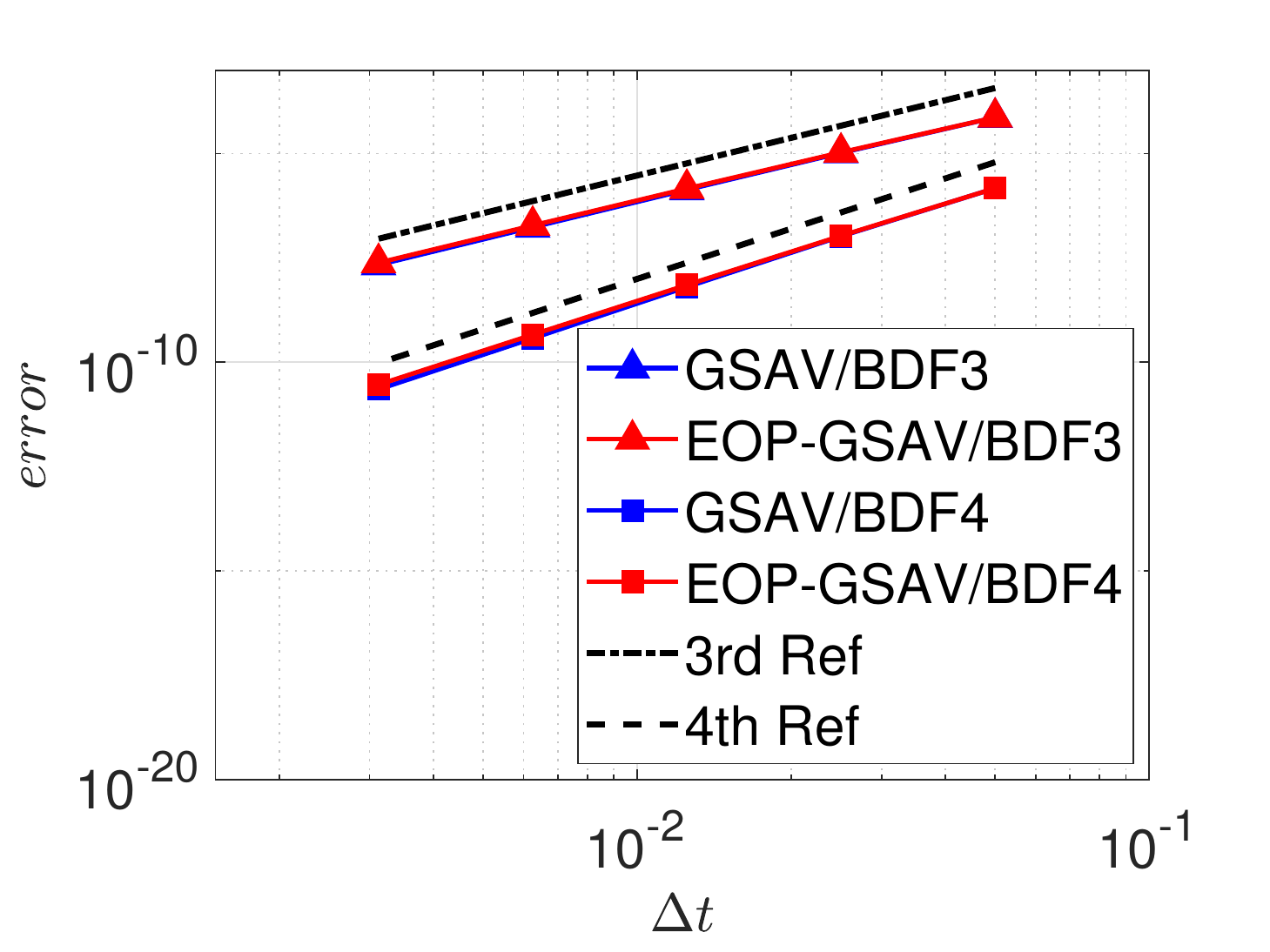}
}
	\caption{Example\,\ref{ex:AC}(Case A). Convergence rates for Allen-Cahn equation using various schemes. (a): CN; (b): BDF$1$; (c): BDF$2$; (d): BDF$k$, $(k=3,4)$.}
	\label{Fig:AC-order-test-CN-BDF}
\end{figure}

\begin{figure}[htbp]
	\centering
	\begin{minipage}{0.4\textwidth}
		\centering
		\includegraphics[width=5.3cm]{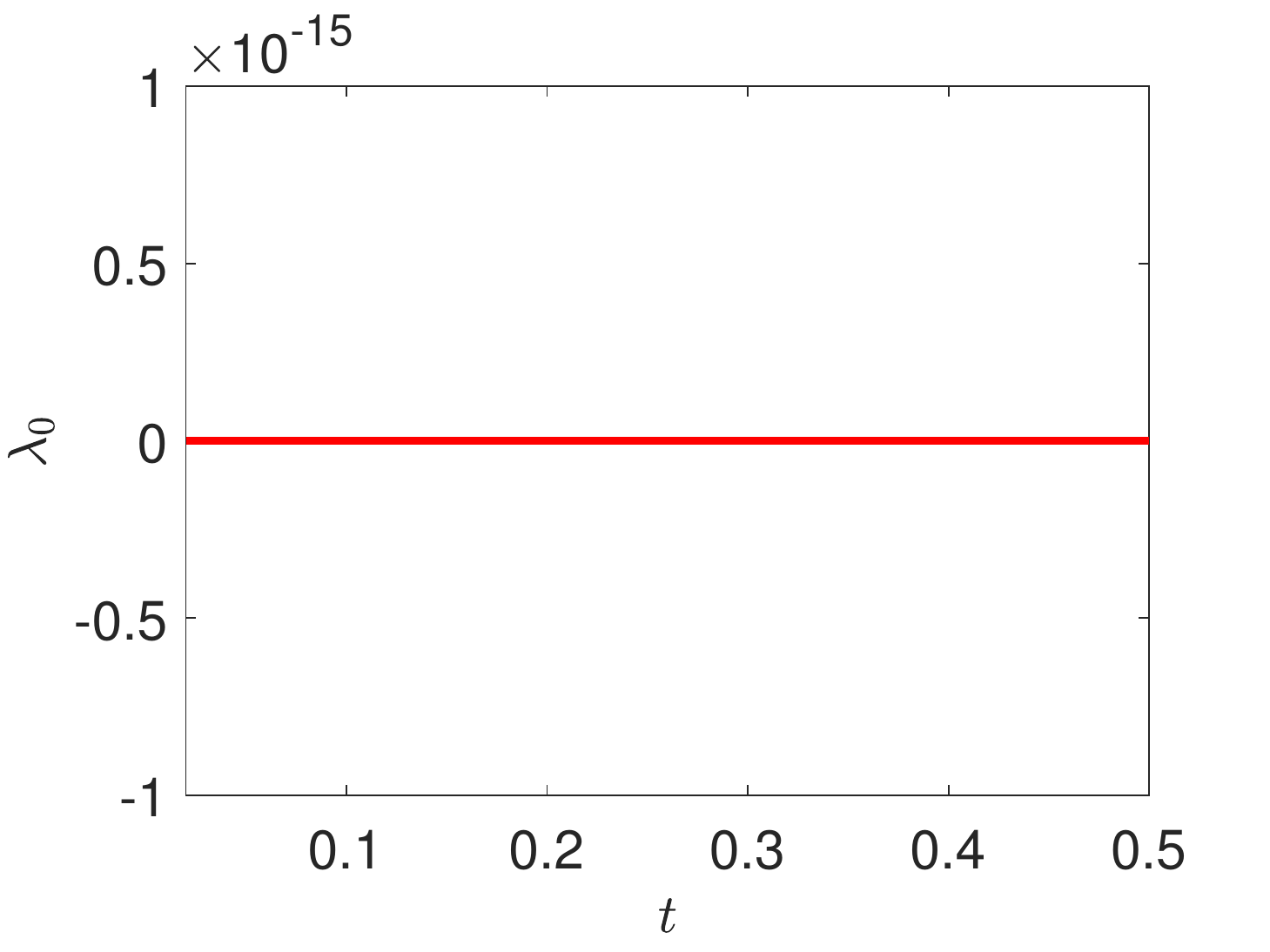}
	\end{minipage}
	\begin{minipage}{0.4\textwidth}
		\centering
		\includegraphics[width=5.3cm]{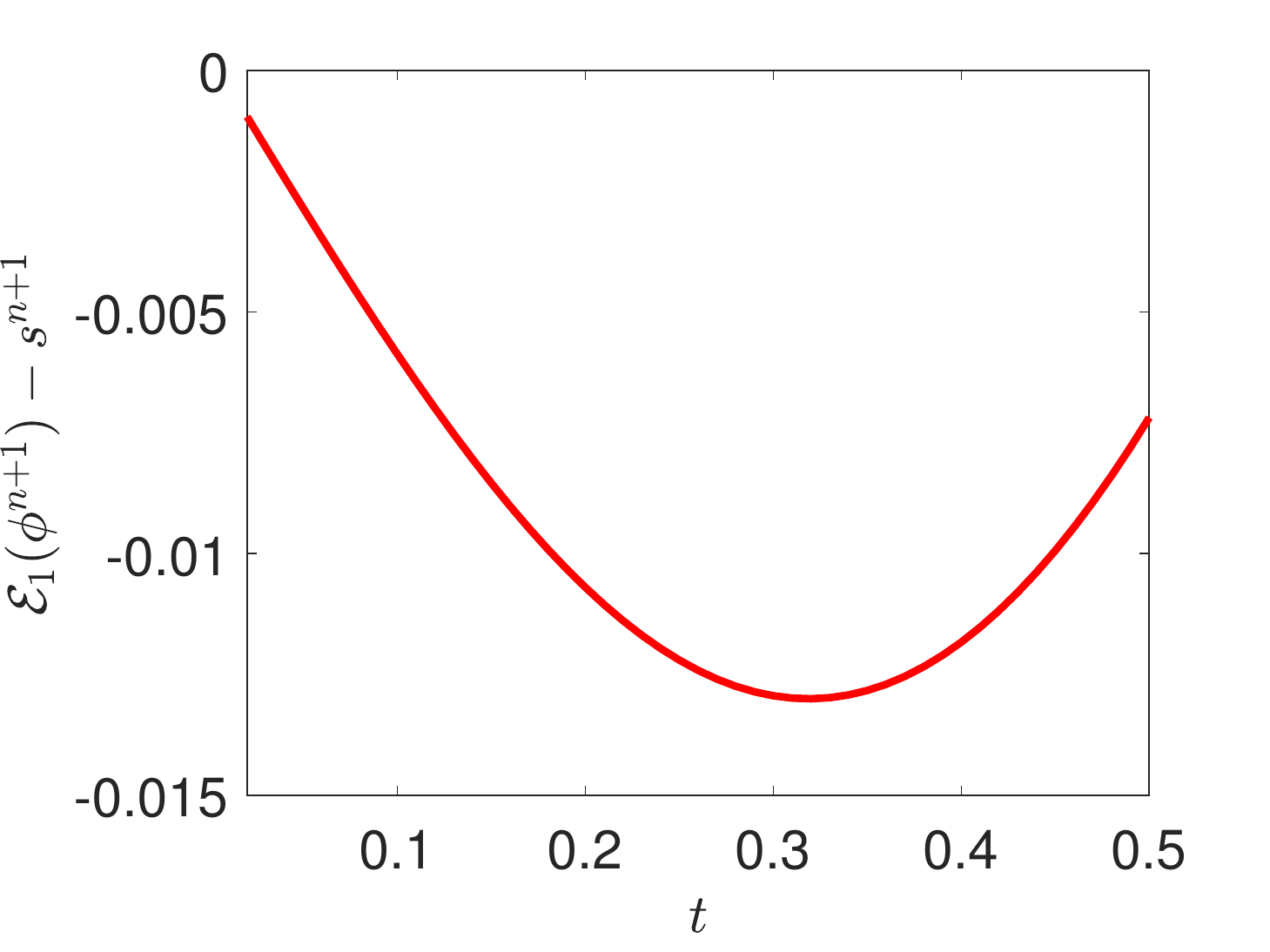}
	\end{minipage}
	\caption{Example\,\ref{ex:AC}(Case A). First: evolution of relaxation $\lambda_{0}$ using R-SAV/CN scheme with $\Delta t=1e-2$; Second: evolution of the difference between the original energy of nonlinear part and $s^{n+1}$ using EOP-SAV/CN scheme with $\Delta t=1e-2$.}
	\label{Fig:AC-zeta-E1}
\end{figure}

\begin{figure}[htbp]
	\centering
\subfigure{
\includegraphics[width=5cm,height=5cm]{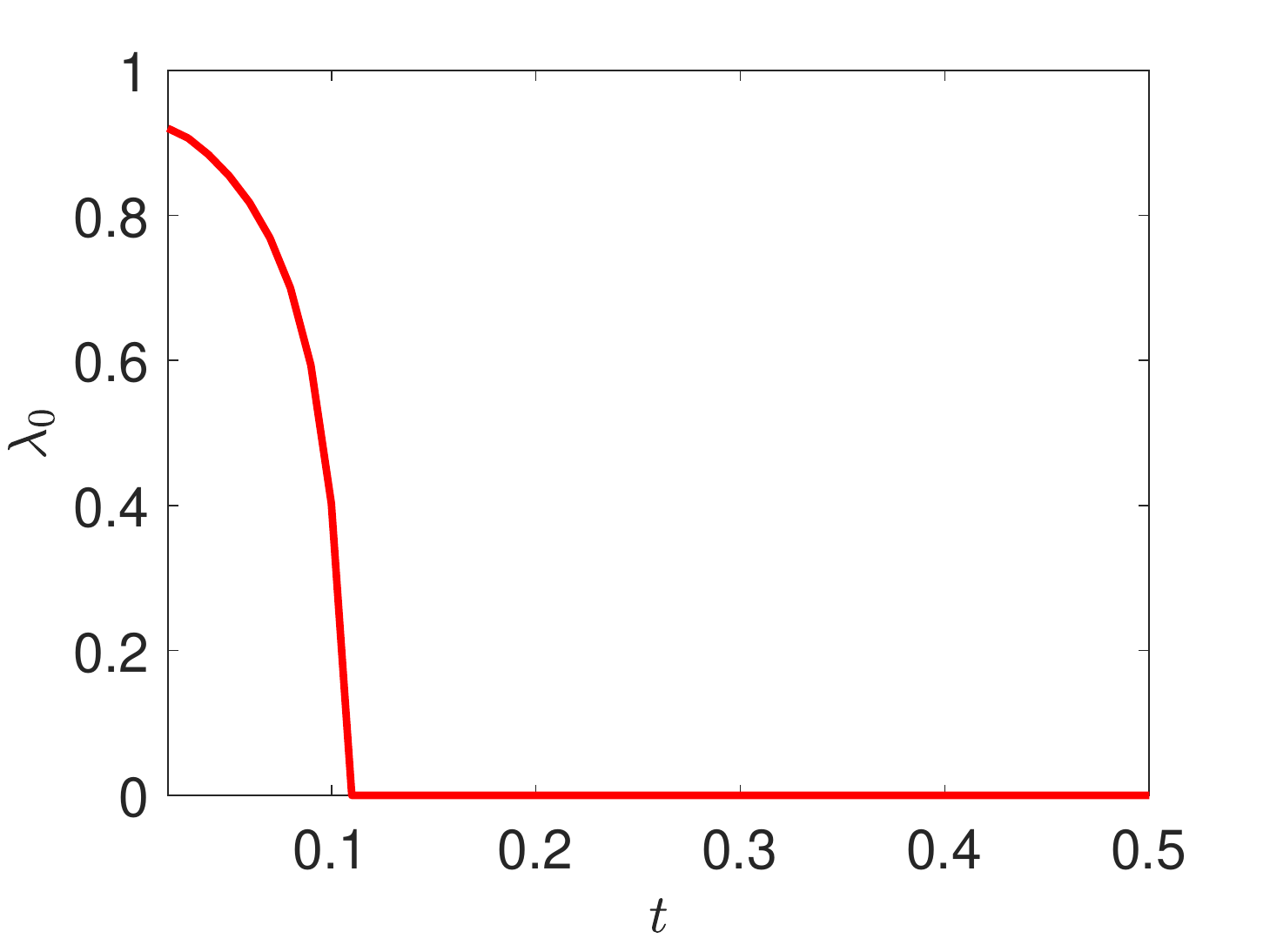}
}
\subfigure
{
\includegraphics[width=5cm,height=5cm]{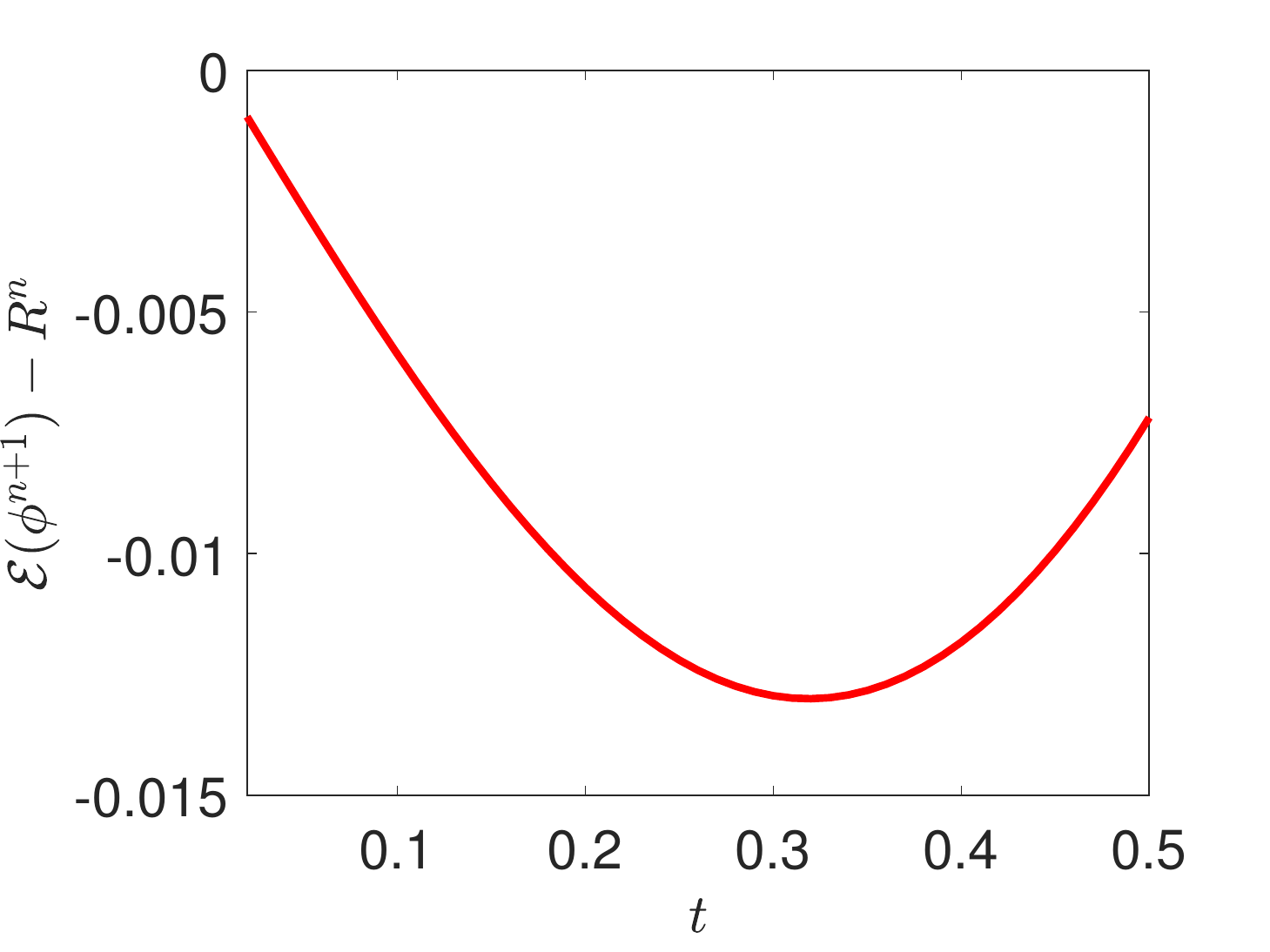}
}
\subfigure
{
\includegraphics[width=5cm,height=5cm]{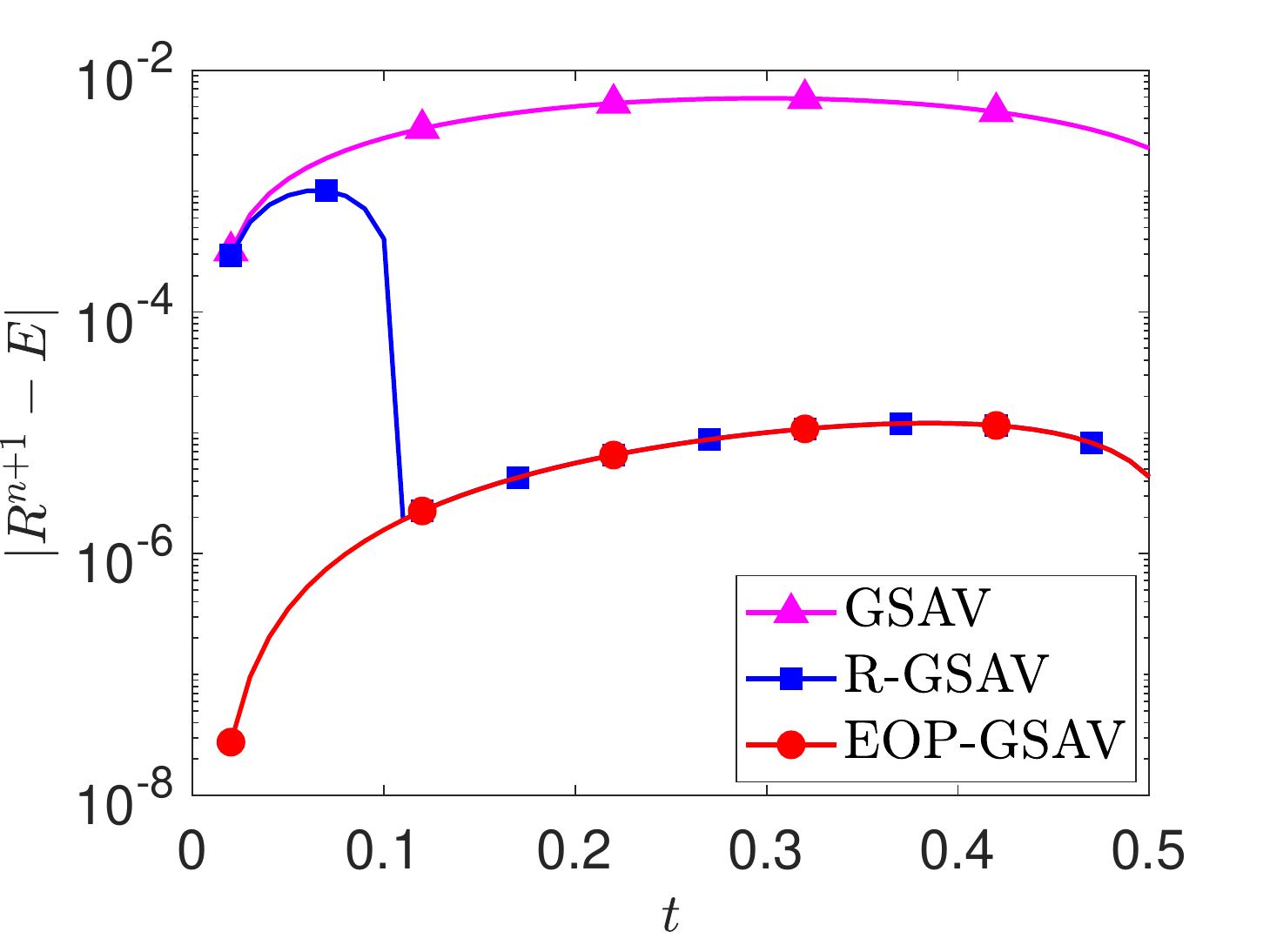}
}
	\caption{Example\,\ref{ex:AC}(Case A). First: evolution of relaxation $\lambda_{0}$ using R-GSAV/BDF$2$ scheme with $\Delta t=0.01$; Second: evolution of the difference between the original energy and the modified energy using EOP-GSAV/BDF$2$ scheme with $\Delta t=0.01$. Third: a comparison of energy for GSAV/BDF2, R-GSAV/BDF$2$ and EOP-GSAV/BDF$2$ schemes. }
	\label{Fig:AC-zeta-E}
\end{figure}

{\em Case B.}  We choose the initial condition as            \begin{equation}\label{eq:AC-CH-initial-condition-star-shape}
	\begin{aligned}
		& \phi(x, y)=\tanh \frac{1.5+1.2 \cos (6 \theta)-2 \pi r}{\sqrt{2\alpha}}, \\
		& \theta=\arctan \frac{y-0.5 L_{y}}{x-0.5 L_{x}}, \quad r=\sqrt{\left(x-\frac{L_{x}}{2}\right)^{2}+\left(y-\frac{L_{y}}{2}\right)^{2}},
	\end{aligned}
\end{equation}
where $(\theta,r)$ are the polar coordinates of $(x, y)$.
We set $\Omega=[0, L_x]\times[0, L_{y}]$ with $L_x=L_y=1$ and the other parameters are $\alpha_{0}=0.01^2, M=1$ and $128^2$ Fourier modes.
We use the results of the semi-implicit/BDF$2$ scheme with $\Delta t = 1e-5$ as the reference solution.
In Fig.\,\ref{Fig:AC-star-shape-energy-xi}, we present  a comparison of energy (first) and energy error (second) and error of $\xi^{n+1}$ (third) of GSAV/BDF$2$ and EOP-GSAV/BDF$2$ scheme with $\Delta t = 1e-3$.
Fig.\,\ref{Fig:AC-star-shape-emGSAV} presents the evolution of Allen-Cahn equation obtained by EOP-GSAV/BDF$2$ scheme with $\Delta t=1e-3$.

\begin{figure}[htbp]
	\centering
\subfigure{
\includegraphics[width=5cm,height=5cm]{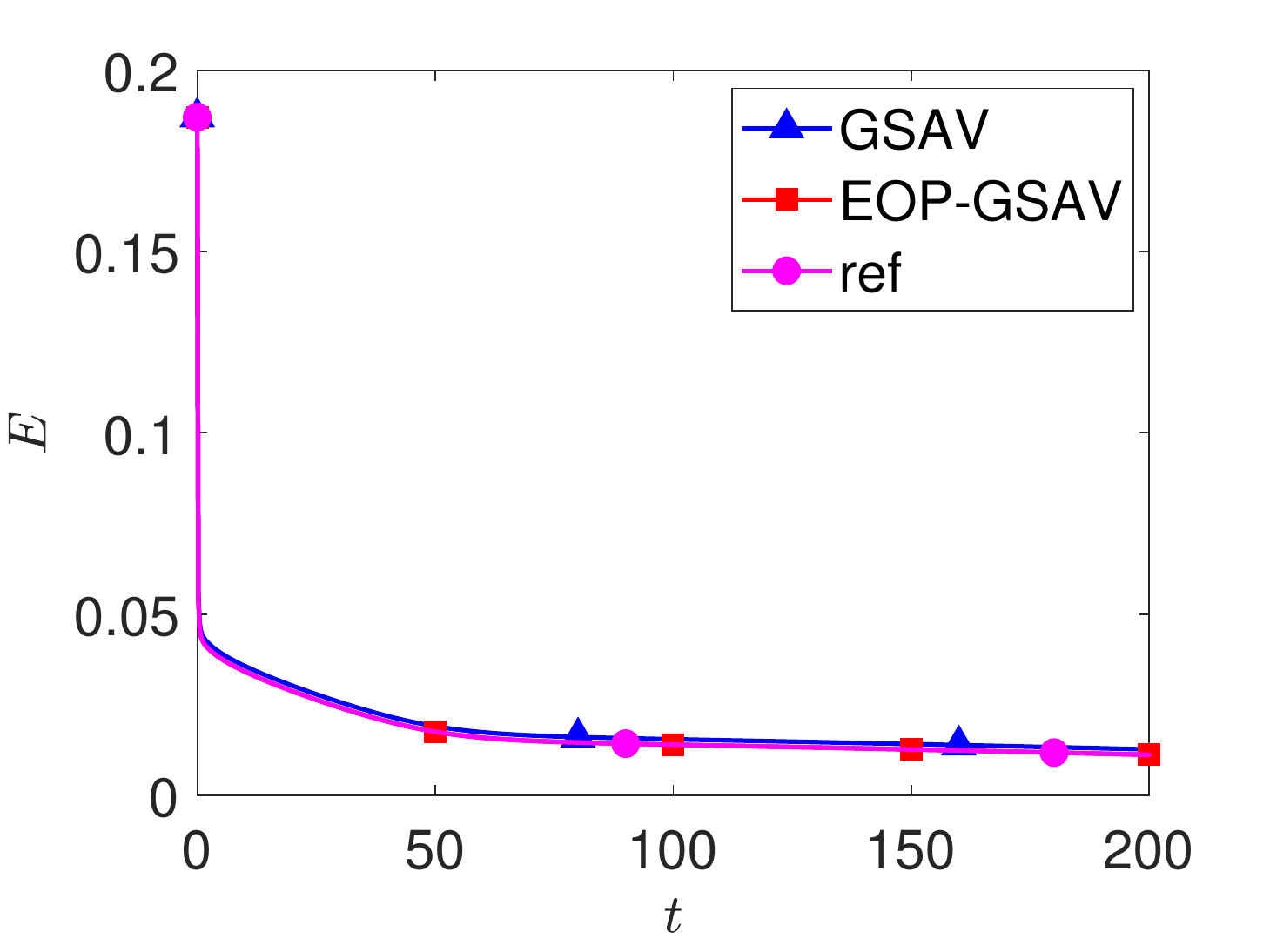}
}
\subfigure
{
\includegraphics[width=5cm,height=5cm]{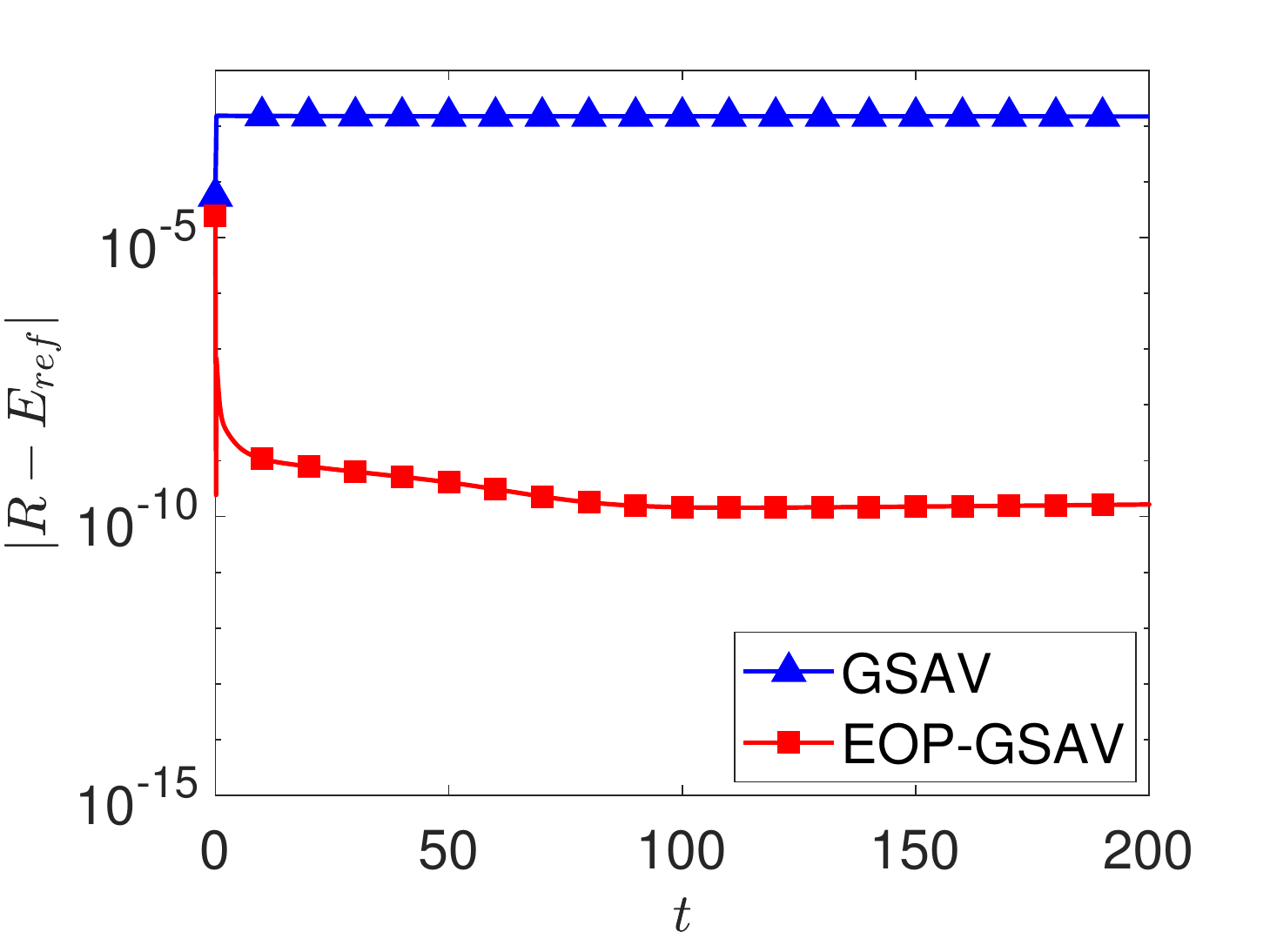}
}
\subfigure
{
\includegraphics[width=5cm,height=5cm]{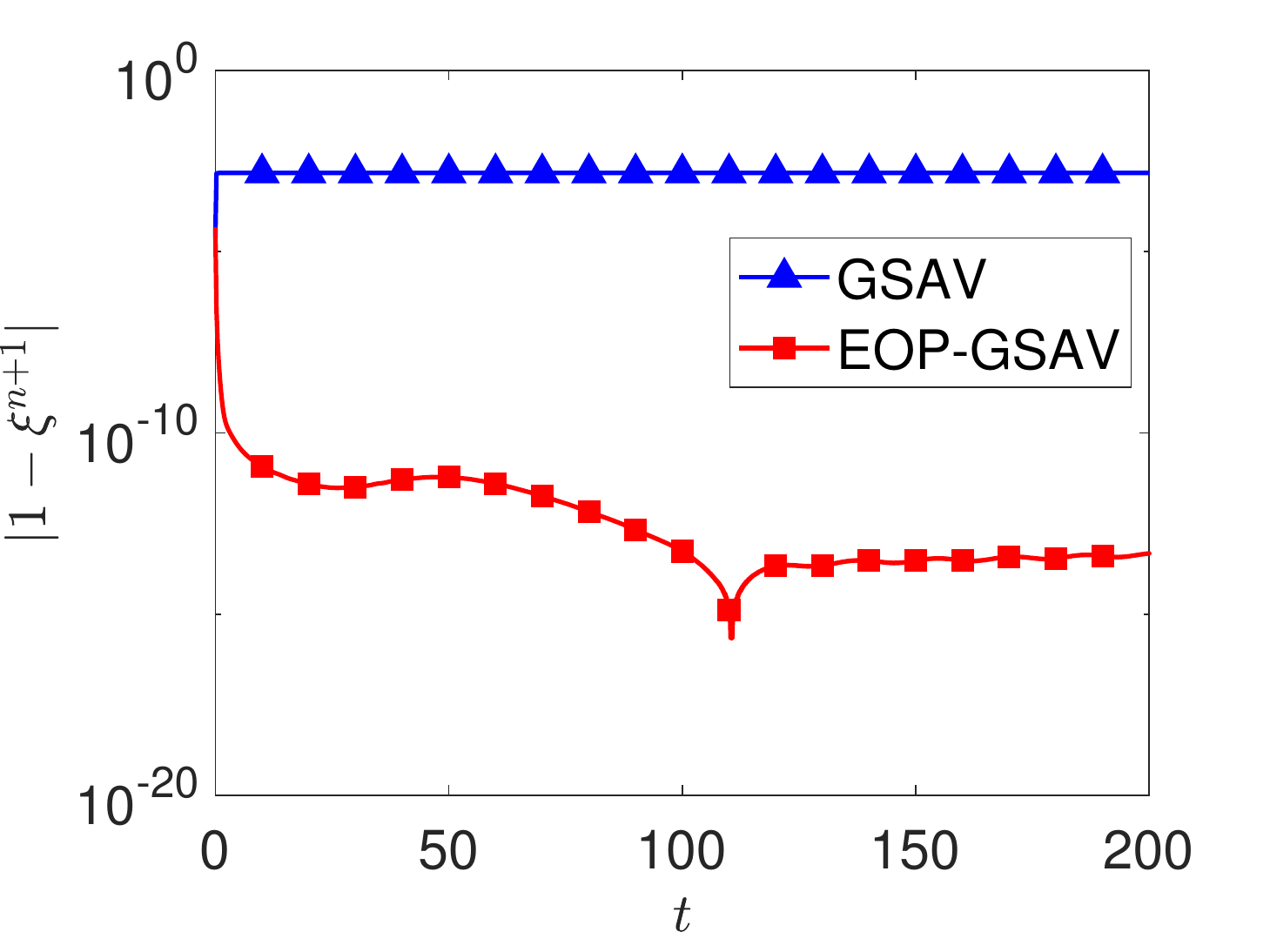}
}
	\caption{Example\,\ref{ex:AC}(Case B). A comparison of energy (first), energy error (second) and a comparison of error of $\xi^{n+1}$ (third) for GSAV/BDF$2$ and EOP-GSAV/BDF$2$ schemes.}
	\label{Fig:AC-star-shape-energy-xi}
\end{figure}

\begin{figure}[htbp]
	\centering
	\begin{minipage}{0.24\textwidth}
		\centering
		\includegraphics[width=4.7cm]{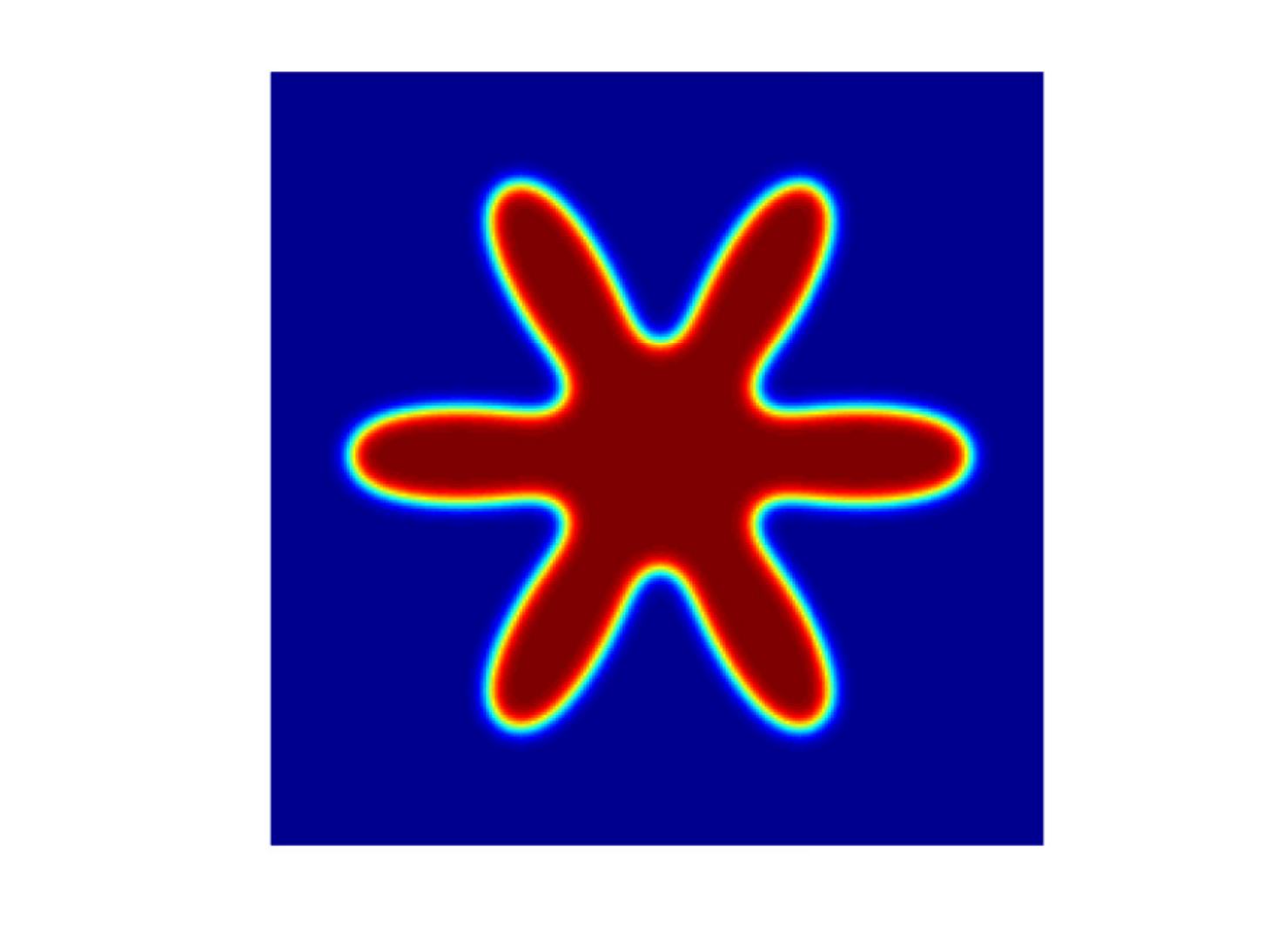}
	\end{minipage}	
	\begin{minipage}{0.24\textwidth}
		\centering
		\includegraphics[width=4.7cm]{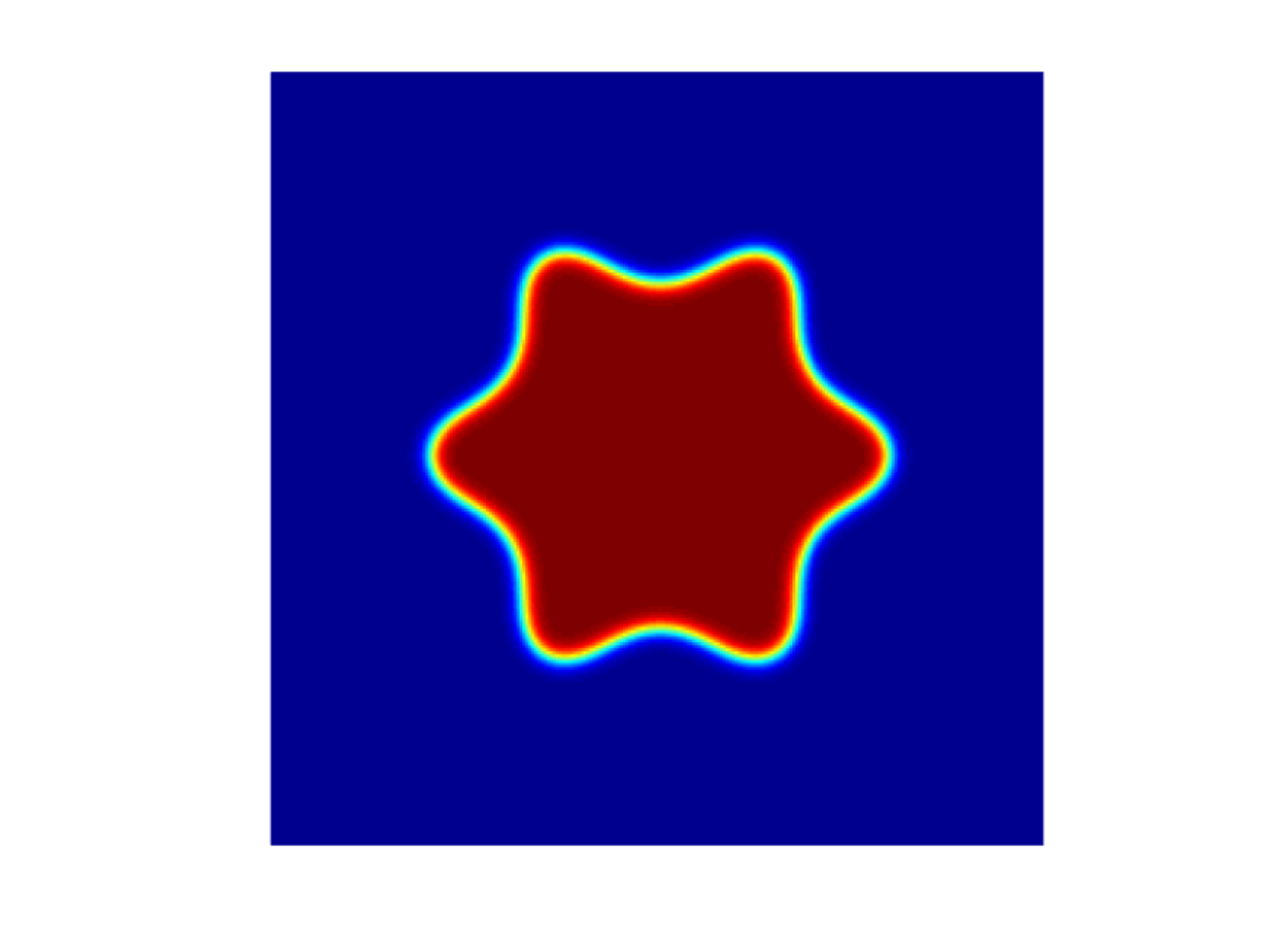}
	\end{minipage}
	\begin{minipage}{0.24\textwidth}
		\centering
		\includegraphics[width=4.7cm]{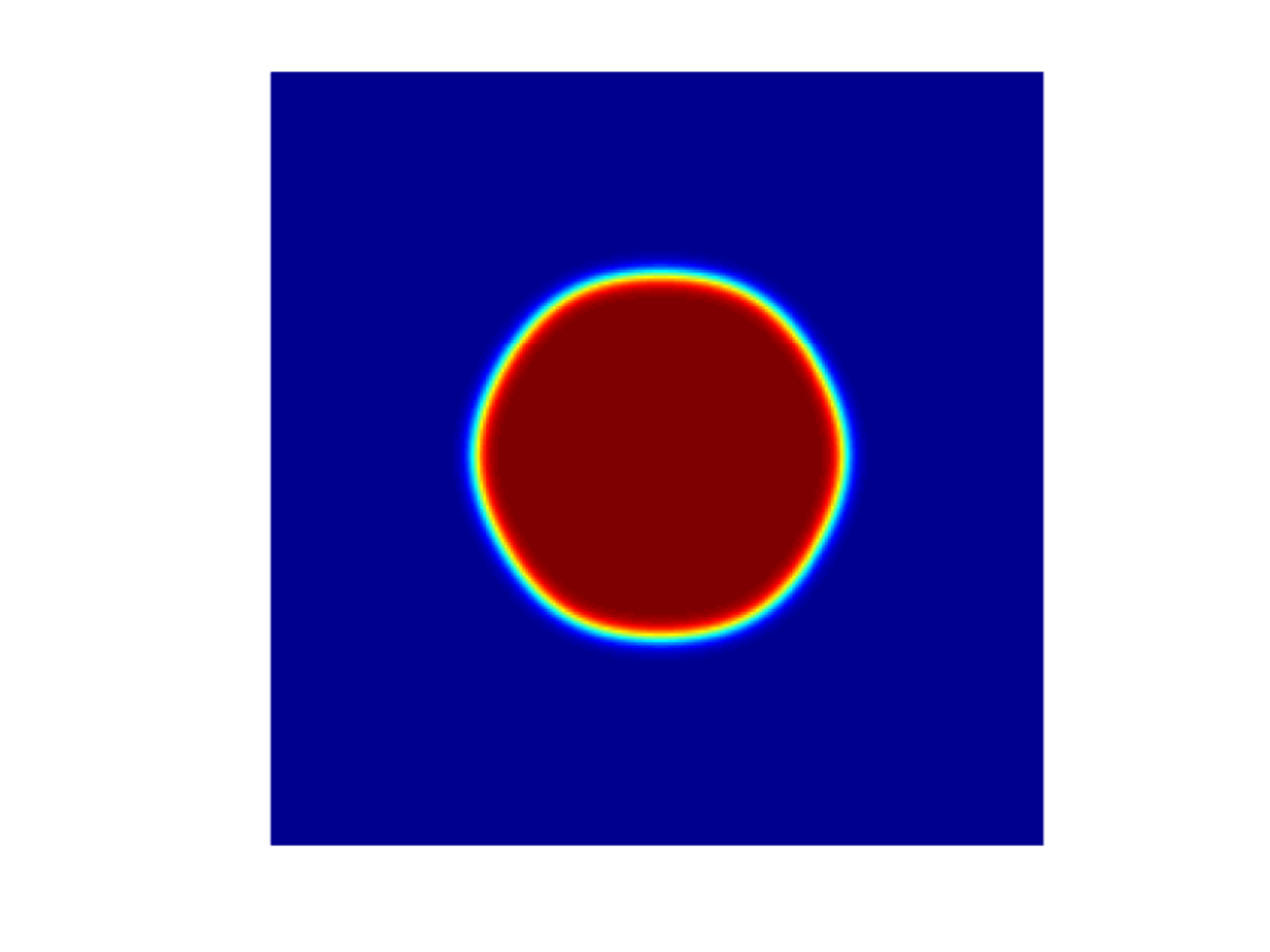}
	\end{minipage}
	\begin{minipage}{0.24\textwidth}
		\centering
		\includegraphics[width=4.7cm]{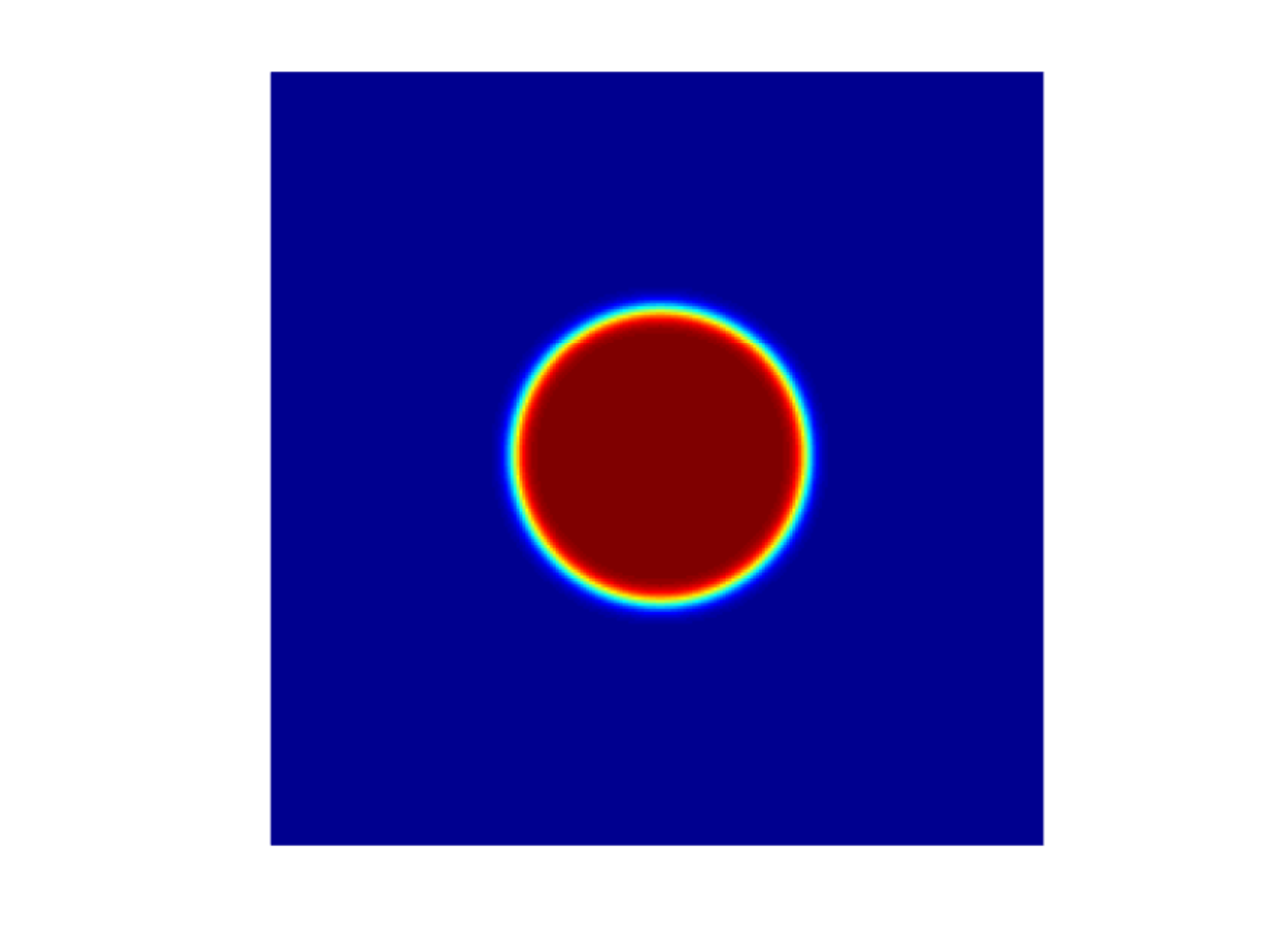}
	\end{minipage}
	\caption{Example\,\ref{ex:AC}(Case B). The $2$D dynamic evolution of Allen-Cahn equation obtained by EOP-GSAV/BDF$2$ scheme. Snapshots of the numerical solution $\phi$ at $T=10,$ $50,$ $100,$ $200,$ respectively.}
	\label{Fig:AC-star-shape-emGSAV}
\end{figure}

\end{example}

\begin{example}\label{ex:CH}
	\rm
	We consider Cahn-Hilliard equation
	\begin{equation}
		\frac{\partial \phi}{\partial t}=-M \Delta\left(\alpha_0 \Delta \phi+\frac{1}{\epsilon^2}\left(1-\phi^{2}\right) \phi\right).
	\end{equation}
	
	{\em Case A.}
	We consider the exact solution given by \eqref{eq:AC-CH-exact-solution-example} and set the parameters to $\alpha_0=0.04$, $M=0.005$, and $\epsilon=1$. The convergence rates of the CN and BDF$k$ (with $k=1,2,3,4$) schemes are presented in Fig.\,\ref{Fig:CH-order-test-CN-BDF}, respectively.
	The results are similar to those obtained for the Allen-Cahn equation.
	Additionally, Fig.\,\ref{Fig:CH-E1-E} shows the evolution of the difference between the original energy of the nonlinear part and $s^{n+1}$ obtained using the EOP-SAV/CN scheme with a time step of $\Delta t=0.01$, as well as the evolution of the difference between the original energy and the modified energy obtained using the EOP-GSAV/BDF$2$ scheme with a time step of $\Delta t=0.01$.
	In both cases, all the values are negative, indicating that the modified energy is equal to the original energy.
	
	\begin{figure}[htbp]
		\centering
\subfigure[]{
\includegraphics[width=7cm,height=6cm]{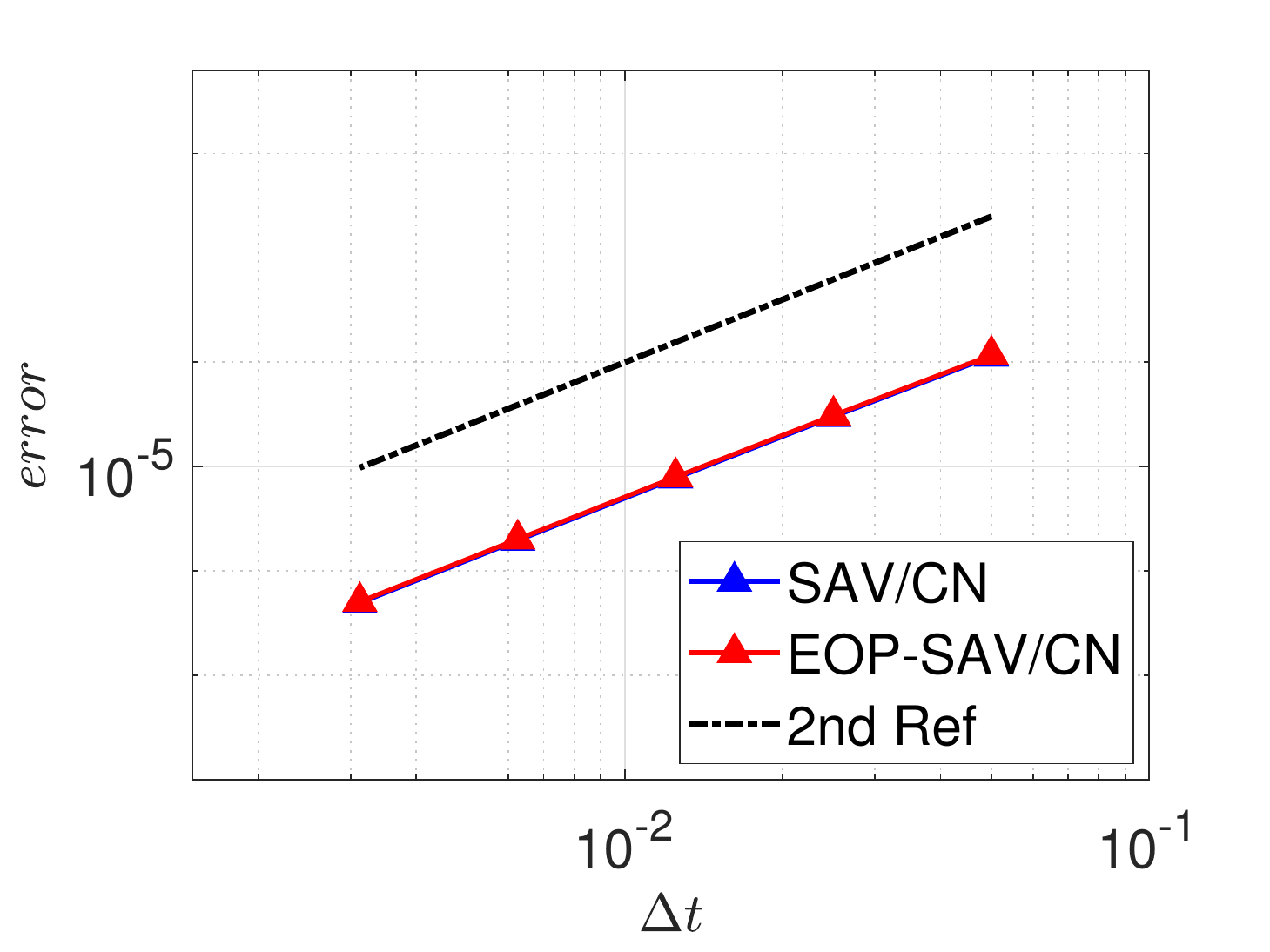}
}
\subfigure[]
{
\includegraphics[width=7cm,height=6cm]{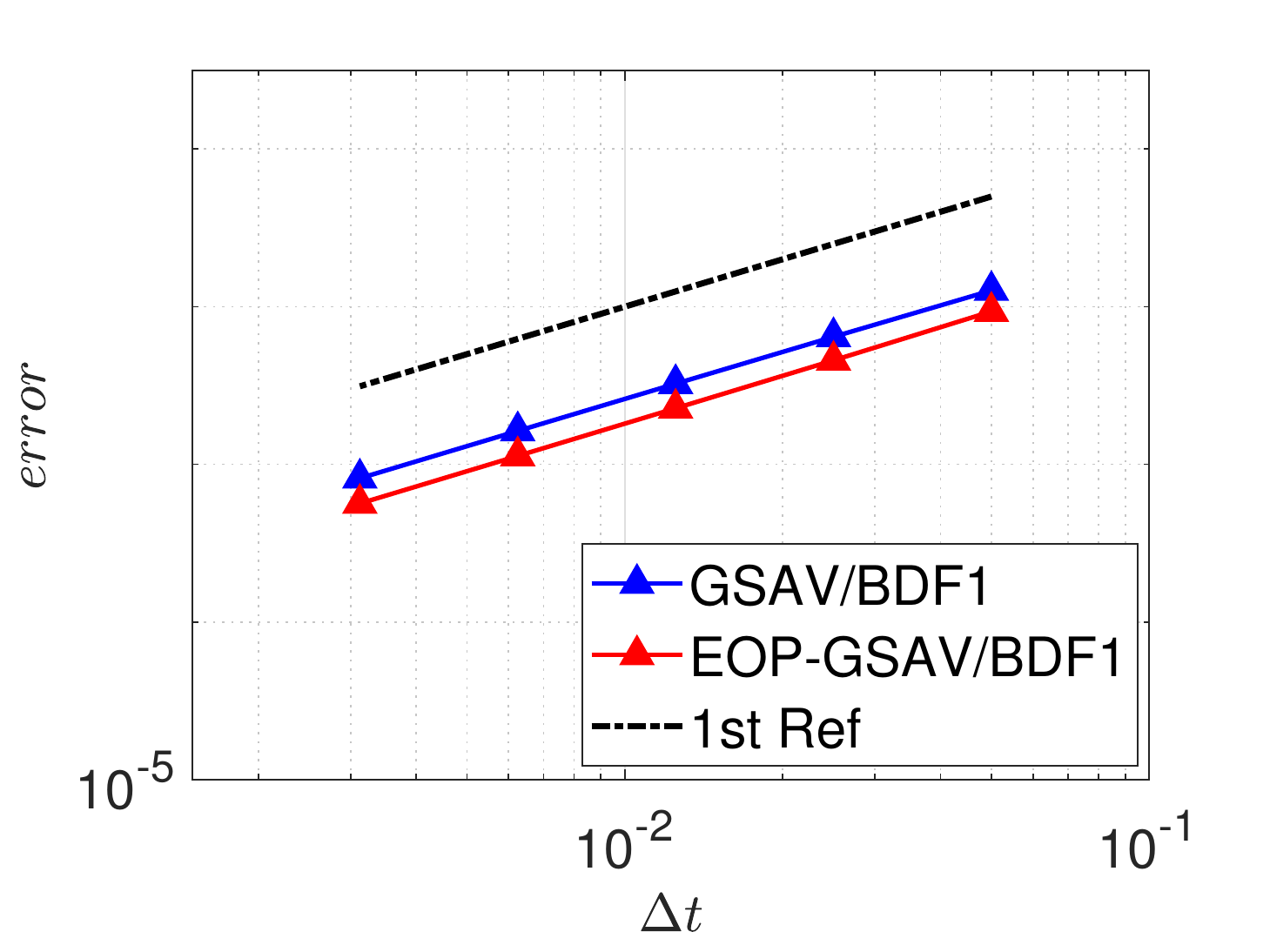}
}
\quad
\subfigure[]
{
\includegraphics[width=7cm,height=6cm]{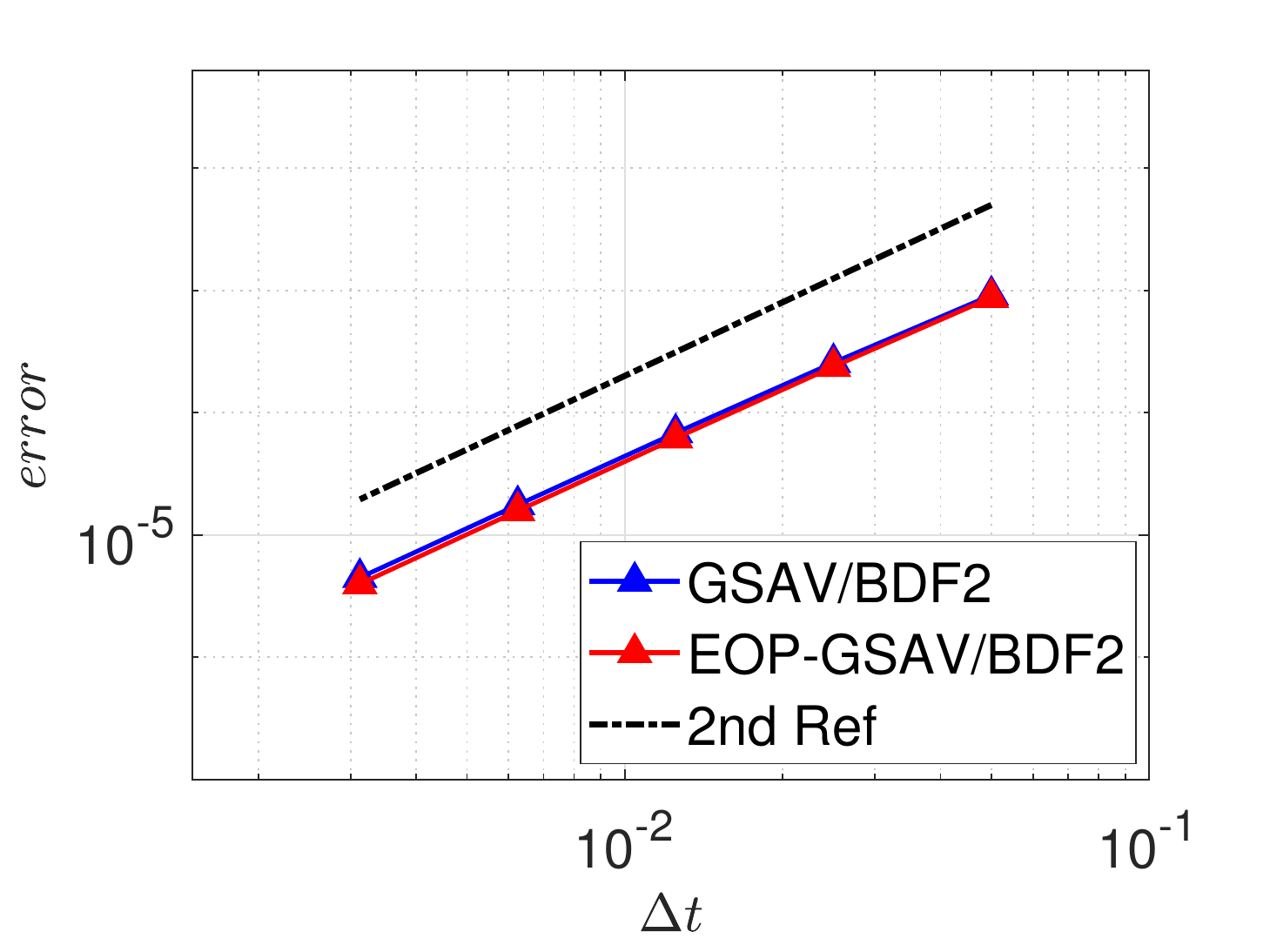}
}
\subfigure[]{
\includegraphics[width=7cm,height=6cm]{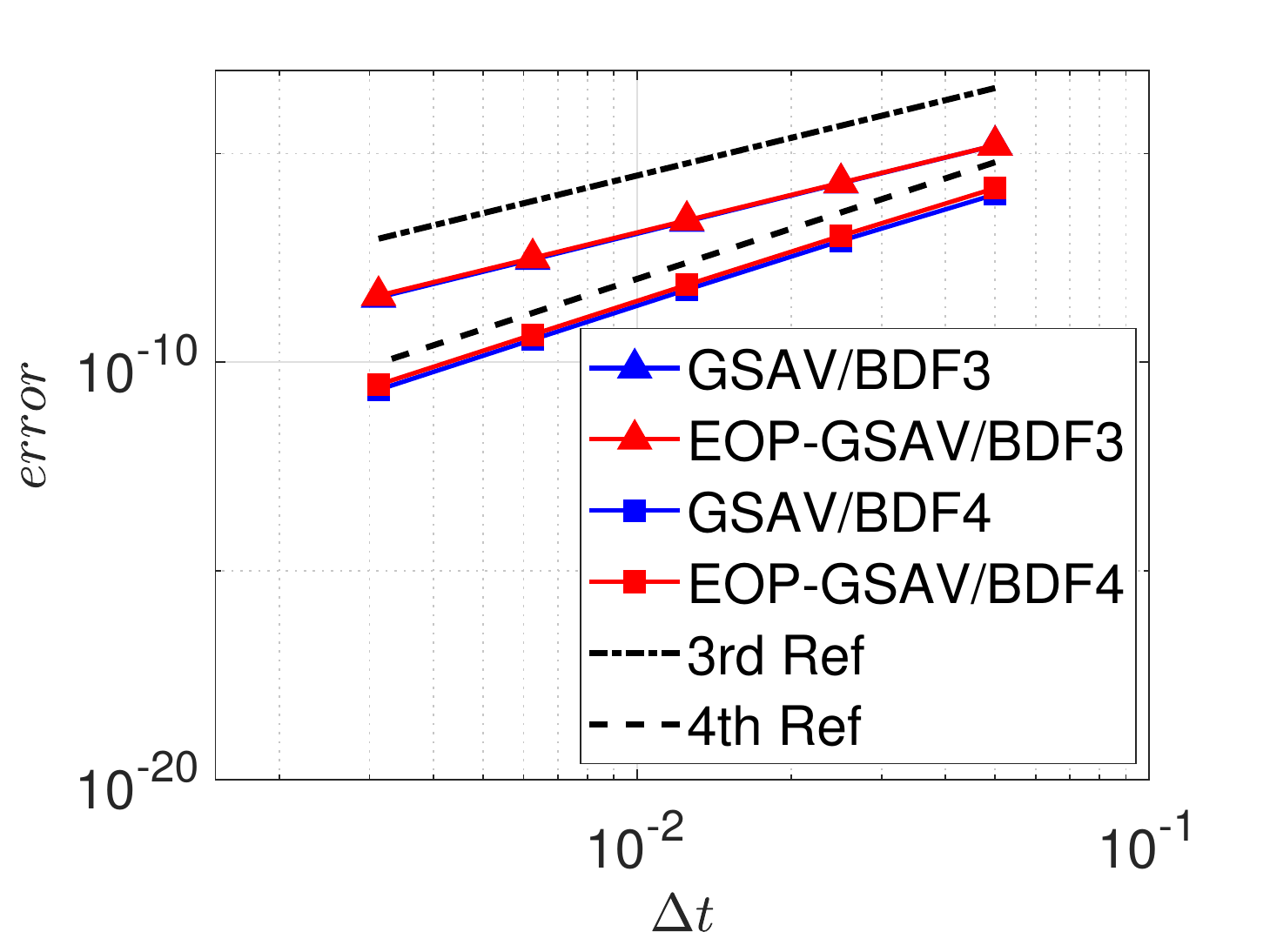}
}
		\caption{Example\,\ref{ex:CH}(Case A). Convergence rates for Cahn-Hilliard equation using various schemes. (a): CN; (b) BDF$1$; (c): BDF$2$; (d): BDF$k$, $(k=3,4)$..}
		\label{Fig:CH-order-test-CN-BDF}
	\end{figure}

\begin{figure}[htbp]
	\centering
	\begin{minipage}{0.4\textwidth}
		\centering
		\includegraphics[width=5.3cm]{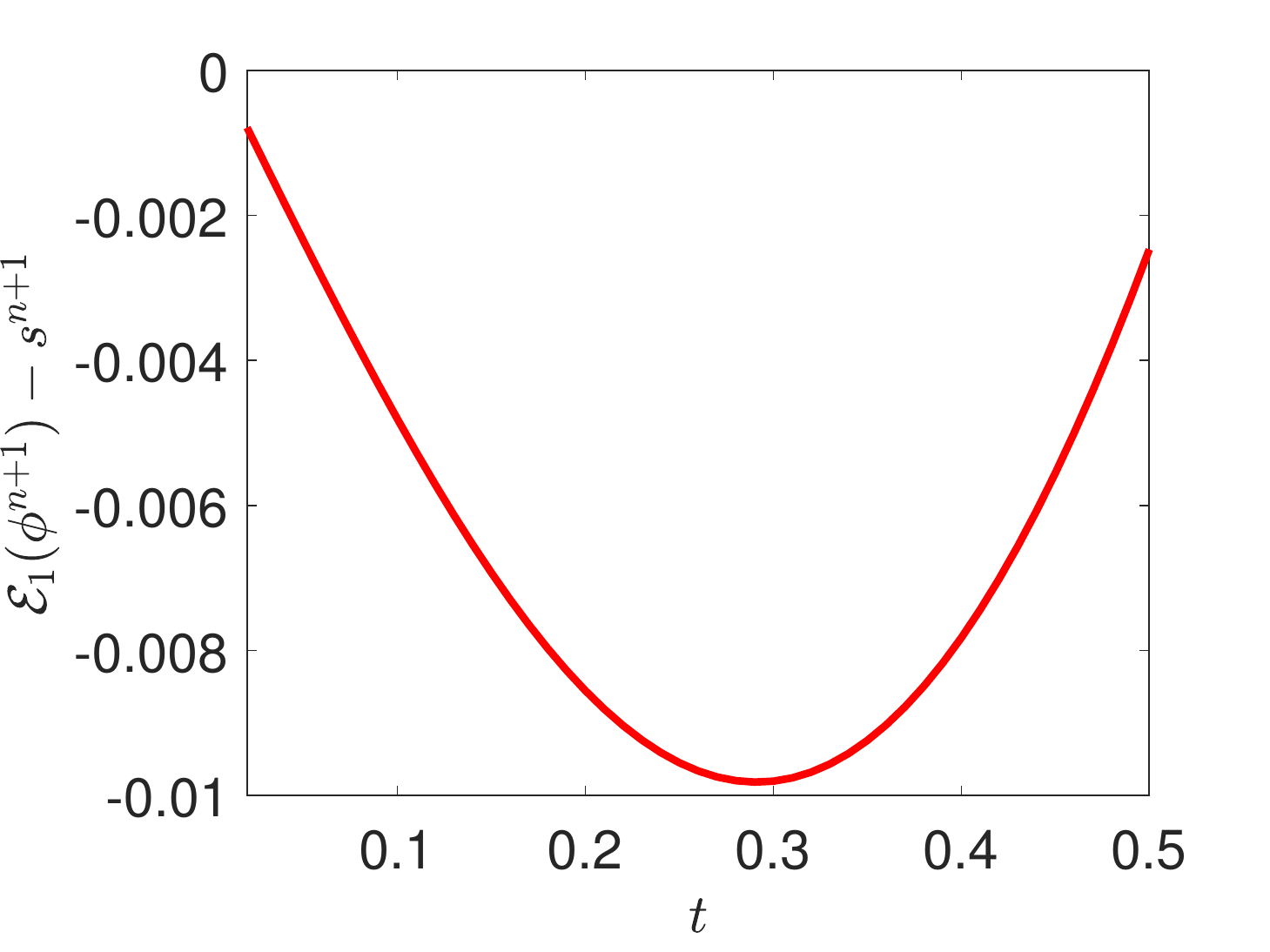}
	\end{minipage}
	\begin{minipage}{0.4\textwidth}
		\centering
		\includegraphics[width=5.3cm]{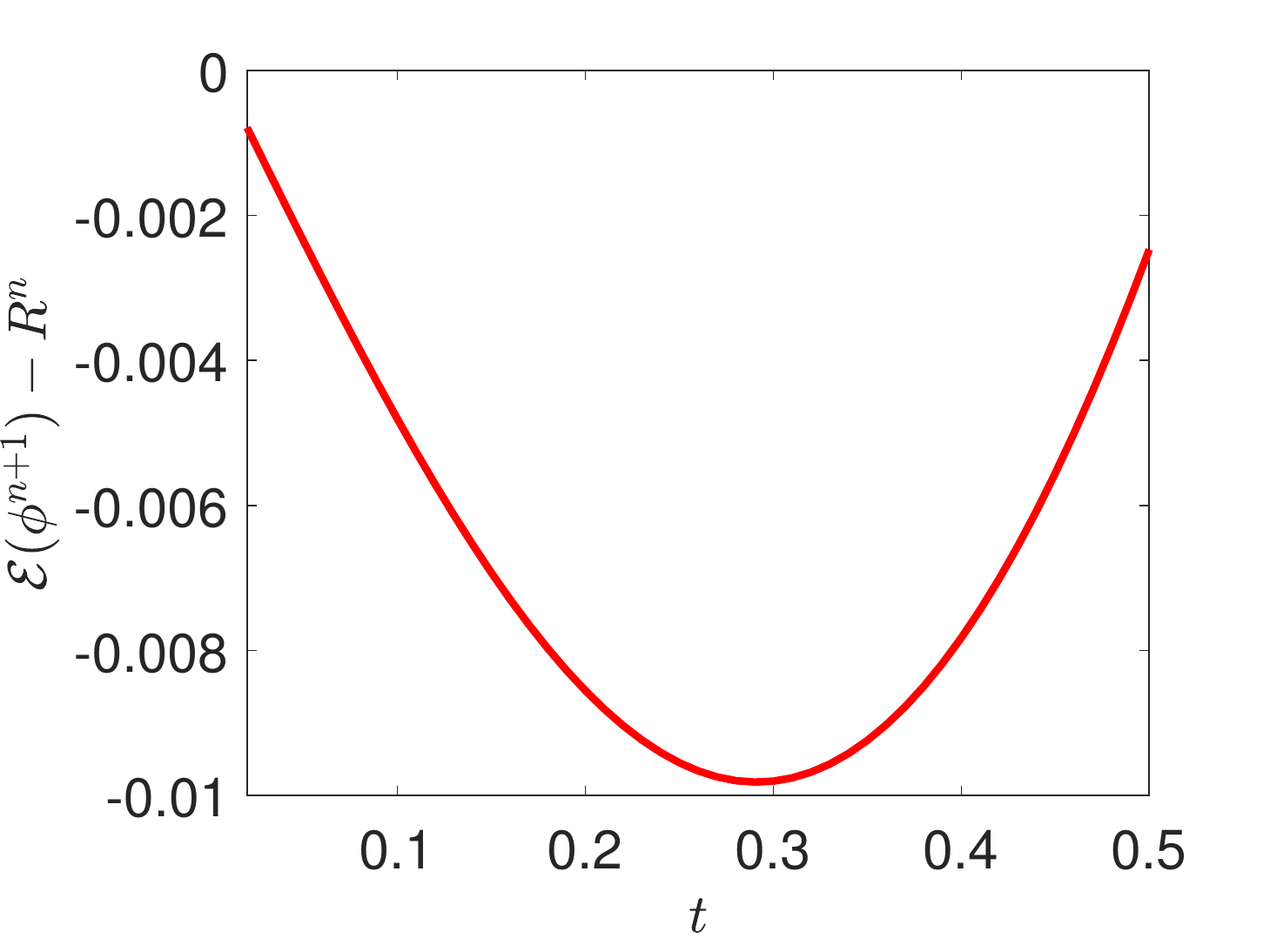}
	\end{minipage}
	\caption{Example\,\ref{ex:CH}(Case A). First: evolution of the difference between the original energy of nonlinear part and $s^{n+1}$ using EOP-SAV/CN scheme with $\Delta t=0.01$; Second: evolution of the difference between the original energy and the modified energy using EOP-GSAV/BDF$2$ scheme with $\Delta t=0.01$.}
	\label{Fig:CH-E1-E}
\end{figure}

{\em Case B.}  As the initial condition, we consider a rectangular arrangement of $9 \times 9$ circles
\begin{equation}
	\phi_{0}(\boldsymbol{x}, t)=80-\sum_{m=1}^{9} \sum_{n=1}^{9} \tanh\left(\frac{\sqrt{\left(x-x_{m}\right)^{2}+\left(y-y_{n}\right)^{2}}-r_{0}}{\sqrt{2} \epsilon}\right),
\end{equation}
where $r_0 = 0.085, x_{m} = 0.2\times m, y_{n} = 0.2\times n$ for $m, n = 1, 2, \cdots, 9$.
For our simulations, we use a computational domain of $[0,2]^2$. The parameters $M$, $\alpha_0$, and $\epsilon$ are set to $1e-6$, $1$, and $0.01$, respectively. We adopt a spatial discretization scheme using $512^2$ Fourier modes. The evolution of a rectangular array of circles governed by the Cahn-Hilliard equation is depicted in Fig.\,\ref{Fig:CH-circles-emGSAV}, obtained using the EOP-SAV/BDF$2$ scheme with a time step of $\Delta t = 1e-3$.

\begin{figure}[htbp]
	\centering
	\begin{minipage}{0.3\textwidth}
		\centering
		\includegraphics[width=5.3cm]{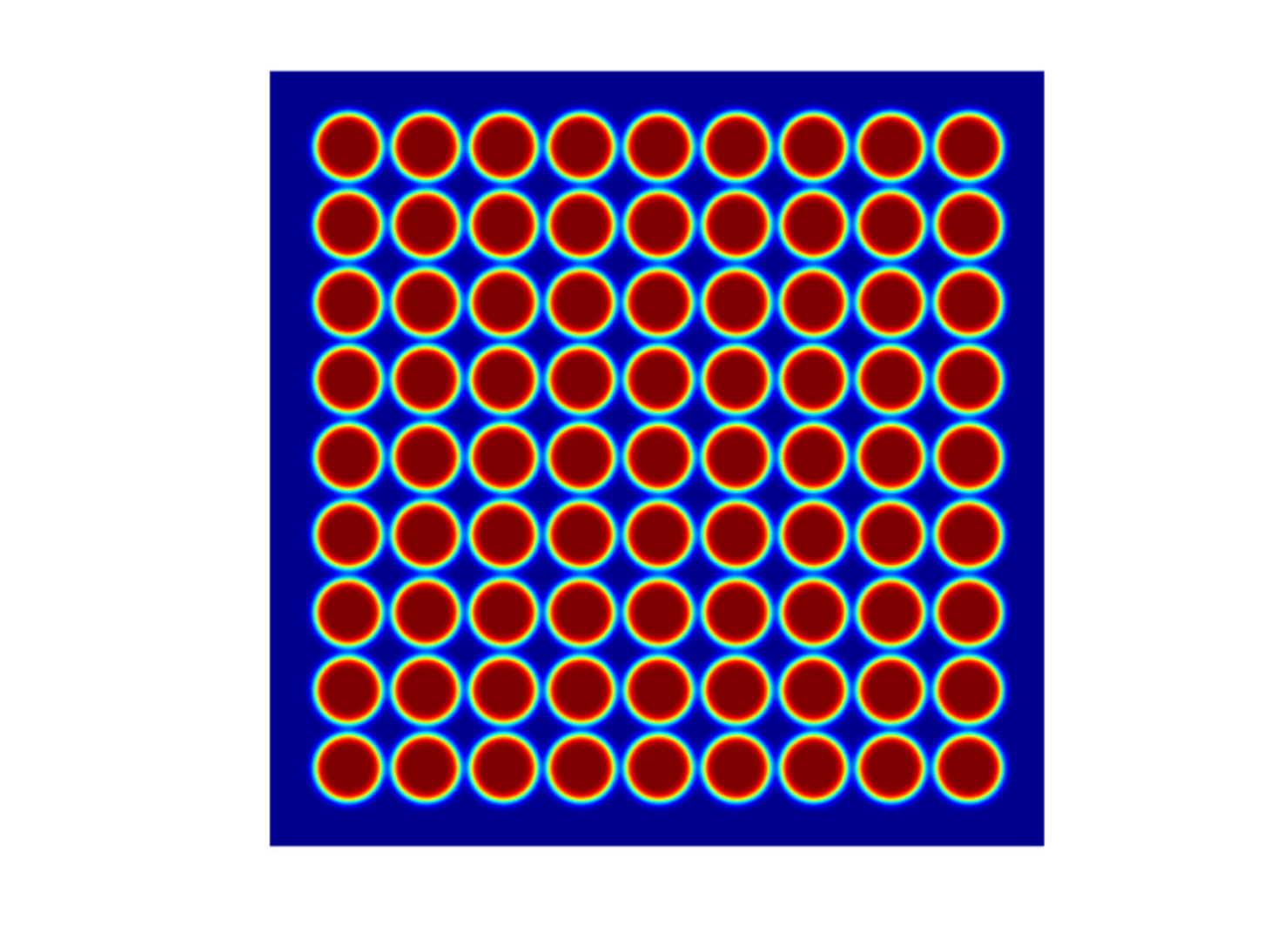}
	\end{minipage}
	\begin{minipage}{0.3\textwidth}
		\centering
		\includegraphics[width=5.3cm]{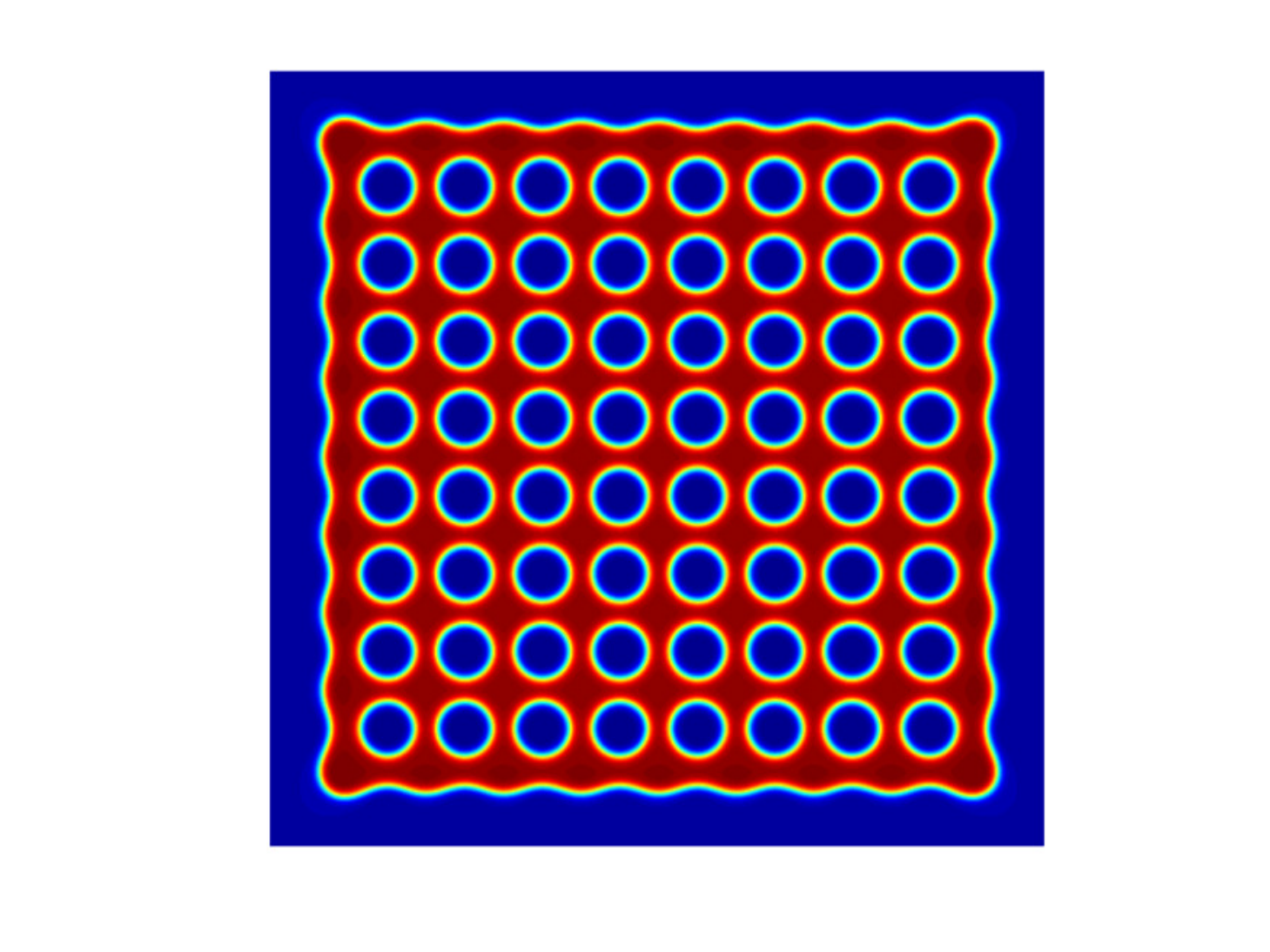}
	\end{minipage}	
	\begin{minipage}{0.3\textwidth}
		\centering
		\includegraphics[width=5.3cm]{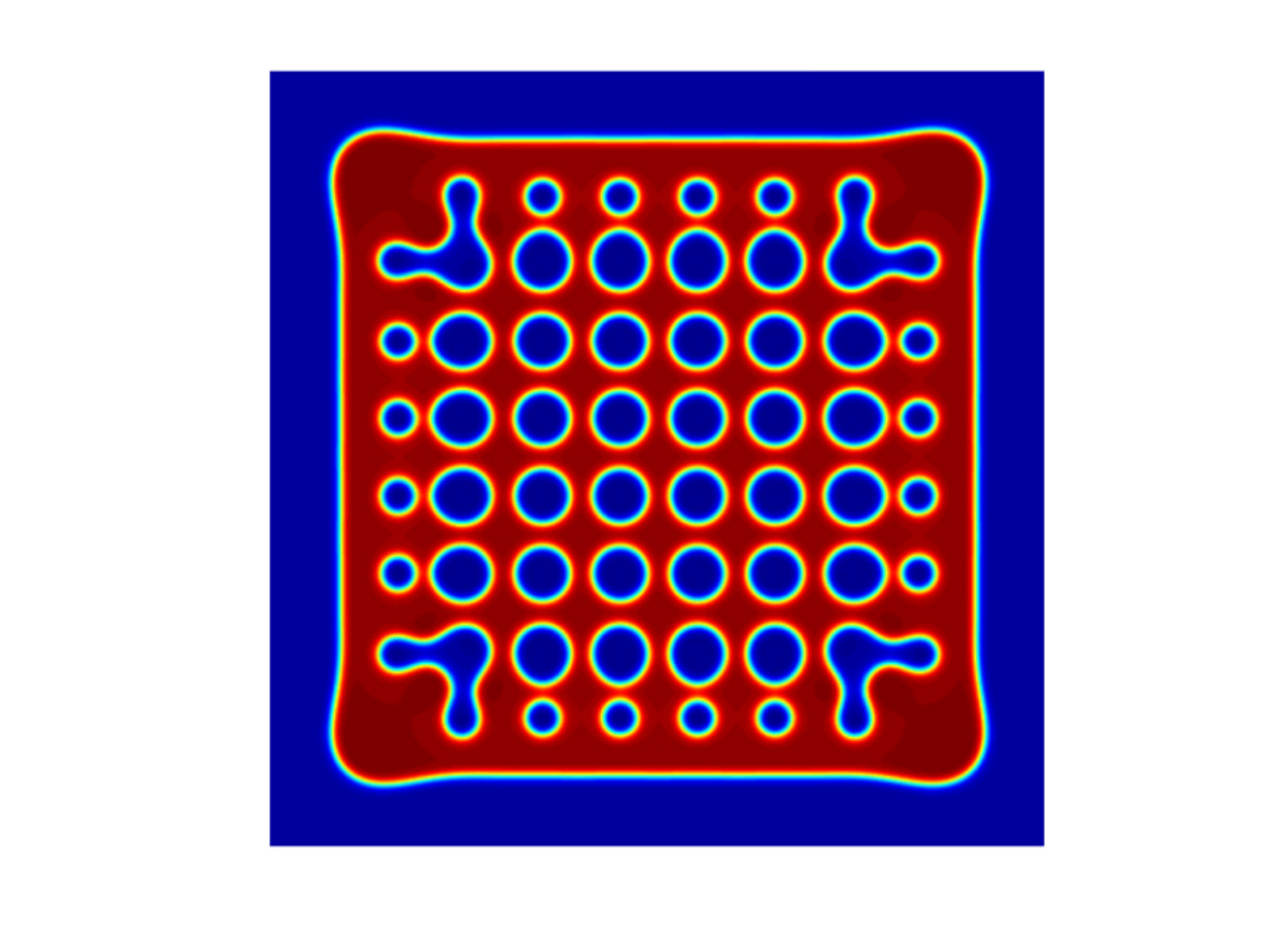}
	\end{minipage}	
	\begin{minipage}{0.3\textwidth}
		\centering
		\includegraphics[width=5.3cm]{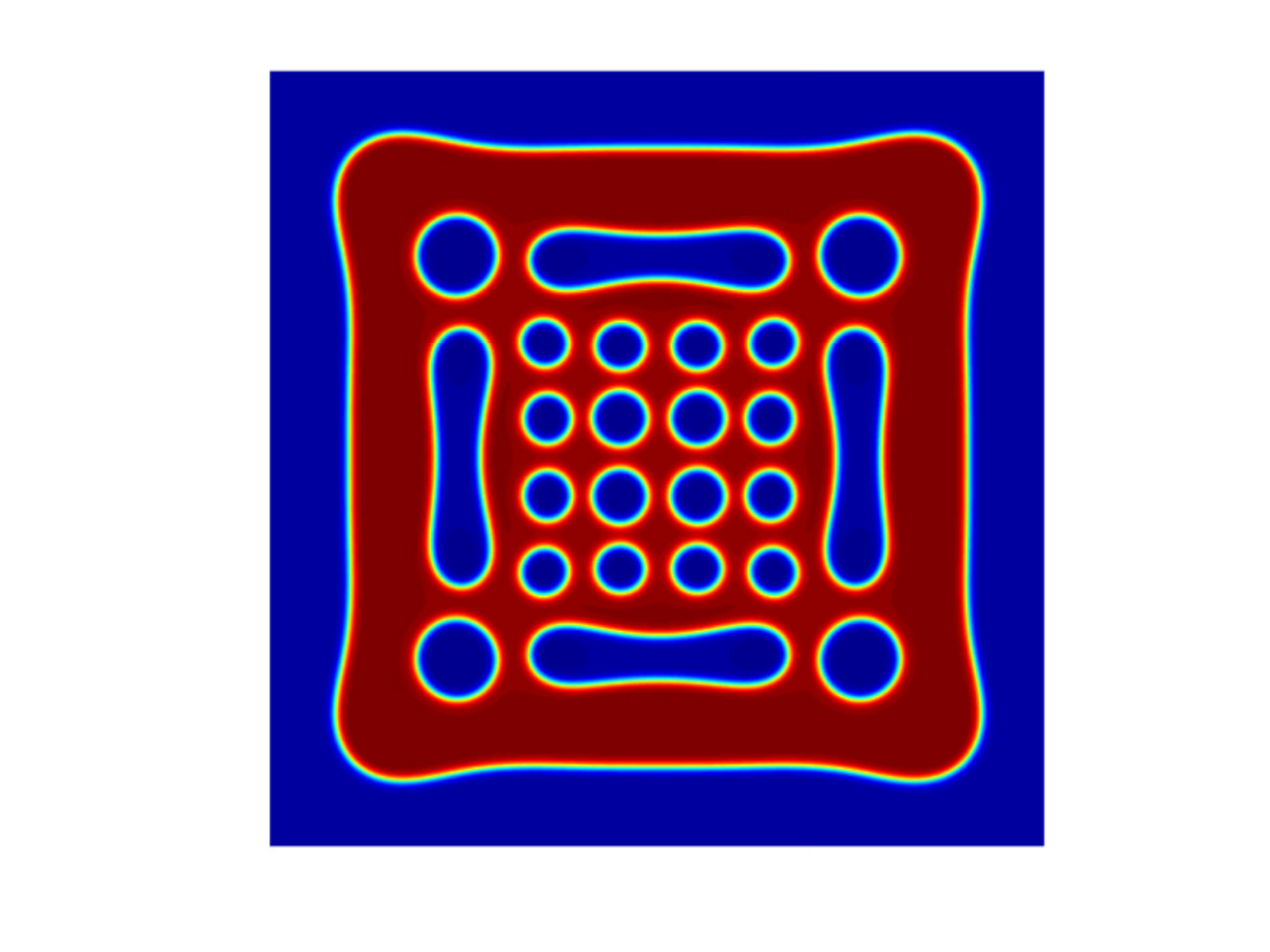}\end{minipage}	
	\begin{minipage}{0.3\textwidth}
		\centering
		\includegraphics[width=5.3cm]{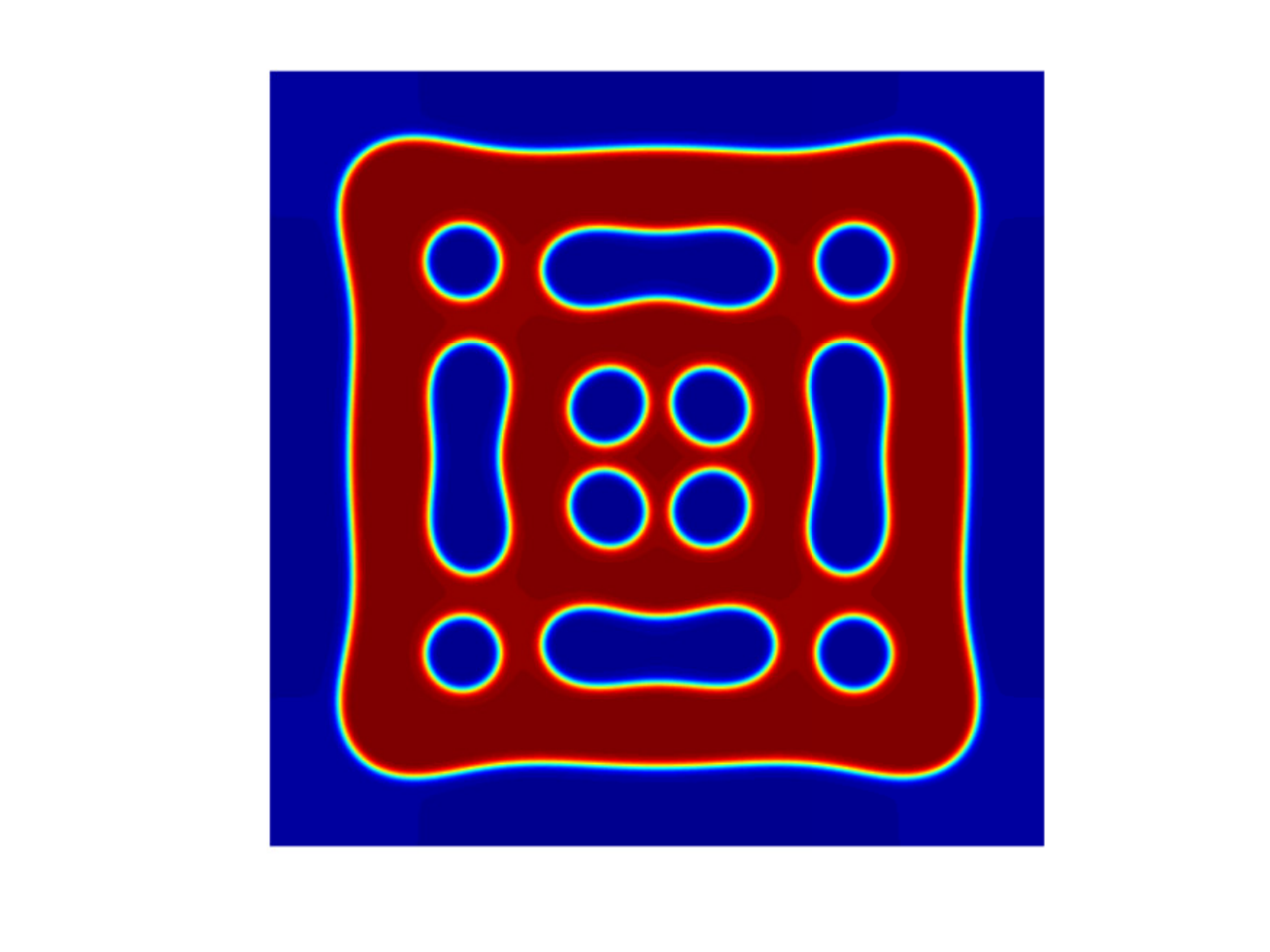}
	\end{minipage}
	\begin{minipage}{0.3\textwidth}
		\centering
		\includegraphics[width=5.3cm]{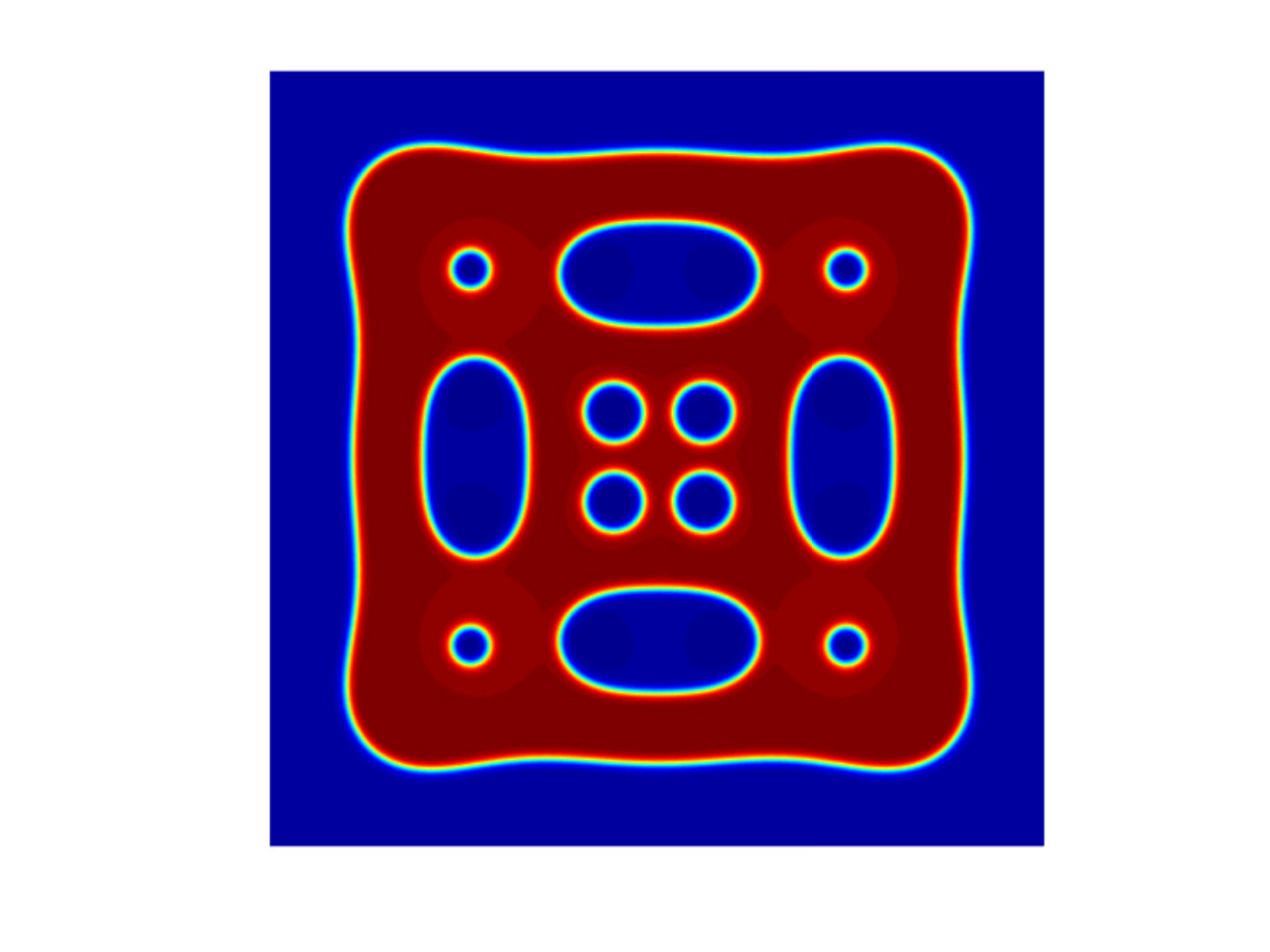}
	\end{minipage}
	\caption{Example\,\ref{ex:CH}(Case B). The dynamic evolution of an array of circles governed by Cahn-Hilliard equation obtained by EOP-GESAV/BDF$2$ scheme.}
	\label{Fig:CH-circles-emGSAV}
\end{figure}

\end{example}

\begin{example}\label{ex:PFC}
	\rm
To demonstrate the versatility of the EOP-GSAV approach in simulating complex nonlinear phenomena, we consider the following phase-field crystal (PFC) model as an illustrative example
\begin{equation}
	\left\{
	\begin{array}{l}\frac{\partial \phi}{\partial t}=M \Delta \mu,  \quad \boldsymbol{x} \in \Omega, t>0, \\
		\mu=(\Delta+\beta)^{2} \phi+\phi^{3}-\epsilon \phi, \quad \boldsymbol{x} \in \Omega, t>0, \\
		\phi(\boldsymbol{x}, 0)=\phi_{0}(\boldsymbol{x}),
	\end{array}\right.
\end{equation}
which is a gradient flow associated with total free energy
\begin{equation}
	E(\phi)=\int_{\Omega}\left(\frac{1}{2} \phi(\Delta+\beta)^{2} \phi+\frac{1}{4} \phi^{4}-\frac{\epsilon}{2} \phi^{2}\right) \mathrm{d} \boldsymbol{x},
\end{equation}
where $M>0$ is the mobility coefficient.
In the following simulations, we choose $M=1, \beta=1$.

{\em Case A.}  We consider the problem of crystal growth in a two-dimensional super-cooled liquid. The initial condition is set to
\begin{equation}
	\phi\left(x_{l}, y_{l}, 0\right)=\bar{\phi}+C_{1}\left(\cos \left(\frac{C_{2}}{\sqrt{3}} y_{l}\right) \cos \left(C_{2} x_{l}\right)-0.5 \cos \left(\frac{2 C_{2}}{\sqrt{3}} y_{l}\right)\right), \quad l=1,2,3,
\end{equation}
where the local system of Cartesian coordinates is defined by $x_l$ and $y_l$, oriented with the crystallite lattice. The constant parameters $\bar{\phi} = 0.285$, $C_{1} = 0.446$, and $C_{2} = 0.66$ are also specified.
To simulate the growth of crystals, we define three crystallites in three small square patches, each with a side length of $40$, located at the coordinates $(350, 400)$, $(200, 200)$, and $(600, 300)$, respectively. These crystallites are initialized perfectly.
To generate crystallites with various orientations, we utilize the following affine transformation to induce rotation
\begin{equation}
	x_{l}(x, y)=x \sin (\theta)+y \cos (\theta), \quad y_{l}(x, y)=-x \cos (\theta)+y \sin (\theta),
\end{equation}
where angles are chosen as $\theta=-\frac{\pi}{4}, 0, \frac{\pi}{4}$ respectively. We discretize the space using $1024^2$ Fourier modes and adopt a relatively small time step of $\Delta t =0.02$ to ensure higher accuracy.
We set the remaining parameters as $\epsilon=0.25$ and $T=2000$.
Fig.\,\ref{Fig:PFC-emGSAV-2D} depicts the crystal growth in a super-cooled liquid driven by the PFC equation using the EOP-GSAV/BDF$2$ scheme.
The simulation results illustrate that the different orientations of the crystallites lead to defects and dislocations, consistent with findings in \cite{li2020stability, yang2017linearly}.
Notably, the modified energy in this example is equal to the original energy.

\begin{figure}[htbp]
	\centering
	\begin{minipage}{0.24\textwidth}
		\centering
		\includegraphics[width=4.7cm]{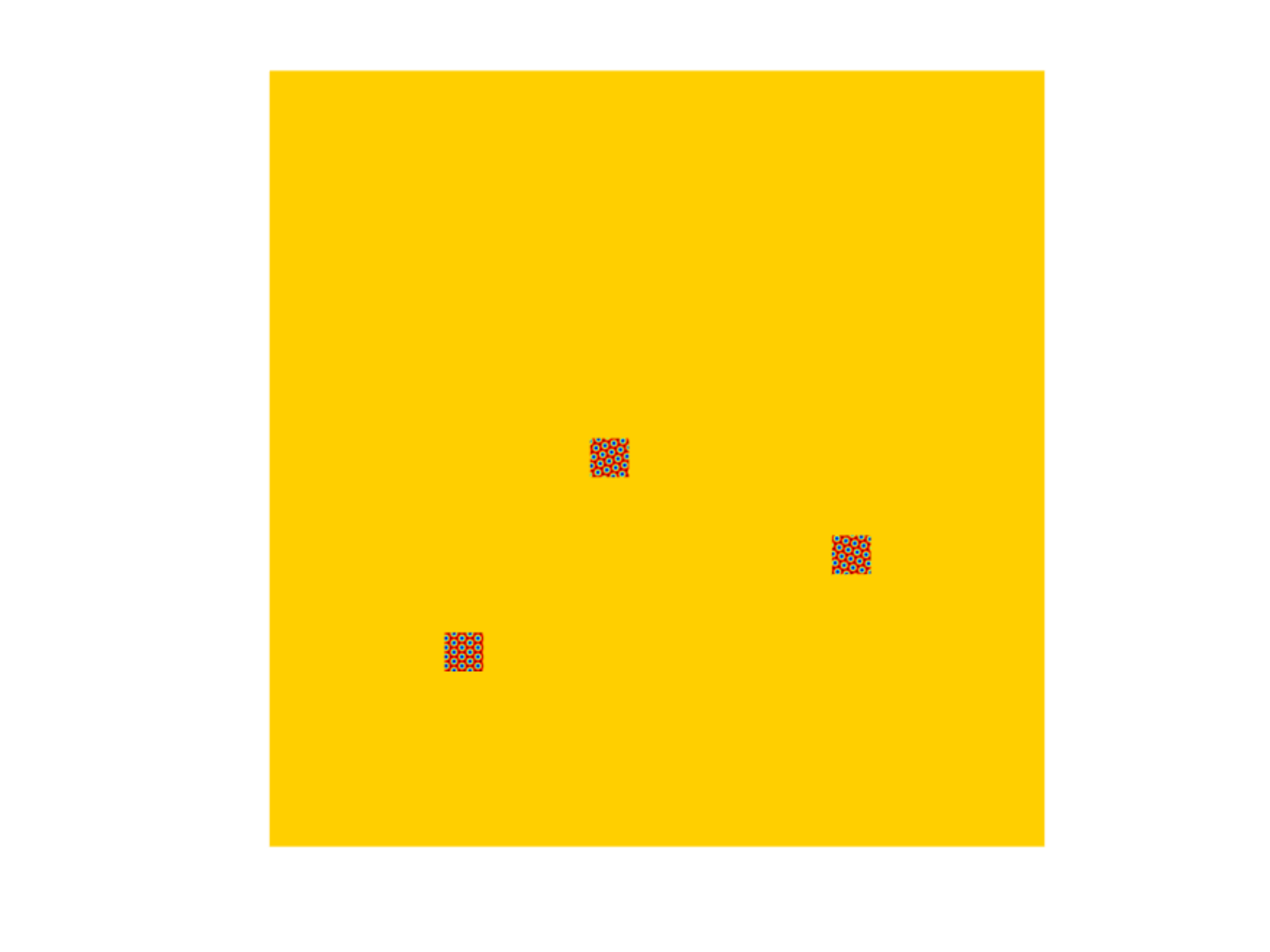}
	\end{minipage}	
	\begin{minipage}{0.24\textwidth}
		\centering
		\includegraphics[width=4.7cm]{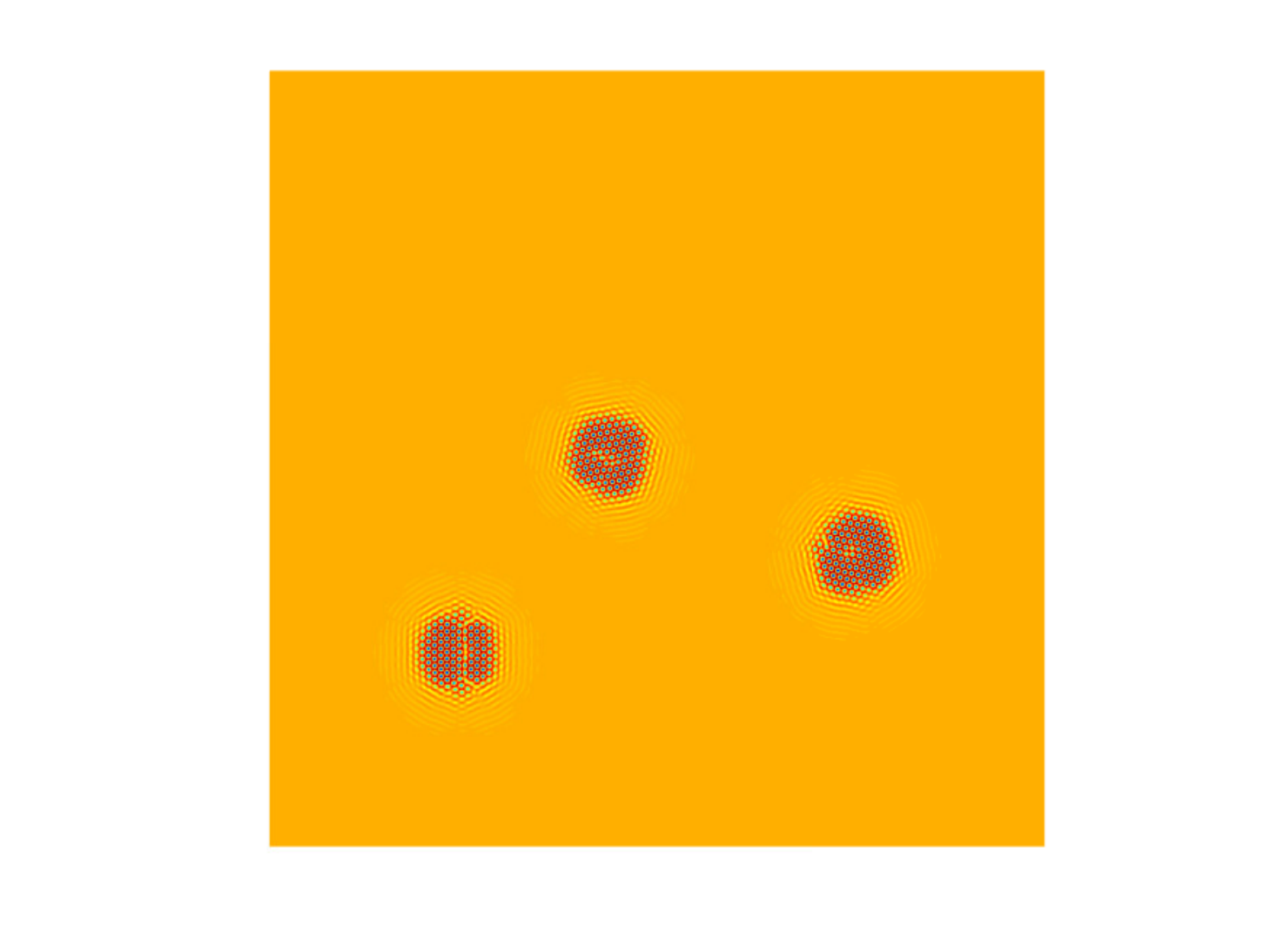}
	\end{minipage}
	\begin{minipage}{0.24\textwidth}
		\centering
		\includegraphics[width=4.7cm]{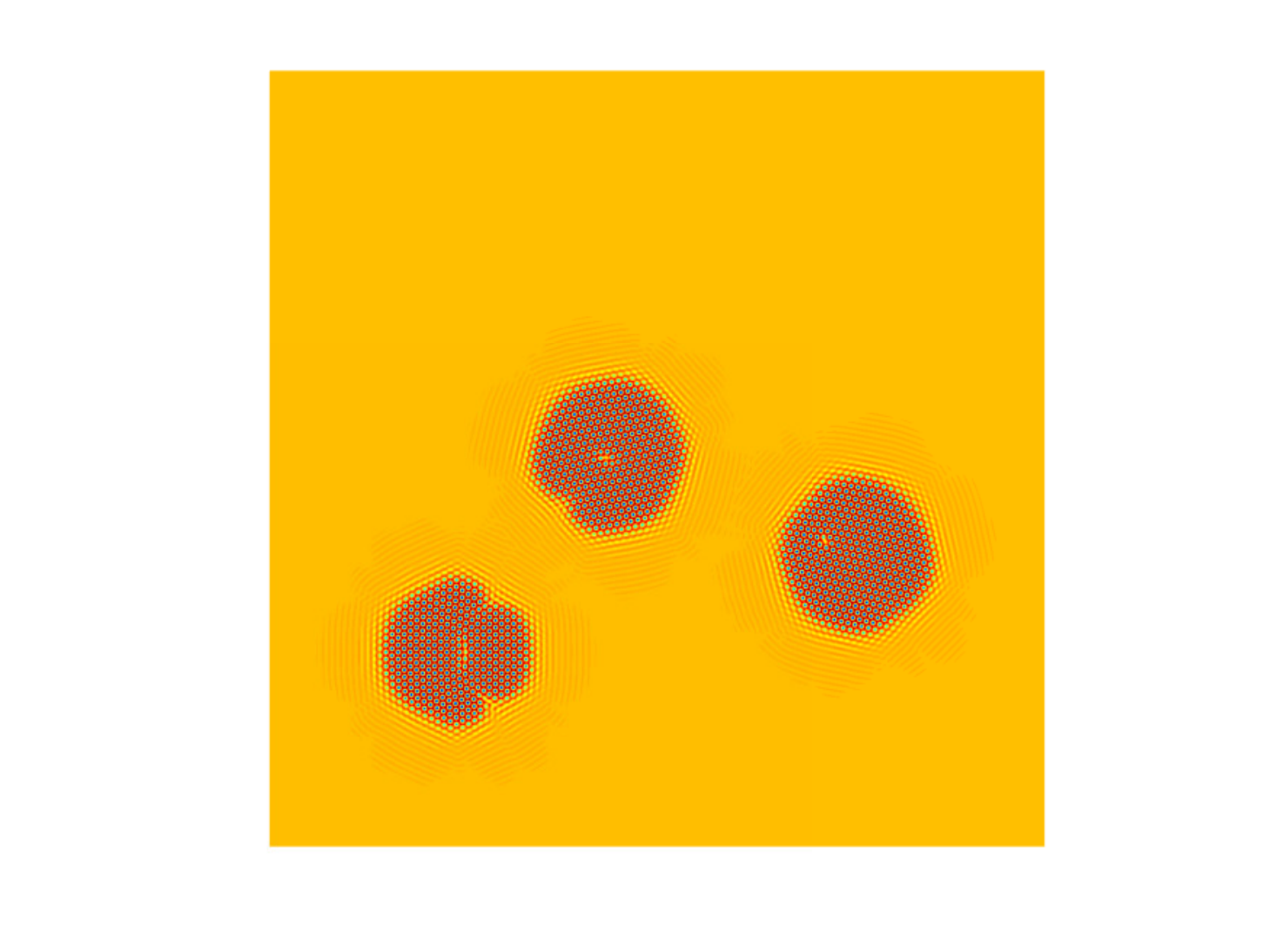}
	\end{minipage}
	\begin{minipage}{0.24\textwidth}
		\centering
		\includegraphics[width=4.7cm]{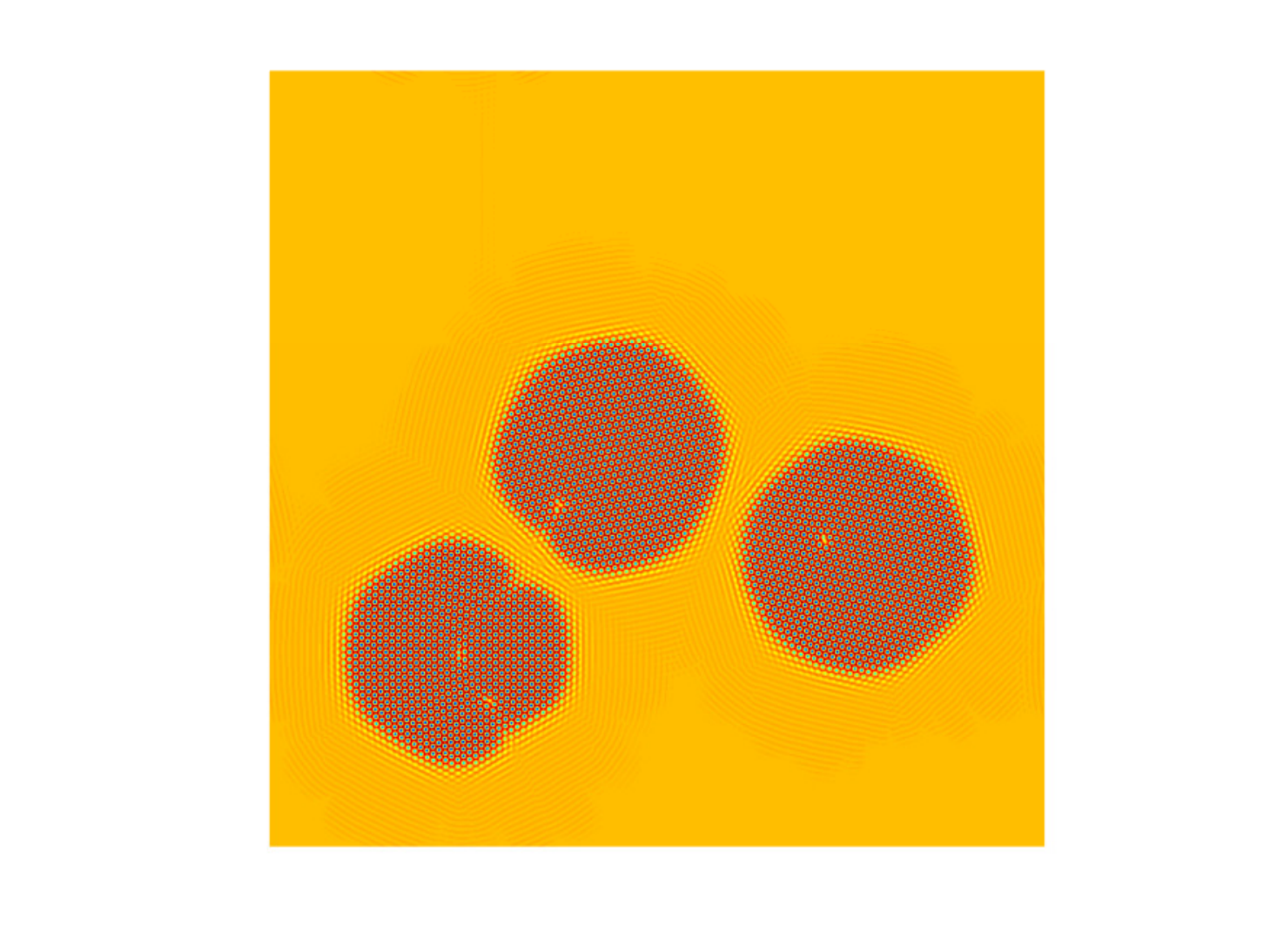}
	\end{minipage}
	\begin{minipage}{0.24\textwidth}
		\centering
		\includegraphics[width=4.7cm]{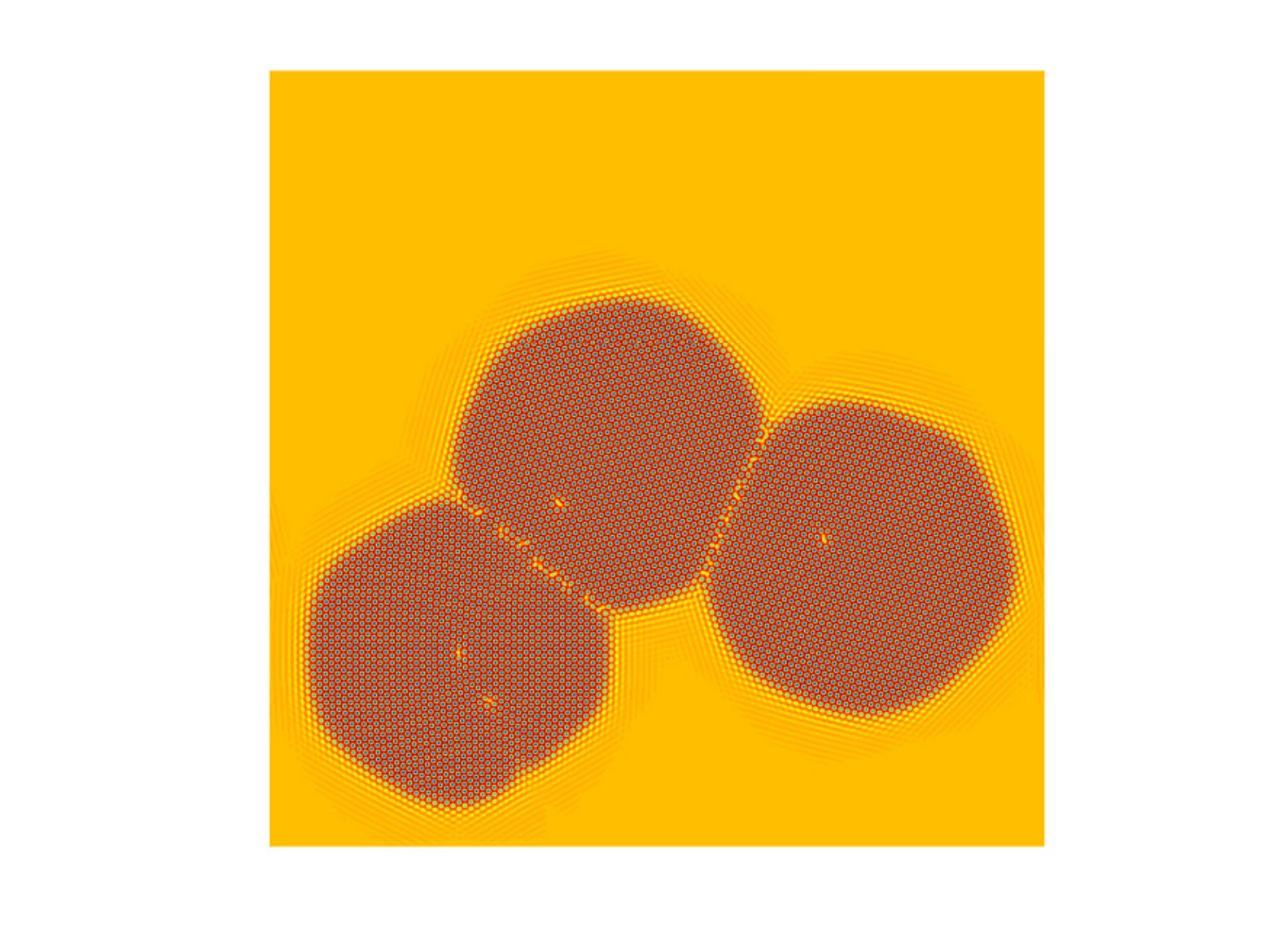}
	\end{minipage}
	\begin{minipage}{0.24\textwidth}
		\centering
		\includegraphics[width=4.7cm]{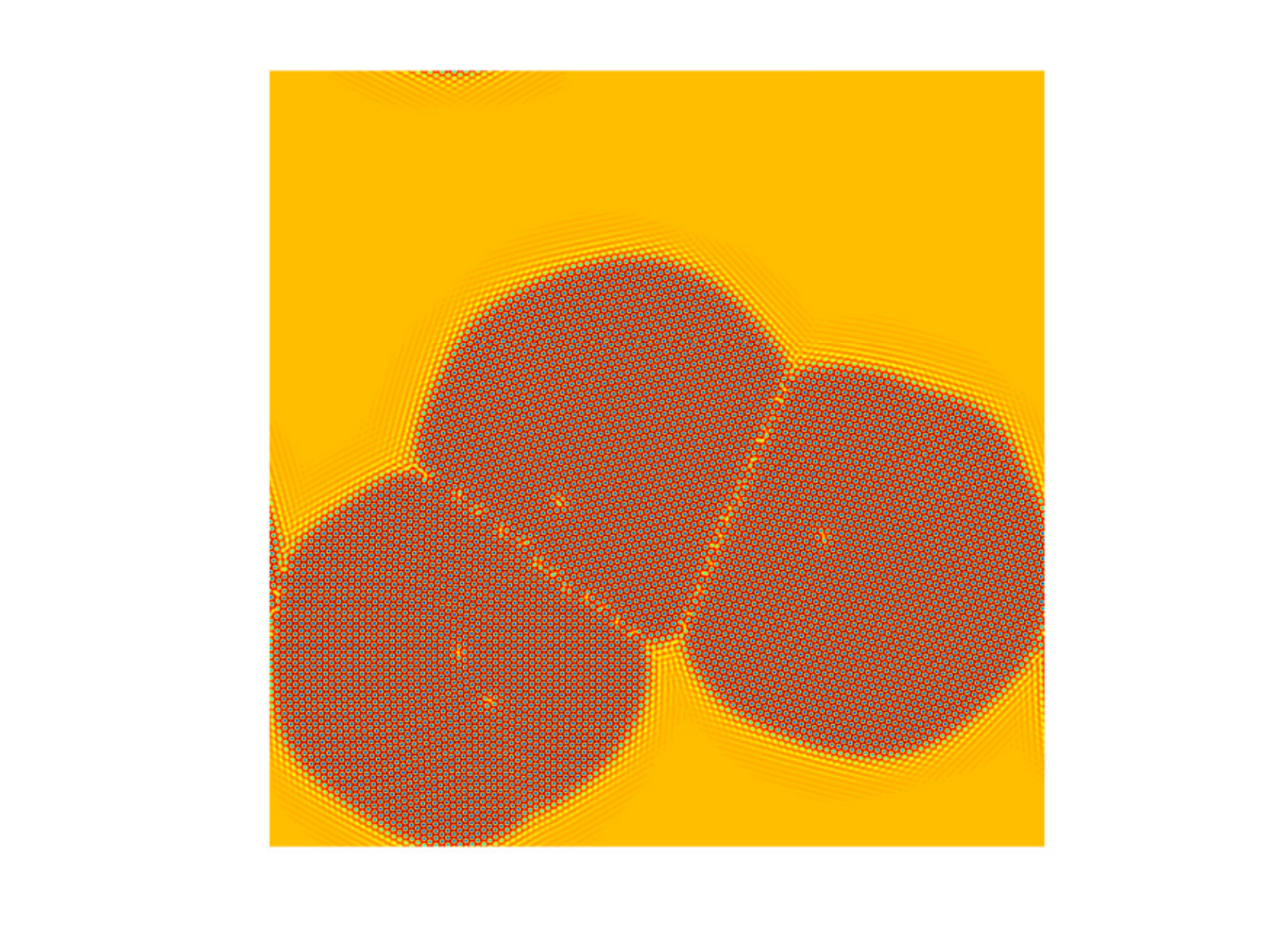}
	\end{minipage}
	\begin{minipage}{0.24\textwidth}
		\centering
		\includegraphics[width=4.7cm]{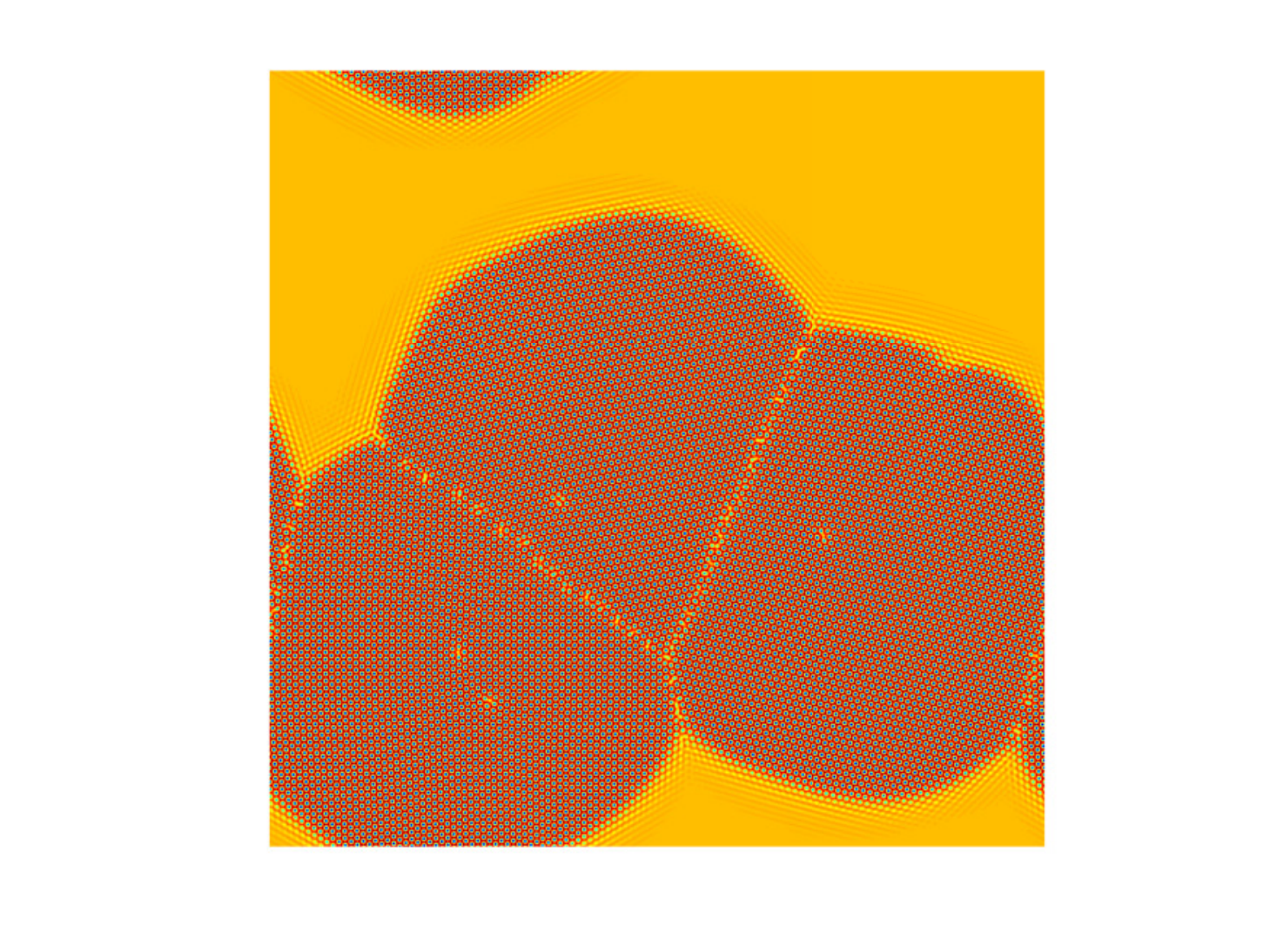}
	\end{minipage}	
	\begin{minipage}{0.24\textwidth}
		\centering
		\includegraphics[width=4.7cm]{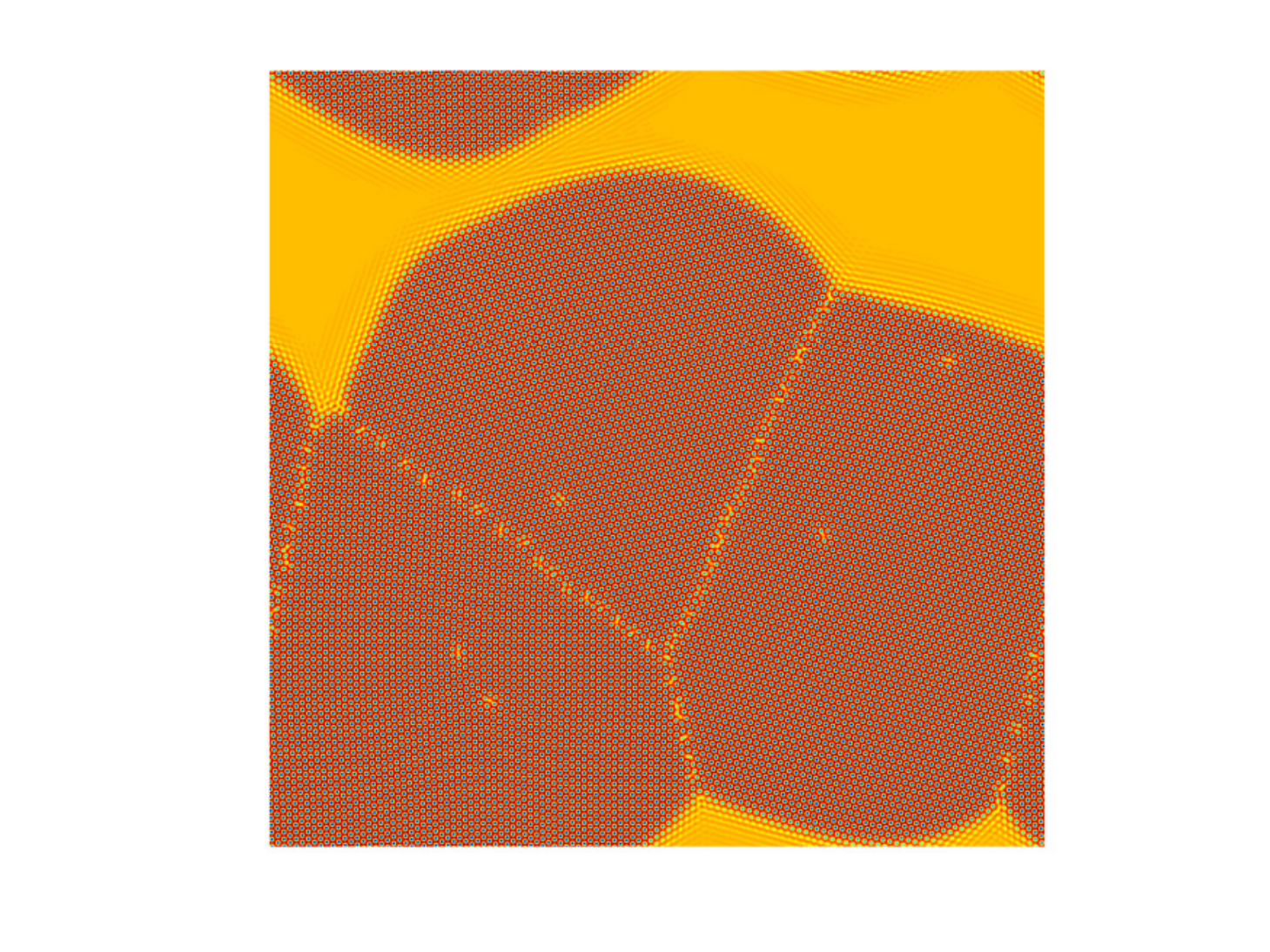}
	\end{minipage}
	\begin{minipage}{0.24\textwidth}
		\centering
		\includegraphics[width=4.7cm]{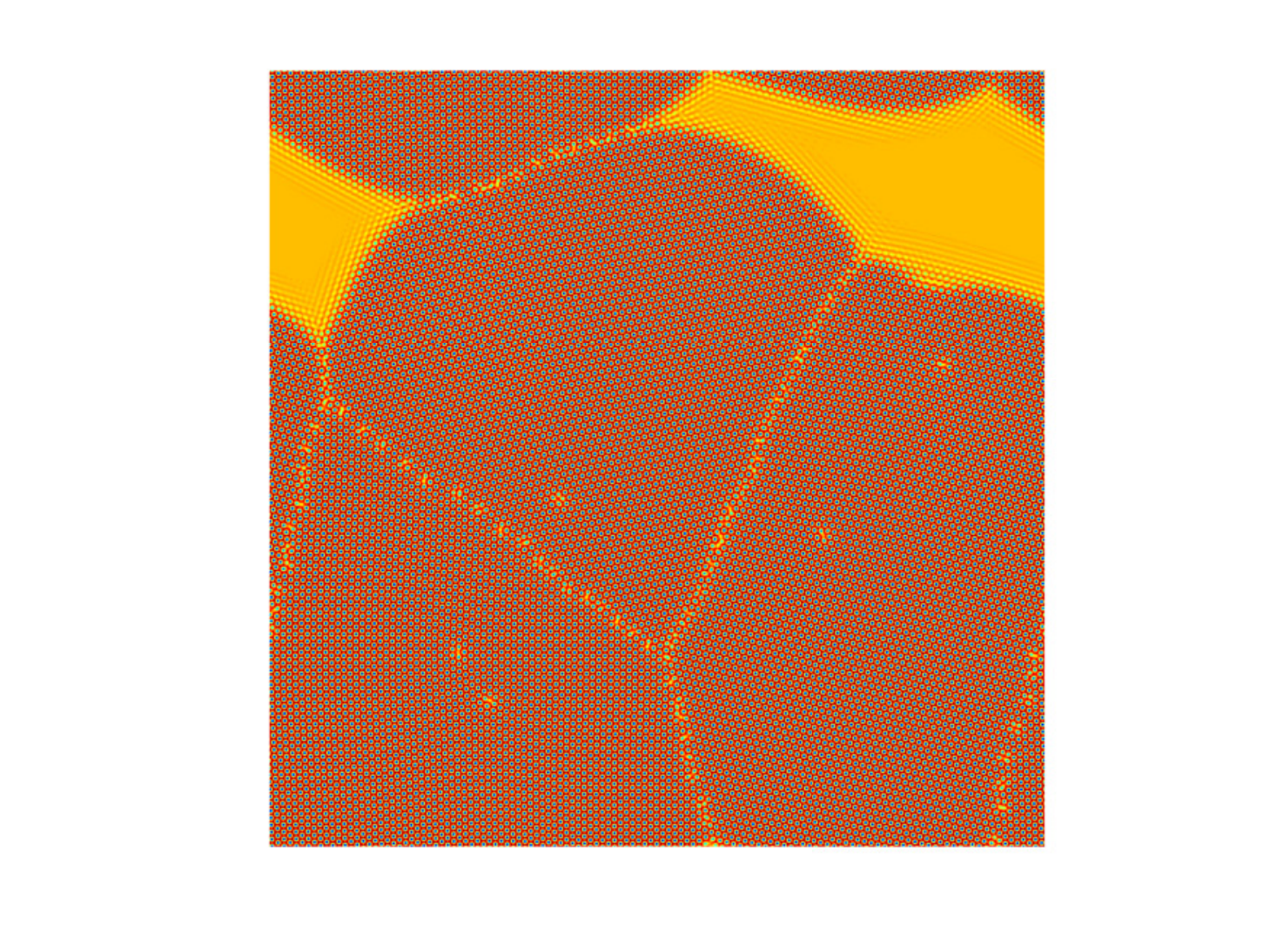}
	\end{minipage}
	\begin{minipage}{0.24\textwidth}
		\centering
		\includegraphics[width=4.7cm]{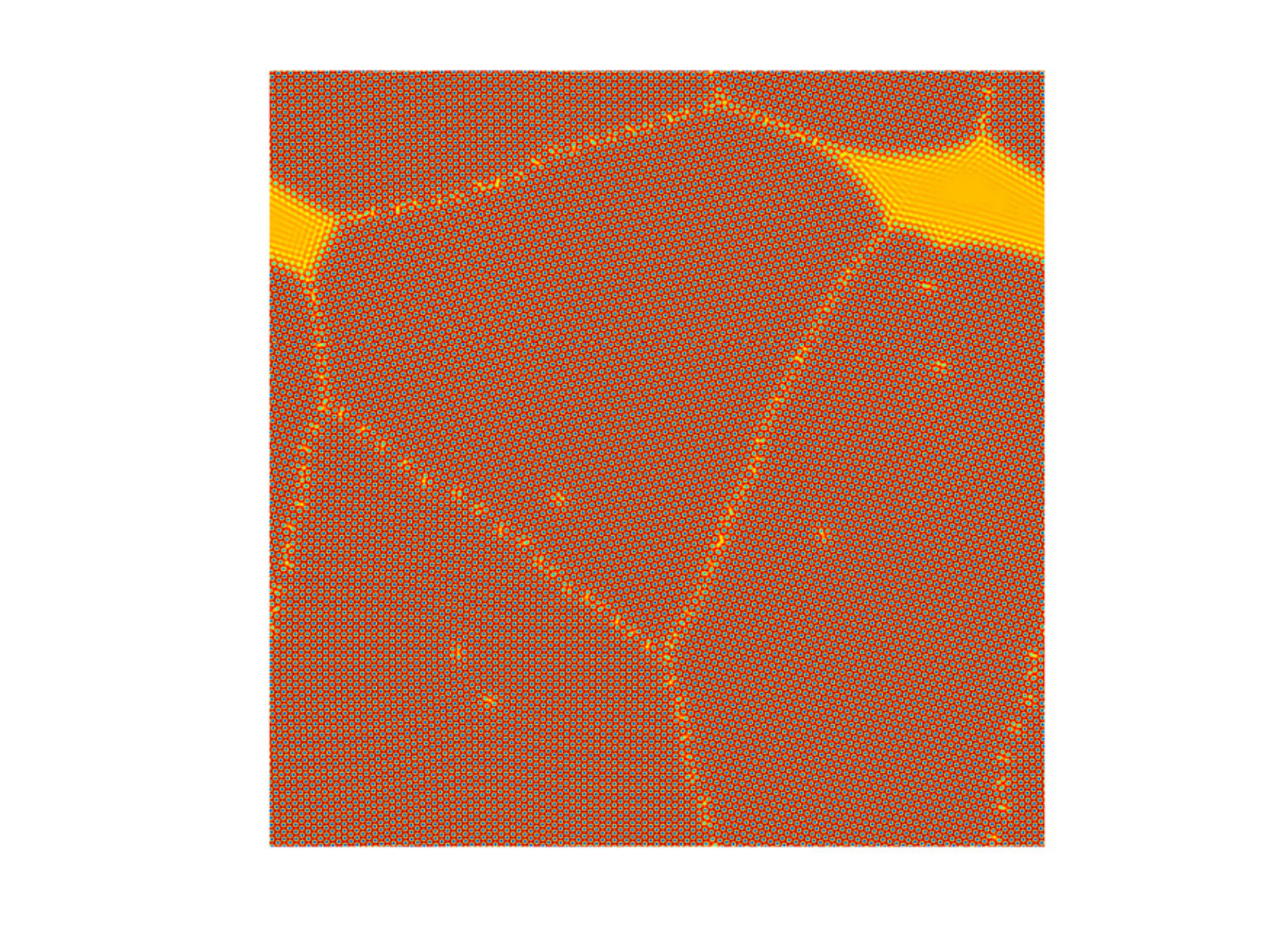}
	\end{minipage}
	\begin{minipage}{0.24\textwidth}
		\centering
		\includegraphics[width=4.7cm]{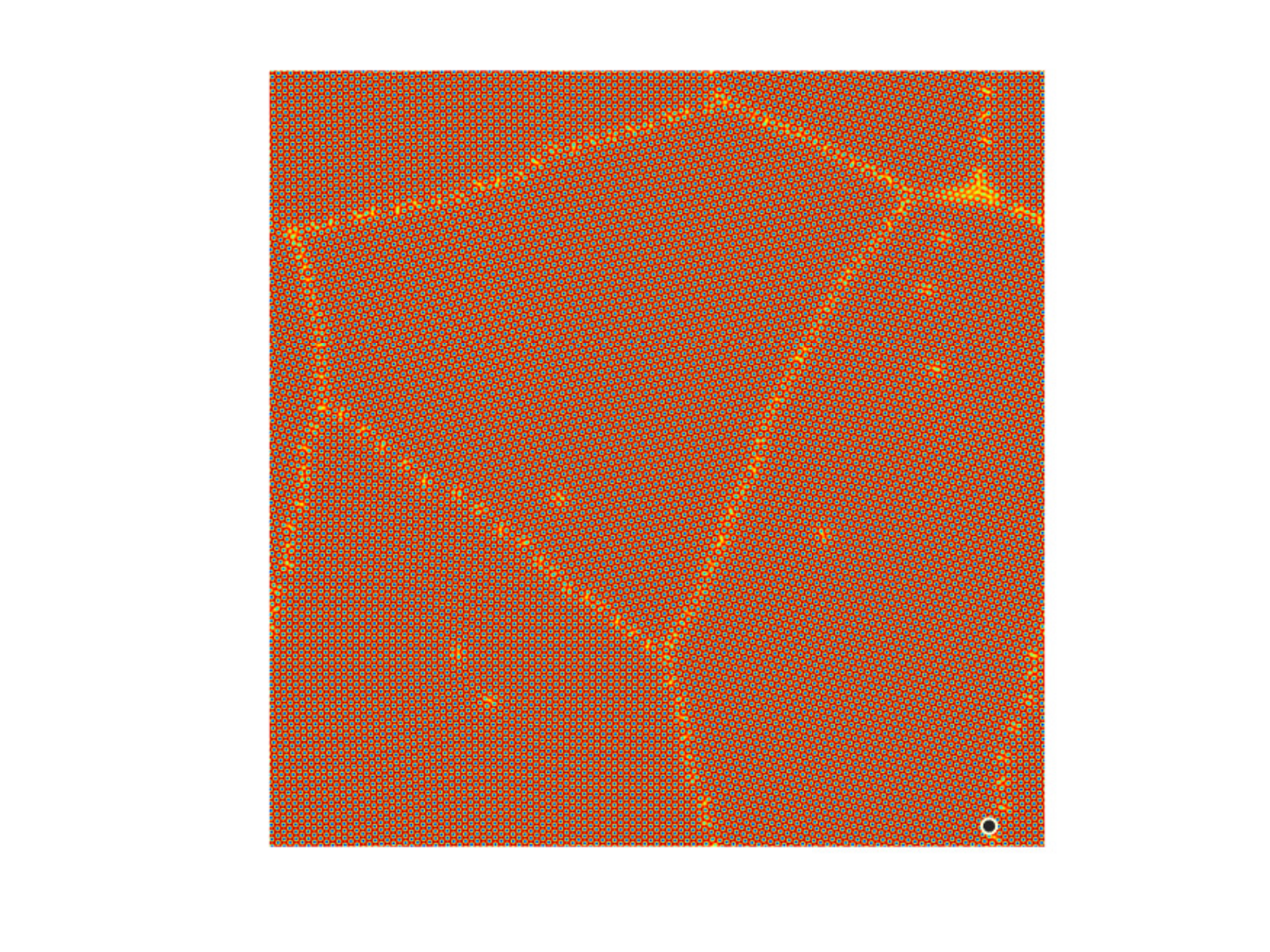}
	\end{minipage}
	\begin{minipage}{0.24\textwidth}
		\centering
		\includegraphics[width=4.7cm]{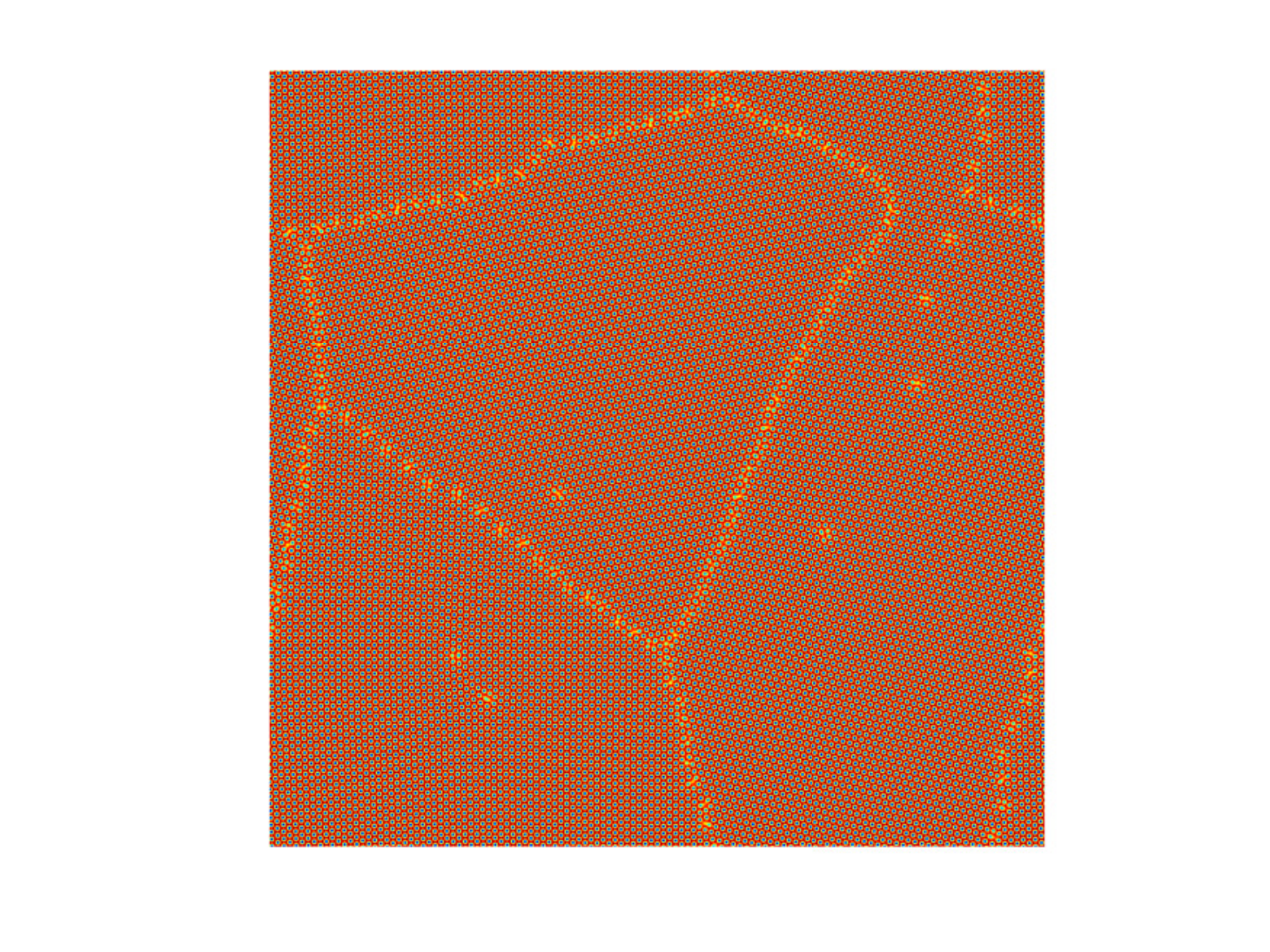}
	\end{minipage}
	\caption{Example\,\ref{ex:PFC}(Case A). Dynamics evolution of crystal growth in a supercooled liquid driven by the PFC equation using EOP-GSAV/BDF$2$ scheme. Snapshots of the numerical solution $\phi$ at $T=0,$ $100,$ $200,$ $300,$ $400,$ $500,$ $600,$ $700,$ $800,$ $900,$ $1000,$ $2000,$ respectively.}
	\label{Fig:PFC-emGSAV-2D}
\end{figure}		
	

{\em Case B.} We investigate the phase transition behaviors in three-dimensional systems using the PFC equation.
We initialize the system with the data $\phi(x, y, t=0)=\bar{\phi}+0.01rand$, where $rand$ is the uniformly distributed random number in $[-1, 1]$ with zeros mean, and set the computational domain to $\left[0, 50\right]^{3}$.
The parameters $\epsilon=0.56$, $\delta=0.02$, $T=3000$, and $64^3$ Fourier modes are chosen for the simulations.
Fig.\,\ref{Fig:PFC-emGSAV-rand-3D} displays the steady-state microstructure of the phase transition behavior for $\bar{\phi}=0.20, 0.35$, and $0.43$, respectively. These results are consistent with those reported in \cite{li2019efficient}.

\begin{figure}[htbp]
	\centering
	\subfigure[$\bar{\phi}=0.2$]{
		\includegraphics[width=5.3cm]{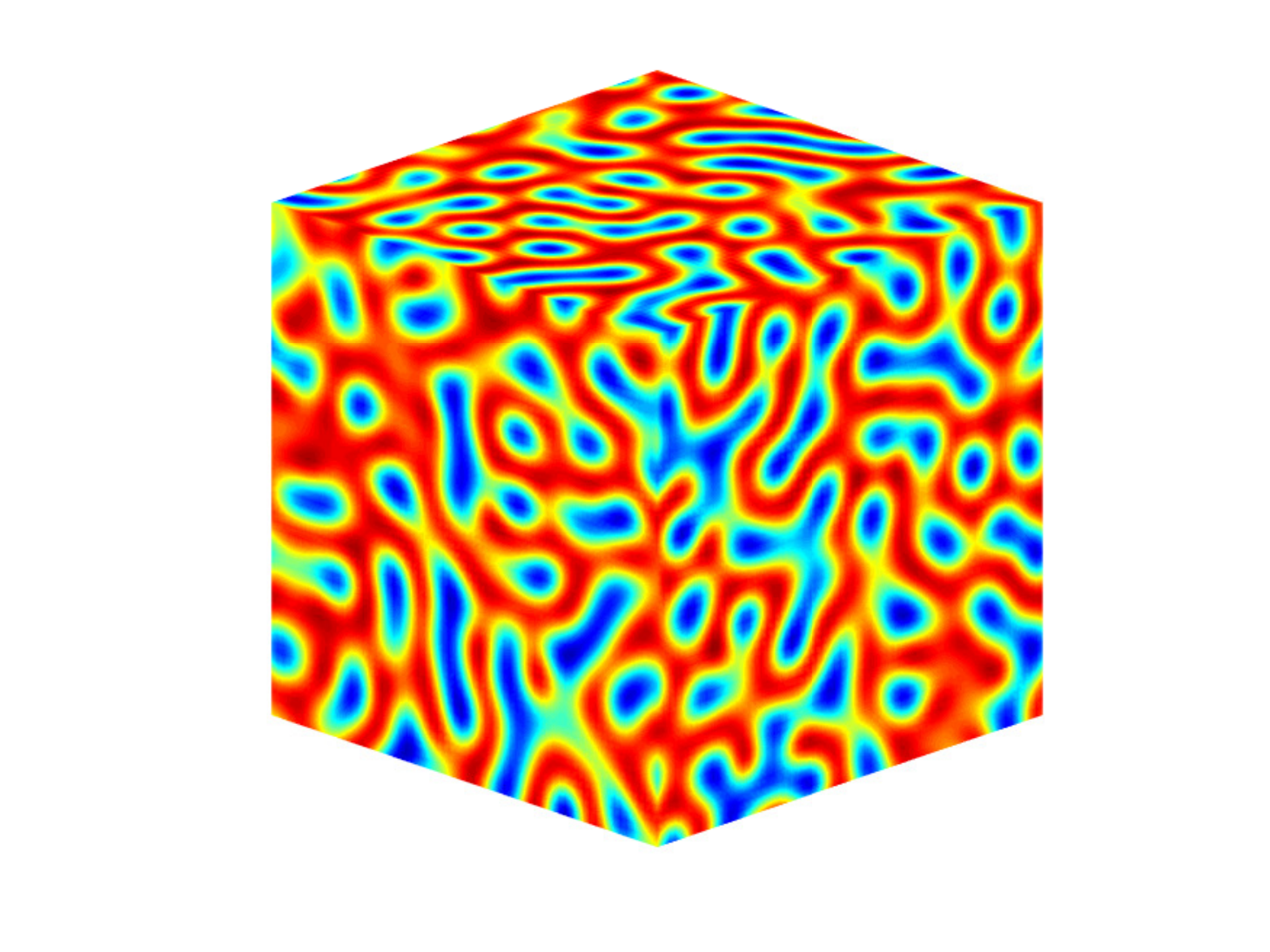}
		\includegraphics[width=5.3cm]{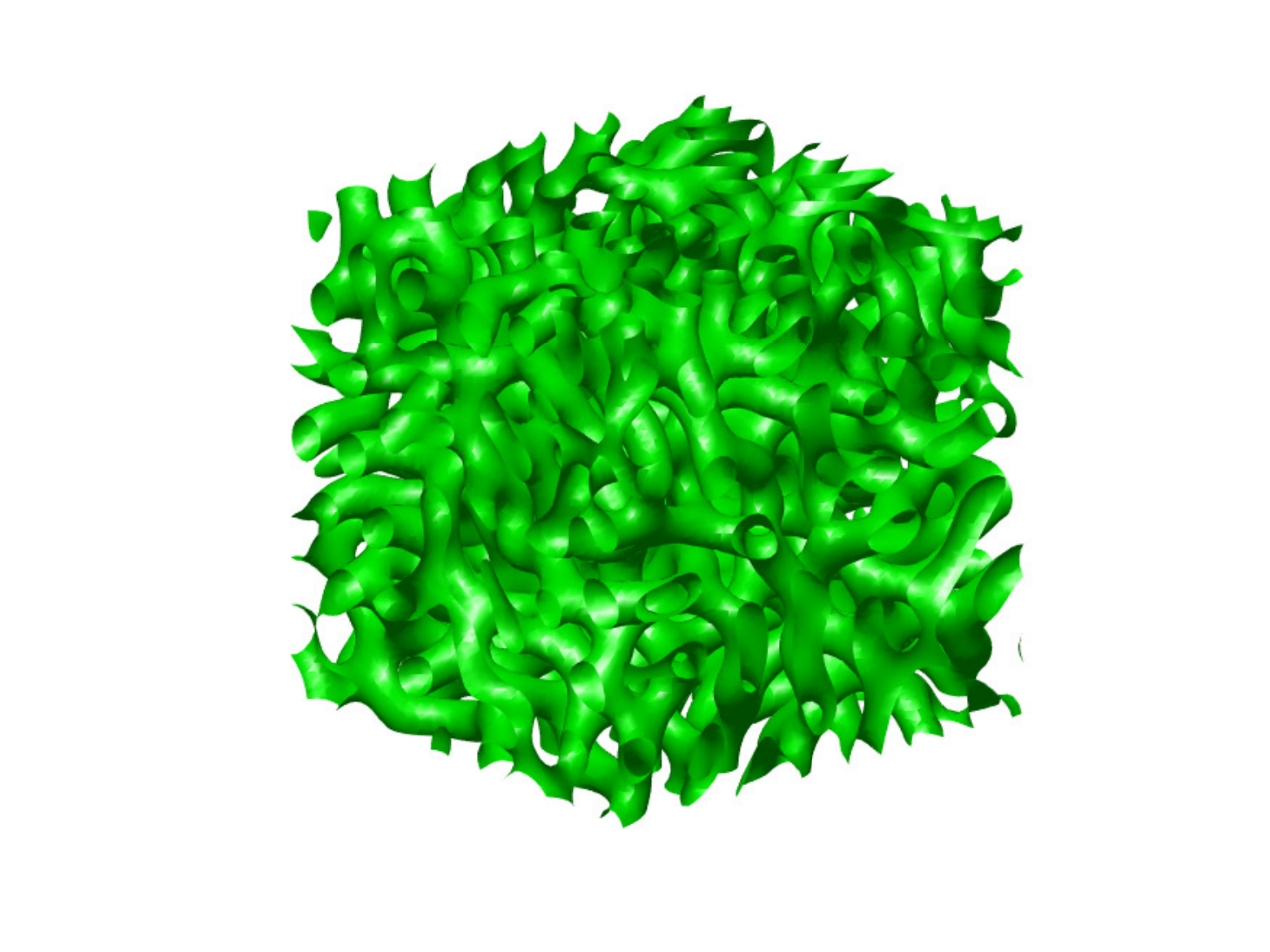}
		
	}
	\subfigure[$\bar{\phi}=0.35$]{
		\includegraphics[width=5.3cm]{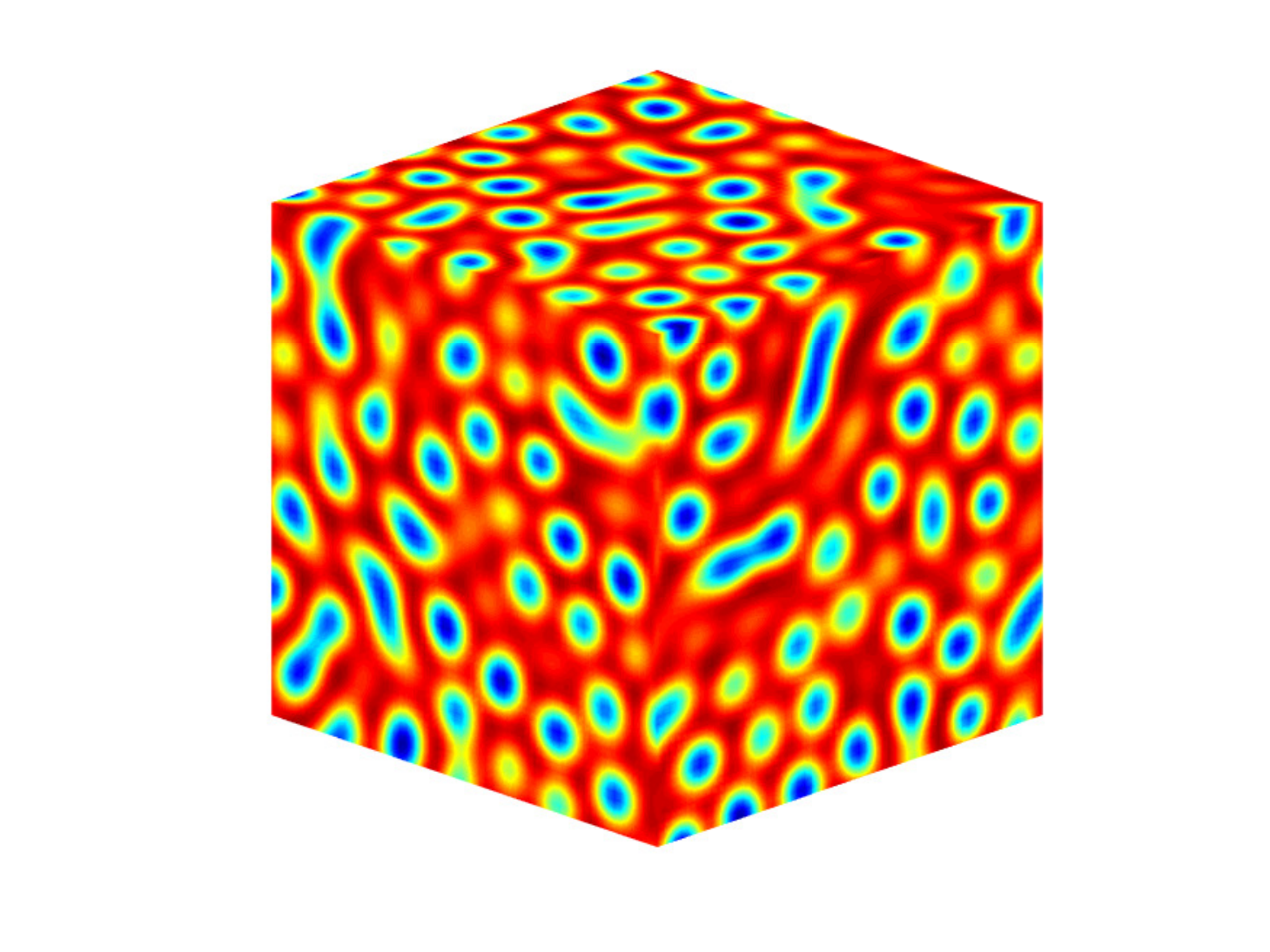}
		\includegraphics[width=5.3cm]{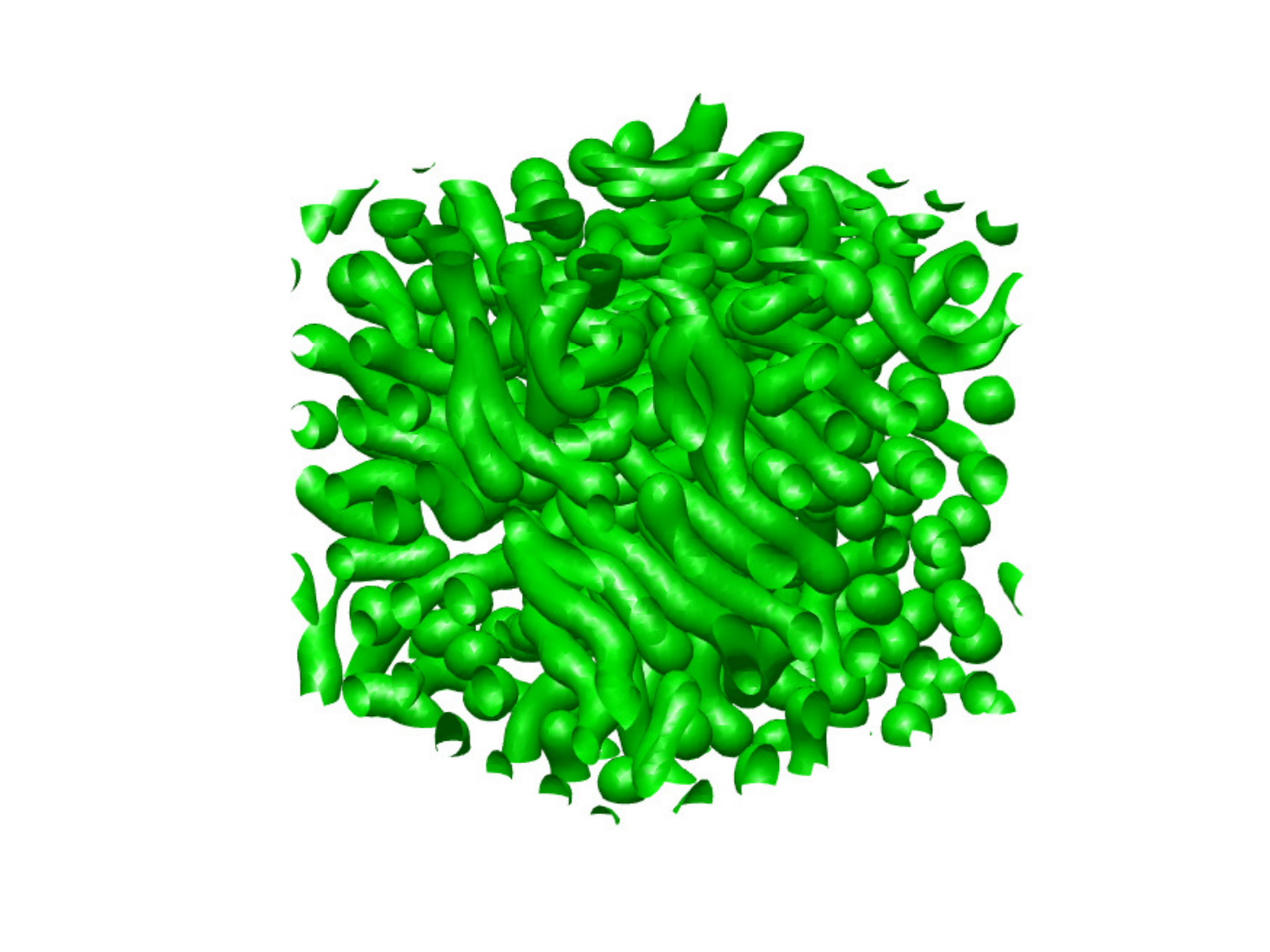}
	}
	\subfigure[$\bar{\phi}=0.43$]{
		\includegraphics[width=5.3cm]{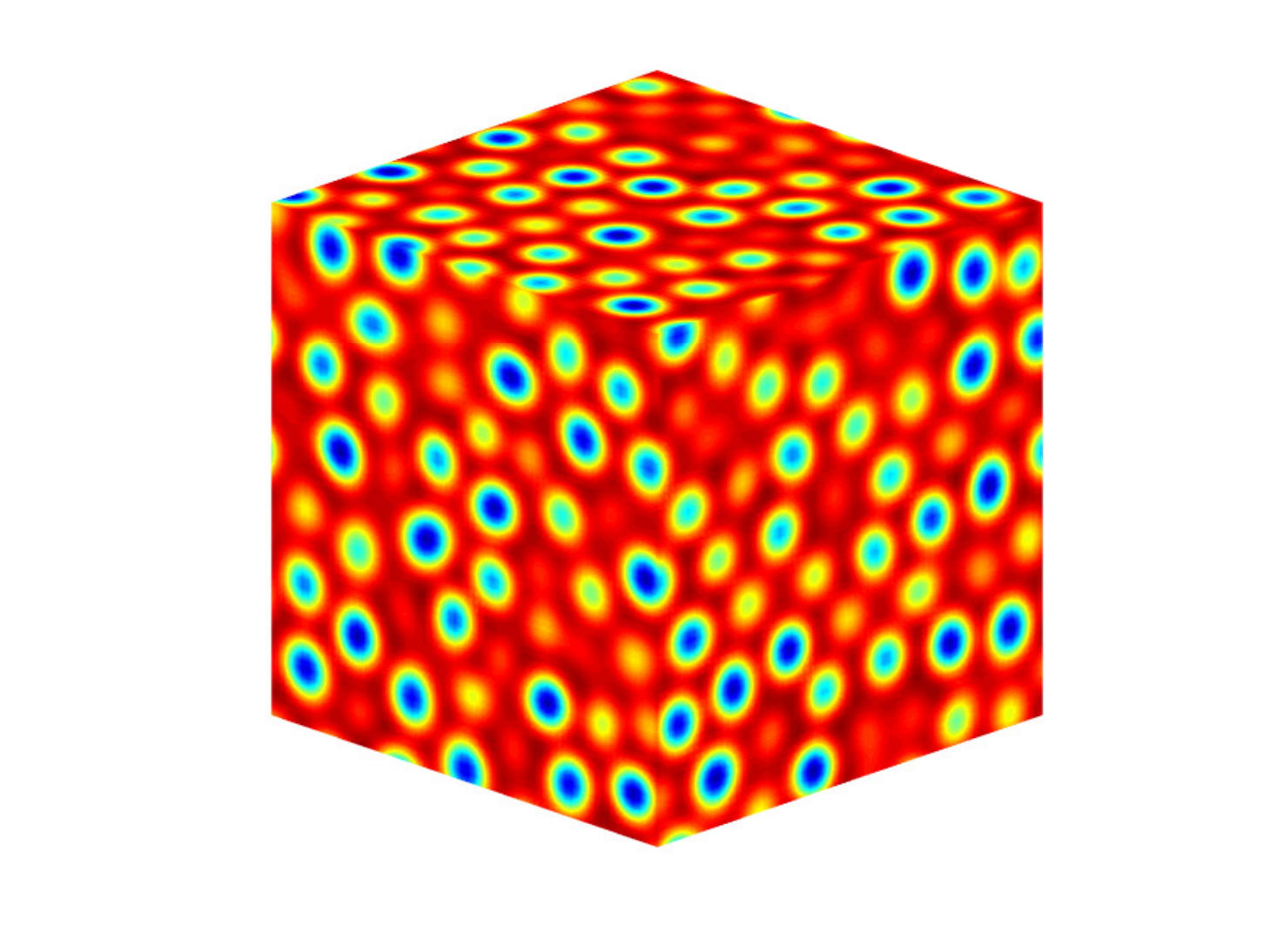}
		\includegraphics[width=5.3cm]{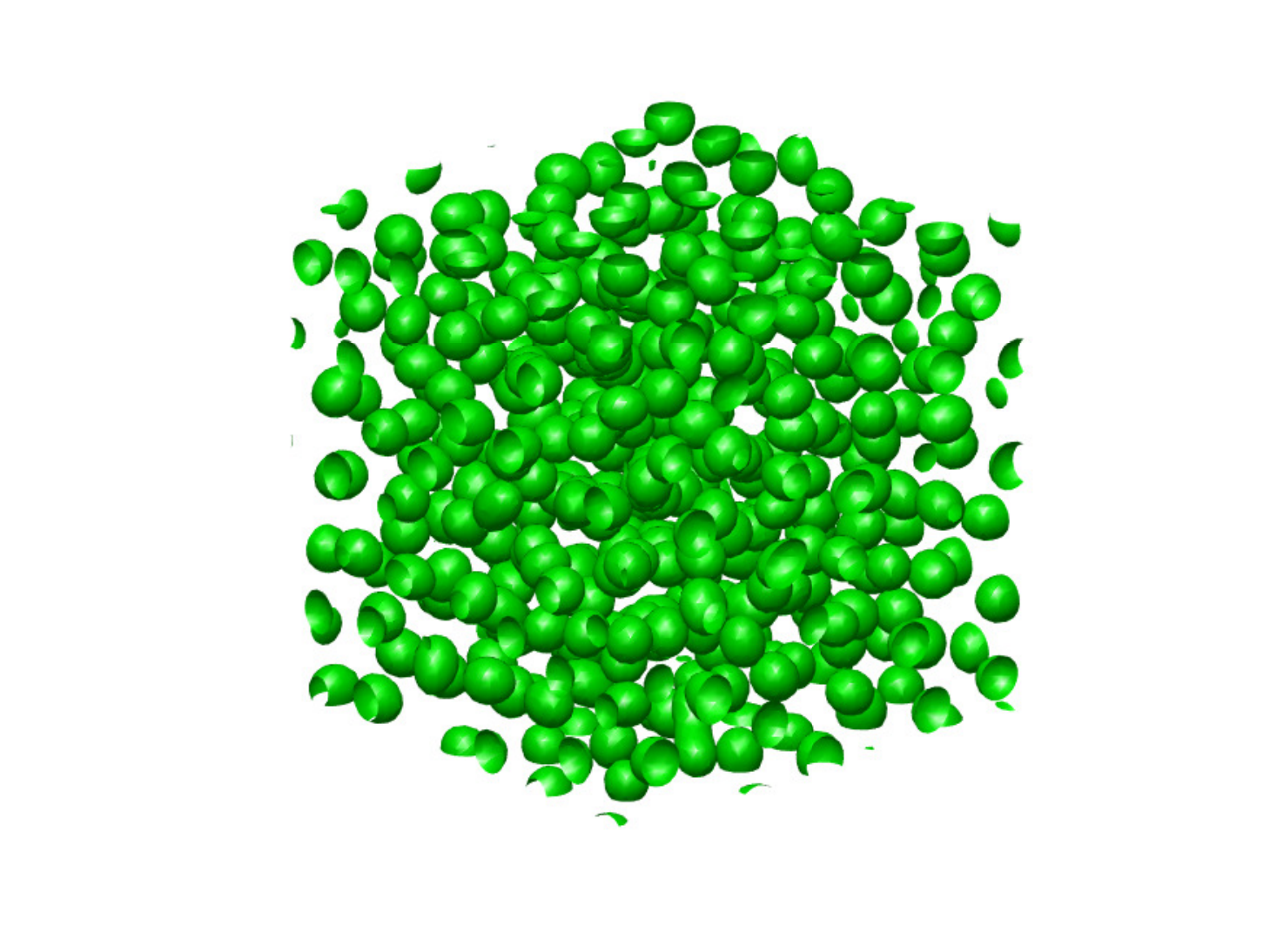}}
	\caption{Example\,\ref{ex:PFC}(Case B). Snapshots of density field $\phi$ (left) and isosurface plots of $\phi=0$ (right) driven by the PFC equation using EOP-GSAV/BDF$2$ scheme at $T=3000$.}
	\label{Fig:PFC-emGSAV-rand-3D}
\end{figure}
\end{example}	

\begin{example}\label{ex:NS}
	\rm
In this numerical example, we evaluate the Navier-Stokes equation by using EOP-GSAV/BDF$k$ ($k=1,2,3,4$) schemes. The Navier-Stokes equation is a well-known dissipative system that can be defined as follows
\begin{equation}\label{eq:Navier-Stokes}
	\left\{\begin{aligned}
		& \frac{\partial \mathbf{u}}{\partial t}-\nu \Delta \mathbf{u}+(\mathbf{u} \cdot \nabla) \mathbf{u}+\nabla p =\mathbf{0} & & \text { in } \Omega \times \mathcal{T} ,\\
		& \nabla \cdot \mathbf{u} =0 & & \text { in } \Omega \times \mathcal{T}, \\
		& \mathbf{u} =\mathbf{0} & & \text { on } \partial \Omega \times \mathcal{T}.
	\end{aligned}\right.
\end{equation}
Let $\Omega$ be an open bounded domain in $\mathbb{R}^{d}$ with a sufficiently smooth boundary $\partial \Omega$, and let $\mathcal{T}=(0, T]$. The unknown velocity and pressure are denoted by $\mathbf{u}$ and $p$, respectively. $\nu>0$ is the viscosity coefficient, and $\mathbf{n}$ is the unit outward normal of the domain $\Omega$.
The system \eqref{eq:Navier-Stokes} satisfies the following law
\begin{equation}\label{eq:Navier-Stokes-energy-law}
	\frac{\mathrm{d}}{\mathrm{d} t} E(\mathbf{u})=-\nu\|\nabla\mathbf{u}\|^{2},
\end{equation}
where $E(\mathbf{u})=\frac{1}{2}\|\mathbf{u}\|^{2}$ is the total energy.

In the case of periodic boundary condition, the operators $\nabla , \nabla \cdot$ and $\Delta^{-1}$ can commute with each other by defining them in the Fourier space.
Applying the divergence operator to both sides of the first equation in \eqref{eq:Navier-Stokes}, we obtain
\begin{equation}
	-\Delta p=\nabla \cdot(\mathbf{u} \cdot \nabla \mathbf{u}),
\end{equation}
then we can derive that
\begin{equation*}
	\begin{aligned}
		\nabla p & =\nabla \Delta^{-1} \Delta p \\
		&=-\nabla \Delta^{-1} \nabla \cdot(\mathbf{u} \cdot \nabla \mathbf{u}) \\
		& =-\nabla \nabla \cdot \Delta^{-1}(\mathbf{u} \cdot \nabla \mathbf{u})\\
		&=-(\Delta+\nabla \times \nabla \times) \Delta^{-1}(\mathbf{u} \cdot \nabla \mathbf{u}) \\
		& =-\mathbf{u} \cdot \nabla \mathbf{u}-\nabla \times \nabla \times \Delta^{-1}(\mathbf{u} \cdot \nabla \mathbf{u})\\
		&=-\boldsymbol{u} \cdot \nabla \mathbf{u}-\mathbf{J}(\mathbf{u} \cdot \nabla \mathbf{u}),
	\end{aligned}
\end{equation*}
where $\mathbf{J}$ is defined by
\begin{equation*}
	\mathbf{J} \mathbf{v}:=\nabla \times \nabla \times \Delta^{-1} \mathbf{v} \quad \forall \mathbf{v} \in \mathbf{L}_{0}^{2}(\Omega).
\end{equation*}
Then the first equation of \eqref{eq:Navier-Stokes} can be rewritten as
\begin{equation*}
	\frac{\partial \mathbf{u}}{\partial t}-\nu \Delta \mathbf{u}-\mathbf{J}(\mathbf{u} \cdot \nabla \mathbf{u})=\mathbf{0}.
\end{equation*}

We introduce a SAV, $R(t)=\mathcal{E}(\mathbf{u})=E(\mathbf{u})+C_0$, with $C_0 \geq 0$, then the time discretization can be constructed as the following form:

\textbf{Step I}: Solve solution $\left(\mathbf{u}^{n+1}, p^{n+1}, \tilde R^{n+1}\right)$:
\begin{equation*}
	\begin{aligned}
		& \frac{\alpha_k \overline{\mathbf{u}}^{n+1}-A_k\left(\overline{\mathbf{u}}^n\right)}{\Delta t}-\nu \Delta \overline{\mathbf{u}}^{n+1}-\mathbf{J}\left(B_k\left(\mathbf{u}^n\right) \cdot \nabla B_k\left(\mathbf{u}^n\right)\right)=0, \\
		& \frac{\left(\tilde{R}^{n+1}-R^n\right)}{\Delta t}=-\nu \frac{\tilde{R}^{n+1}}{E\left(\overline{\mathbf{u}}^{n+1}\right)+C_0}\left\|\nabla \overline{\mathbf{u}}^{n+1}\right\|^2, \\
		& \xi^{n+1}=\frac{\tilde{R}^{n+1}}{E\left(\overline{\mathbf{u}}^{n+1}\right)+C_0}, \\
		& \mathbf{u}^{n+1}=\eta_k^{n+1} \overline{\mathbf{u}}^{n+1} \text { with } \eta_k^{n+1}=1-\left(1-\xi^{n+1}\right)^k, \\
		& \Delta p^{n+1}=-\nabla \cdot\left(\mathbf{u}^{n+1} \cdot \nabla \mathbf{u}^{n+1}\right),
	\end{aligned}
\end{equation*}

\textbf{Step II}: Update the scalar auxiliary variable $R^{n+1}$ via
\begin{equation*}
	R^{n+1}=\min\left\{R^n,\mathcal{E}(\mathbf{u}^{n+1})\right\}.
\end{equation*}
	
{\em Case A.}
We begin by conducting an accuracy test, where the right-hand side is computed based on the following analytical solution
\begin{equation*}
	\begin{aligned}
		& u_1(x, y)=\pi \exp (\sin (\pi x)) \exp (\sin (\pi y)) \cos (\pi y) \sin ^2(t), \\
		& u_2(x, y)=-\pi \exp (\sin (\pi x)) \exp (\sin (\pi y)) \cos (\pi x) \sin ^2(t), \\
		& p(x, y)=\exp (\cos (\pi x) \sin (\pi y)) \sin ^2(t).
	\end{aligned}
\end{equation*}
For this simulation, we consider the computational domain $\Omega=(0,2)^2$, with $\nu=1$ and a computational time interval from $t=2$ to $t=3$. We adopt a spatial discretization scheme using $N^2=40^2$ Fourier modes. The $L^{2}$-norm errors for EOP-GSAV/BDF$k$ schemes with $k=1,2,3,4$ are presented in Fig.\,\ref{Fig:NS-order-test}, demonstrating the expected convergence rates of the numerical scheme.
	
\begin{figure}[htbp]
	\centering
	\begin{minipage}{0.4\textwidth}
		\centering
		\includegraphics[width=5.3cm]{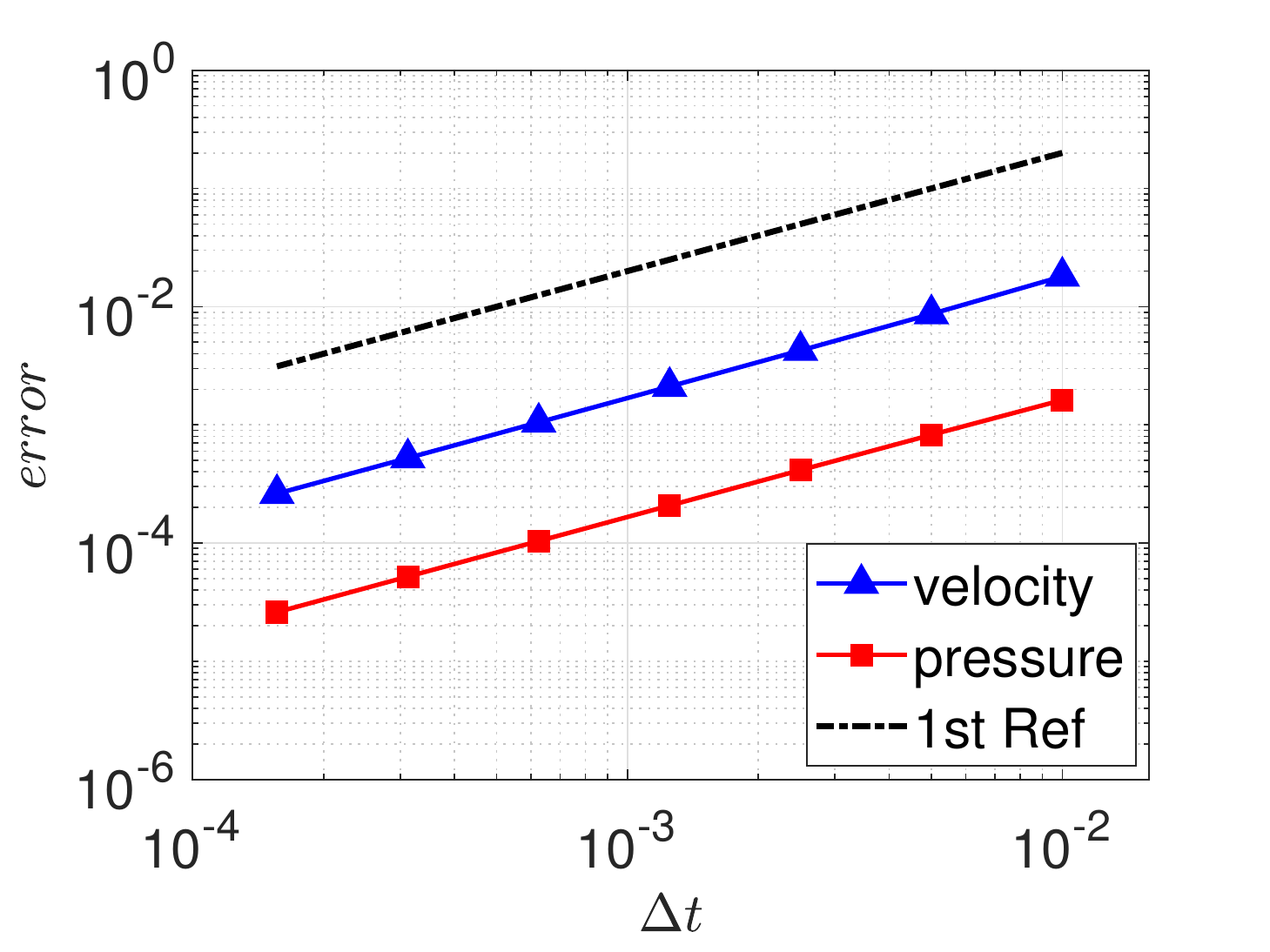}
	\end{minipage}	
	\begin{minipage}{0.4\textwidth}
		\centering
		\includegraphics[width=5.3cm]{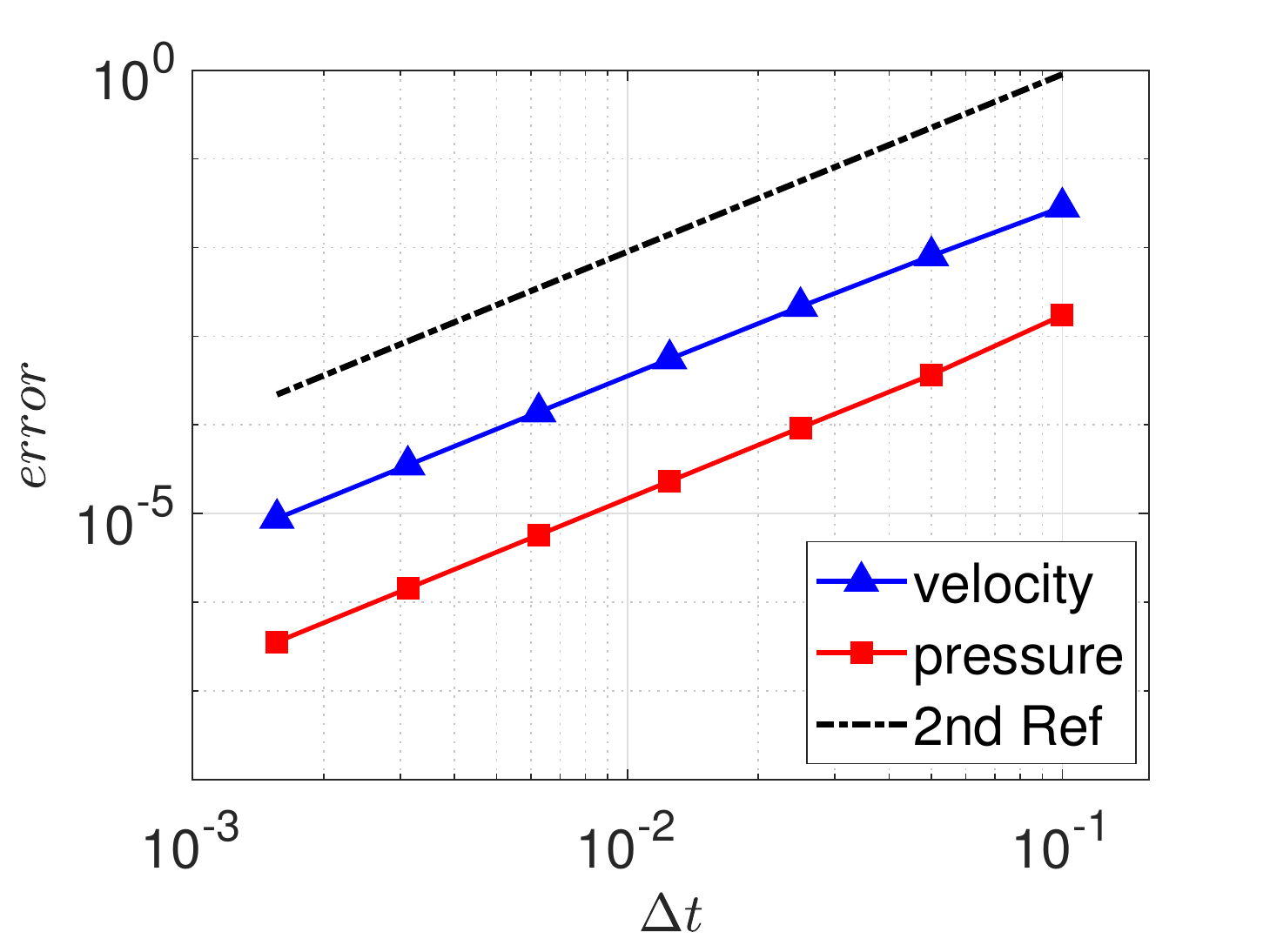}
	\end{minipage}
	\begin{minipage}{0.4\textwidth}
		\centering
		\includegraphics[width=5.3cm]{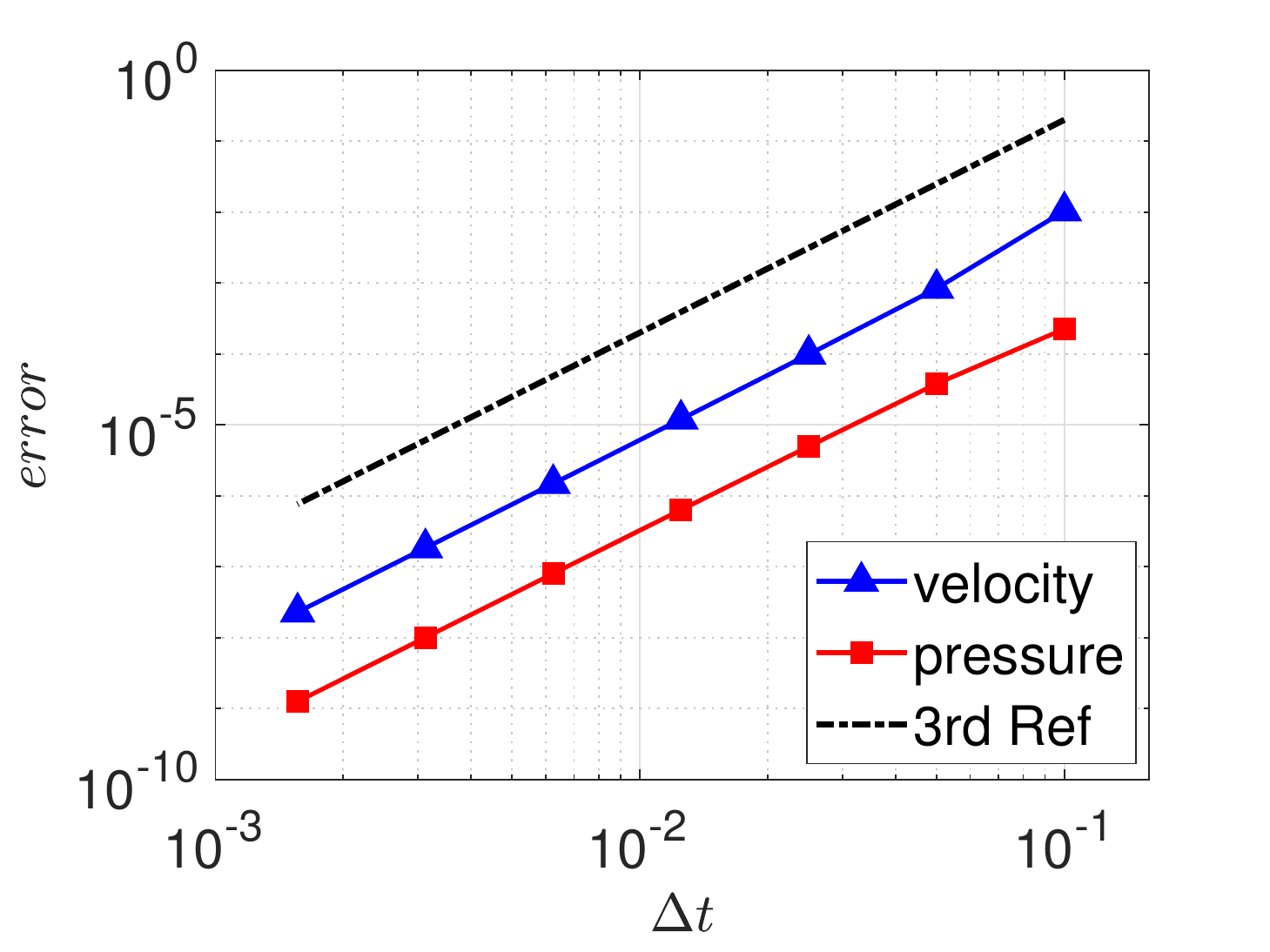}
	\end{minipage}
	\begin{minipage}{0.4\textwidth}
		\centering
		\includegraphics[width=5.3cm]{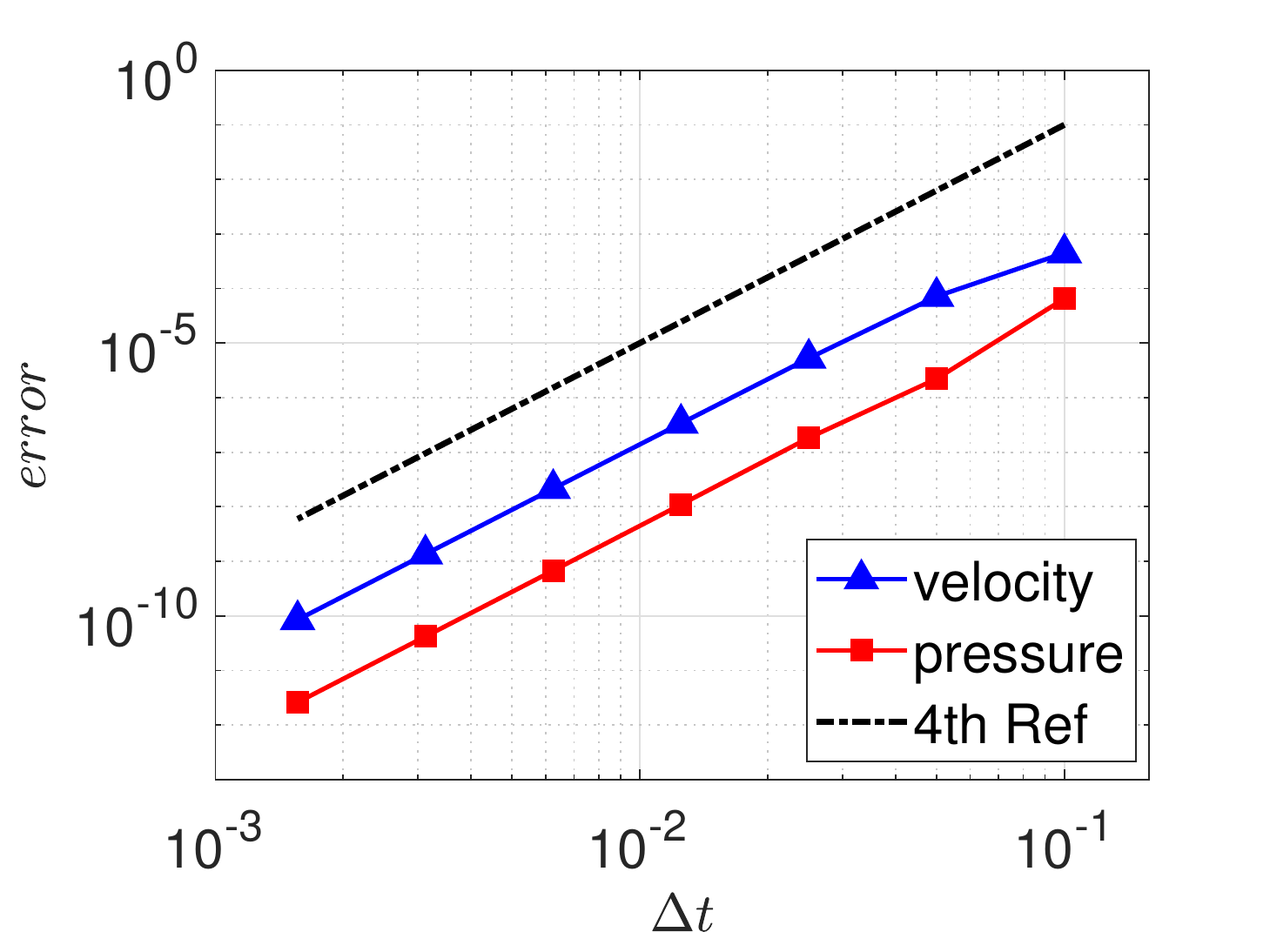}
	\end{minipage}
	\caption{Example\,\ref{ex:NS}(Case A). Convergence rates for Navier-Stokes equation using EOP-GSAV/BDF$k$, $(k=1,2,3,4)$ schemes.}
	\label{Fig:NS-order-test}
\end{figure}

{\em Case B.}  Next we simulate double shear layer problem.
We consider the initial condition as follows
\begin{equation*}
	\begin{array}{l}u_{1}(x, y, 0)=\left\{\begin{array}{l}\tanh (\rho(y-0.25)), y \leq 0.5, \\
			\tanh (\rho(0.75-y)), y>0.5,
		\end{array}\right. \\
		u_{2}(x, y, 0)=\epsilon \sin (2 \pi x).
	\end{array}
\end{equation*}
Here, we consider the double shear layer problem in the Navier-Stokes equation, where the parameter $\rho$ represents the width of the shear layer and $\epsilon$ denotes the size of the perturbation.
We set $\epsilon=0.05$ and choose a computational domain of $\Omega=(0,1)^2$ for the simulations.
The evolution of vorticity contours with $\rho=30$, $\nu=1e-4$, $N^2=128^2$, and a time step of $\Delta t = 6e-4$ obtained using the EOP-GSAV/BDF$2$ scheme is depicted in Fig.\,\ref{Fig:NS-double-shear-1}.
The results illustrate that the vortex gradually increases over time.
We also simulate a more challenging case with $\rho=100$, $\nu=5e-4$, $N^2=256^2$, and a time step of $\Delta t = 2e-4$, as shown in Fig.\,\ref{Fig:NS-double-shear-2}.
Fig.\,\ref{Fig:NS-double-shear-E-difference} displays the evolution of the difference between the original energy and the modified energy for these two cases.
We also test the Navier-Stokes equation using $\rho=30$, $\nu=5e-4$, $N^2=128^2$, and a time step of $\Delta t = 6.7e-4$. The vorticity contours at $T=1.2$ using EOP-GSAV/BDF$k$ with $k=1,2,3,4$ schemes are presented in Fig.\,\ref{Fig:NS-double-shear-comparison}.
The results show that the BDF$3$ and BDF$4$ schemes yield correct solutions, while the BDF$1$ scheme leads to a completely wrong result and the BDF$2$ scheme produces inaccurate results. This numerical phenomenon highlights the superiority of high-order schemes.

	\begin{figure}[htbp]
	\centering
	\begin{minipage}{0.3\textwidth}
		\centering
		\includegraphics[width=5.3cm]{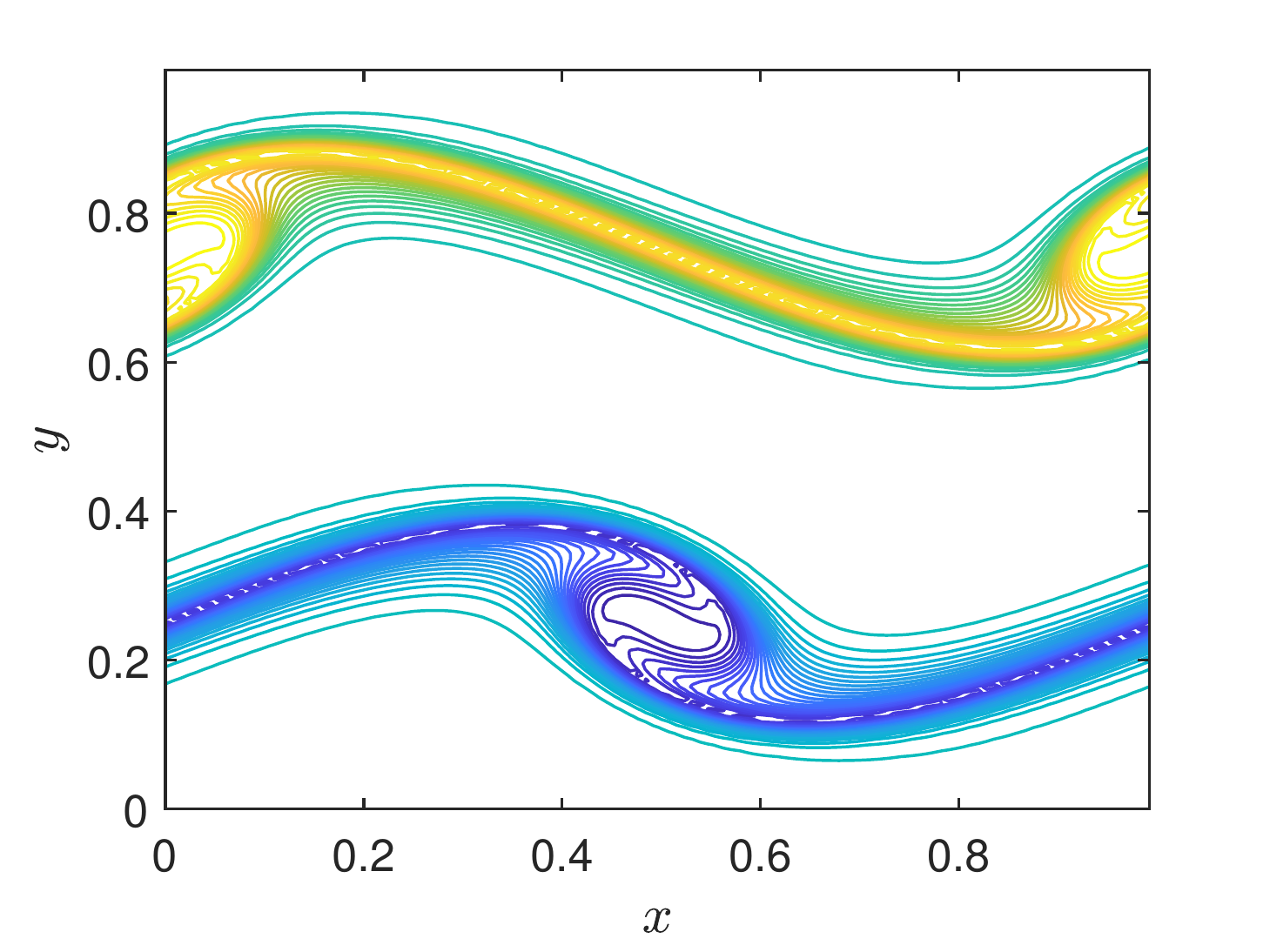}
	\end{minipage}
	\begin{minipage}{0.3\textwidth}
		\centering
		\includegraphics[width=5.3cm]{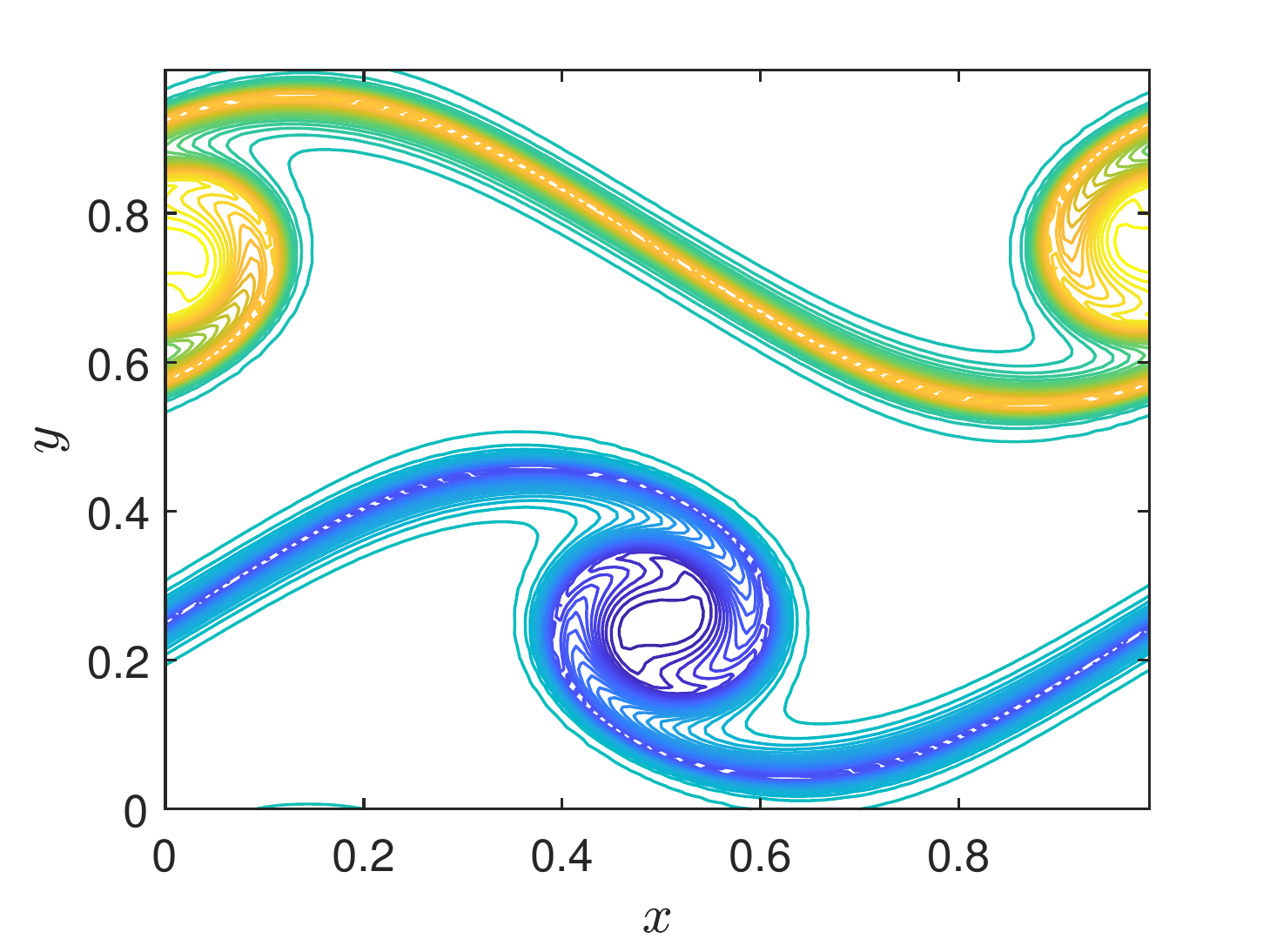}
	\end{minipage}
	\begin{minipage}{0.3\textwidth}
		\centering
		\includegraphics[width=5.3cm]{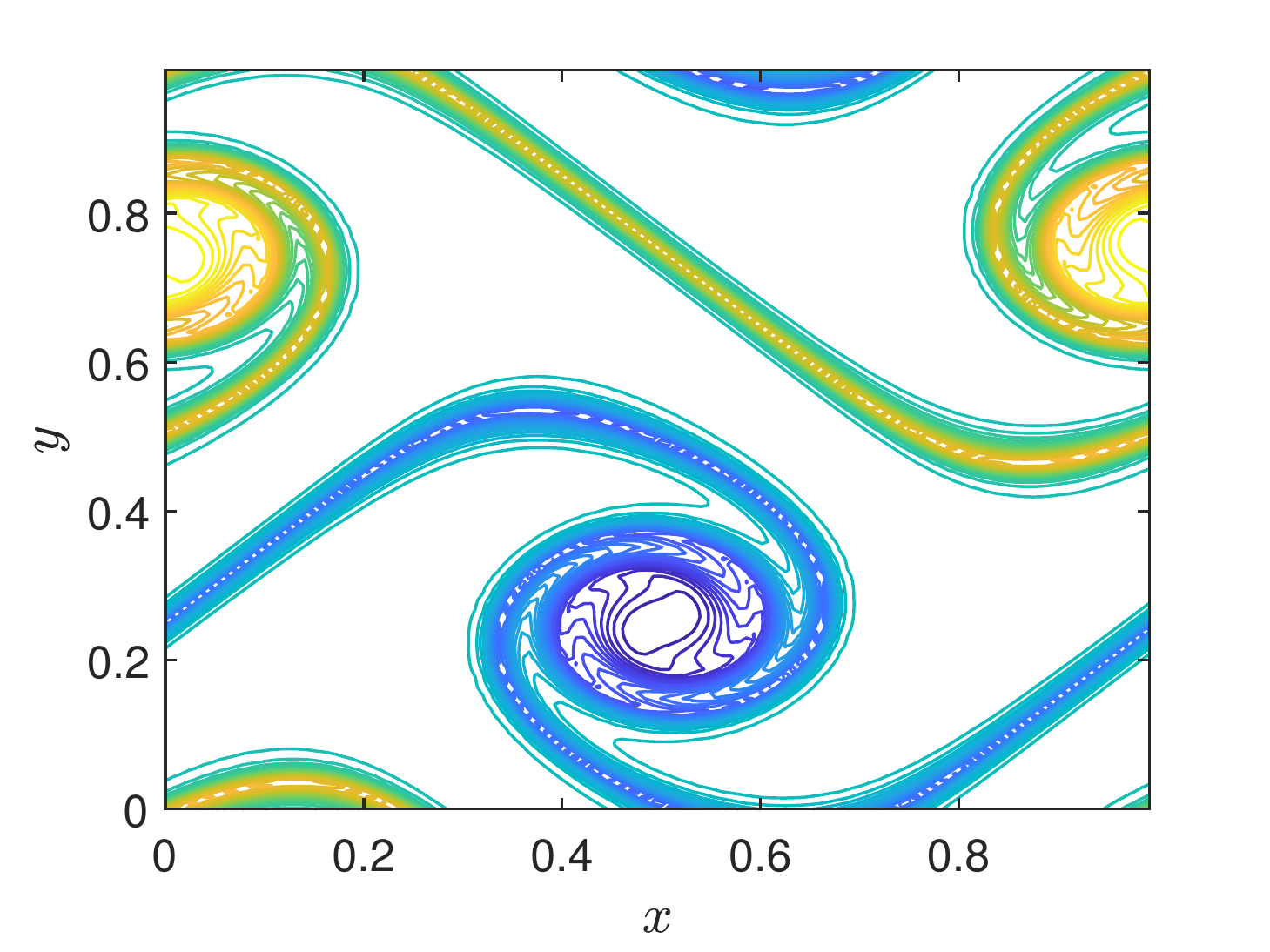}
	\end{minipage}
	\caption{Example\,\ref{ex:NS}(Case B). The evolution of vorticity of Navier-Stokes equation with $\rho = 30, \nu  = 1e-4$, using EOP-GSAV/BDF$2$ scheme at $T=0.8, 1, 1.2$.}
	\label{Fig:NS-double-shear-1}
\end{figure}

	\begin{figure}[htbp]
	\centering
	\begin{minipage}{0.3\textwidth}
		\centering
		\includegraphics[width=5.3cm]{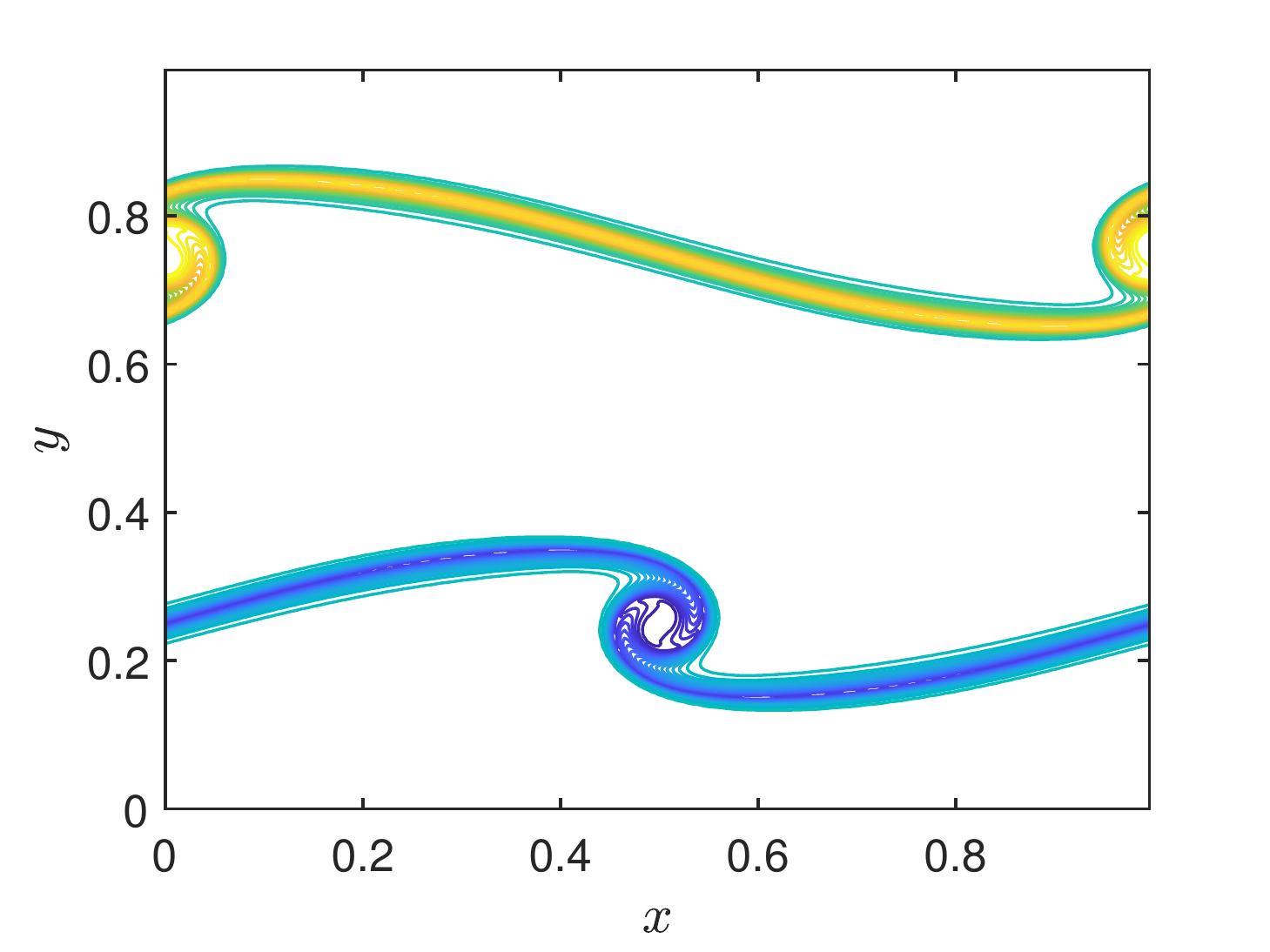}
	\end{minipage}
	\begin{minipage}{0.3\textwidth}
		\centering
		\includegraphics[width=5.3cm]{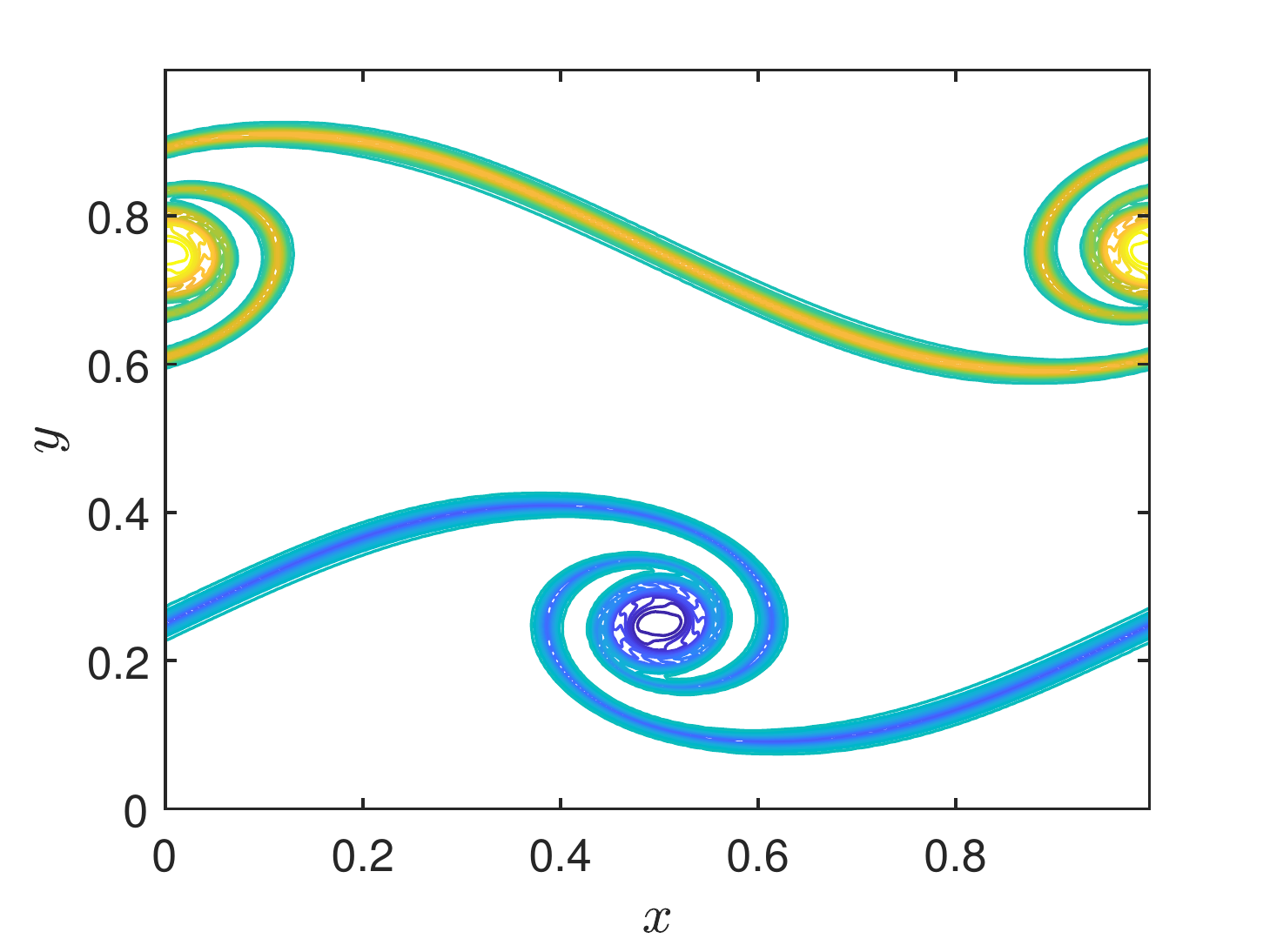}
	\end{minipage}
	\begin{minipage}{0.3\textwidth}
		\centering
		\includegraphics[width=5.3cm]{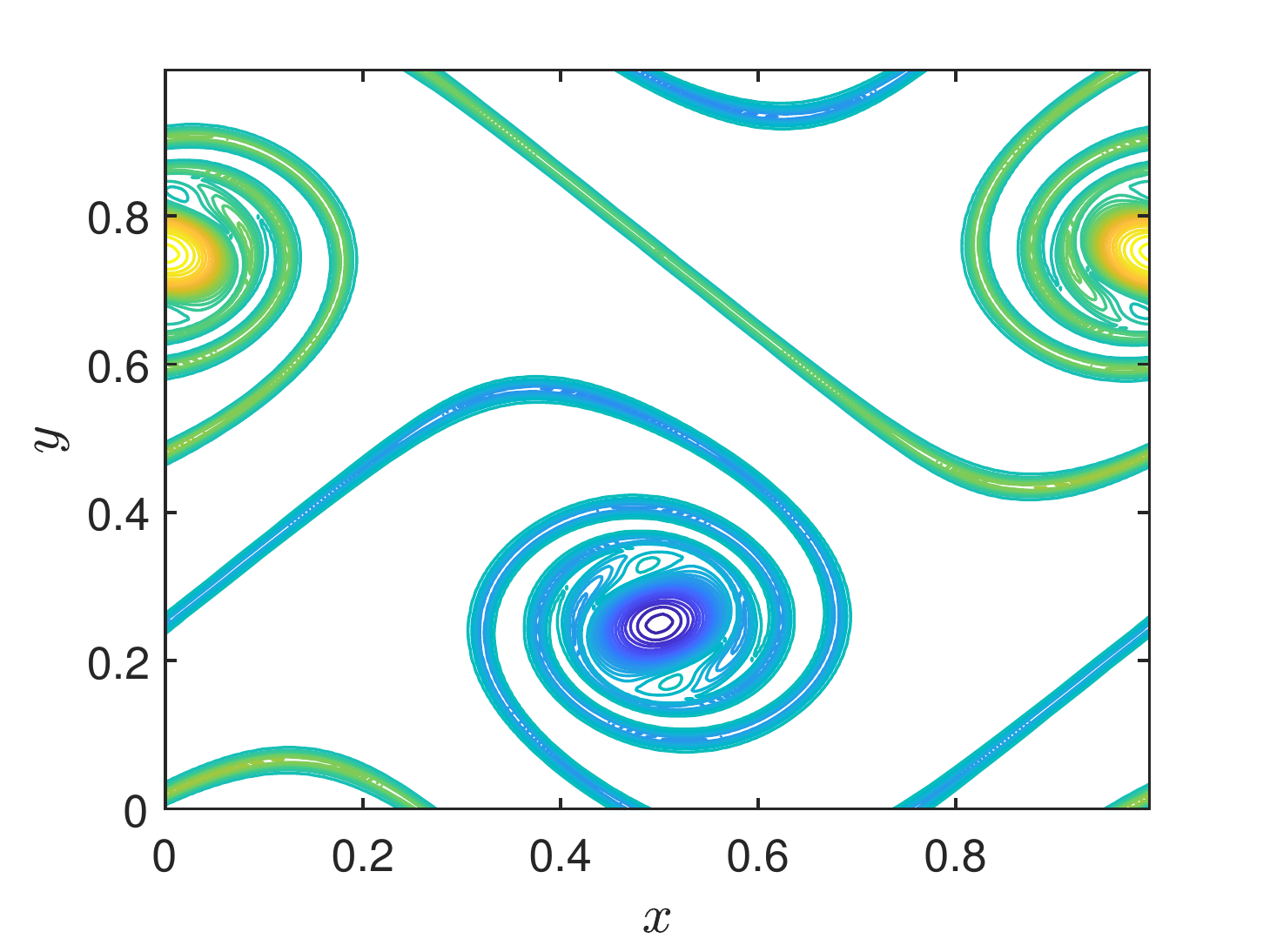}
	\end{minipage}
	\caption{Example\,\ref{ex:NS}(Case B). The evolution of vorticity of Navier-Stokes equation with $\rho = 100, \nu  = 5e-5$ using EOP-GSAV/BDF$2$ scheme at $T=0.6, 0.8, 1.2$.}
	\label{Fig:NS-double-shear-2}
\end{figure}

	\begin{figure}[htbp]
	\centering
	\begin{minipage}{0.4\textwidth}
		\centering
		\includegraphics[width=5.3cm]{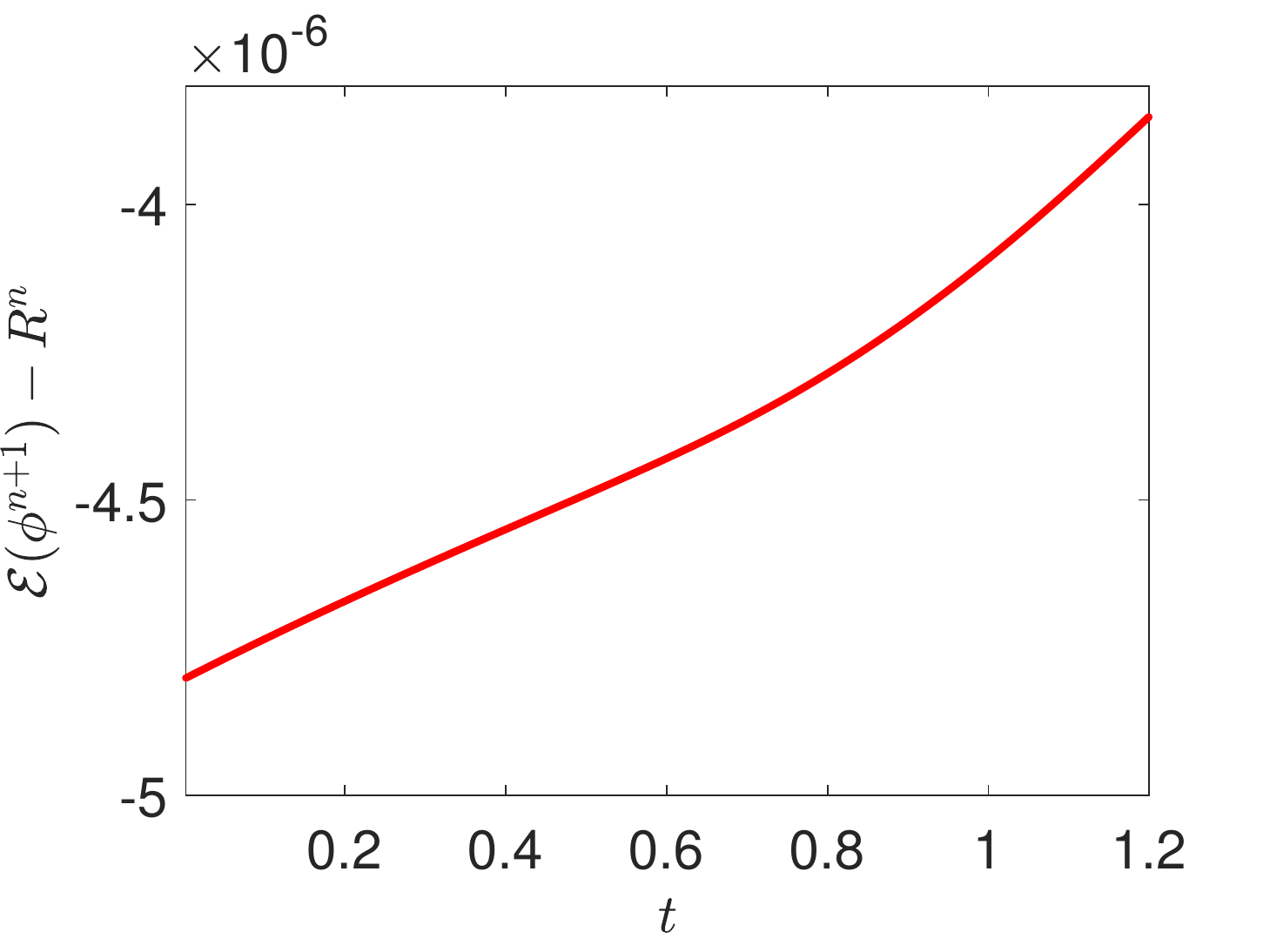}
	\end{minipage}
	\begin{minipage}{0.4\textwidth}
		\centering
		\includegraphics[width=5.3cm]{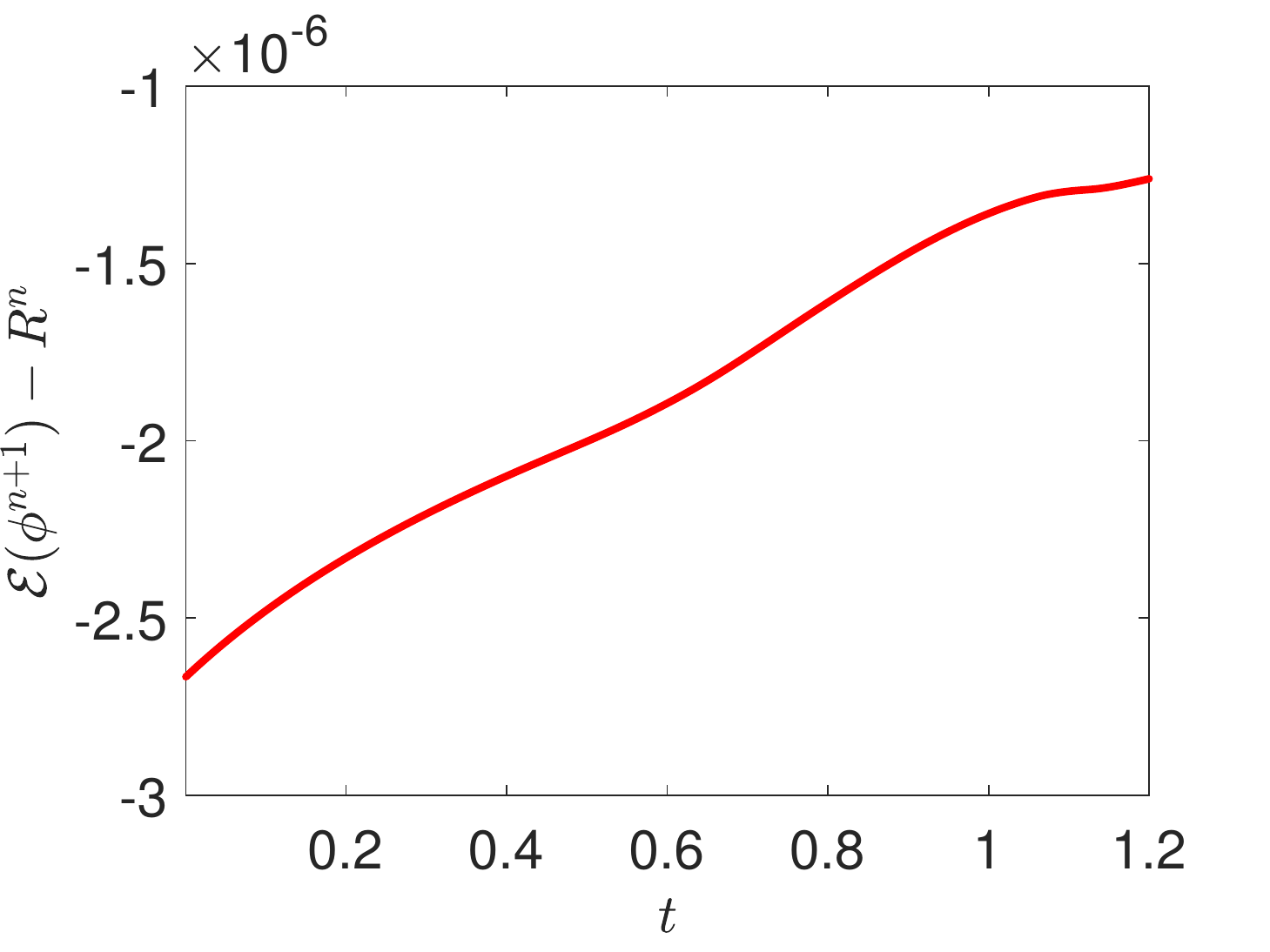}
	\end{minipage}
	\caption{Example\,\ref{ex:NS}(Case B). Evolution of the difference between the original energy and the modified energy using EOP-GSAV/BDF$2$ scheme.}
	\label{Fig:NS-double-shear-E-difference}
\end{figure}

\begin{figure}[htbp]
	\centering
	\begin{minipage}{0.4\textwidth}
		\centering
		\includegraphics[width=5.3cm]{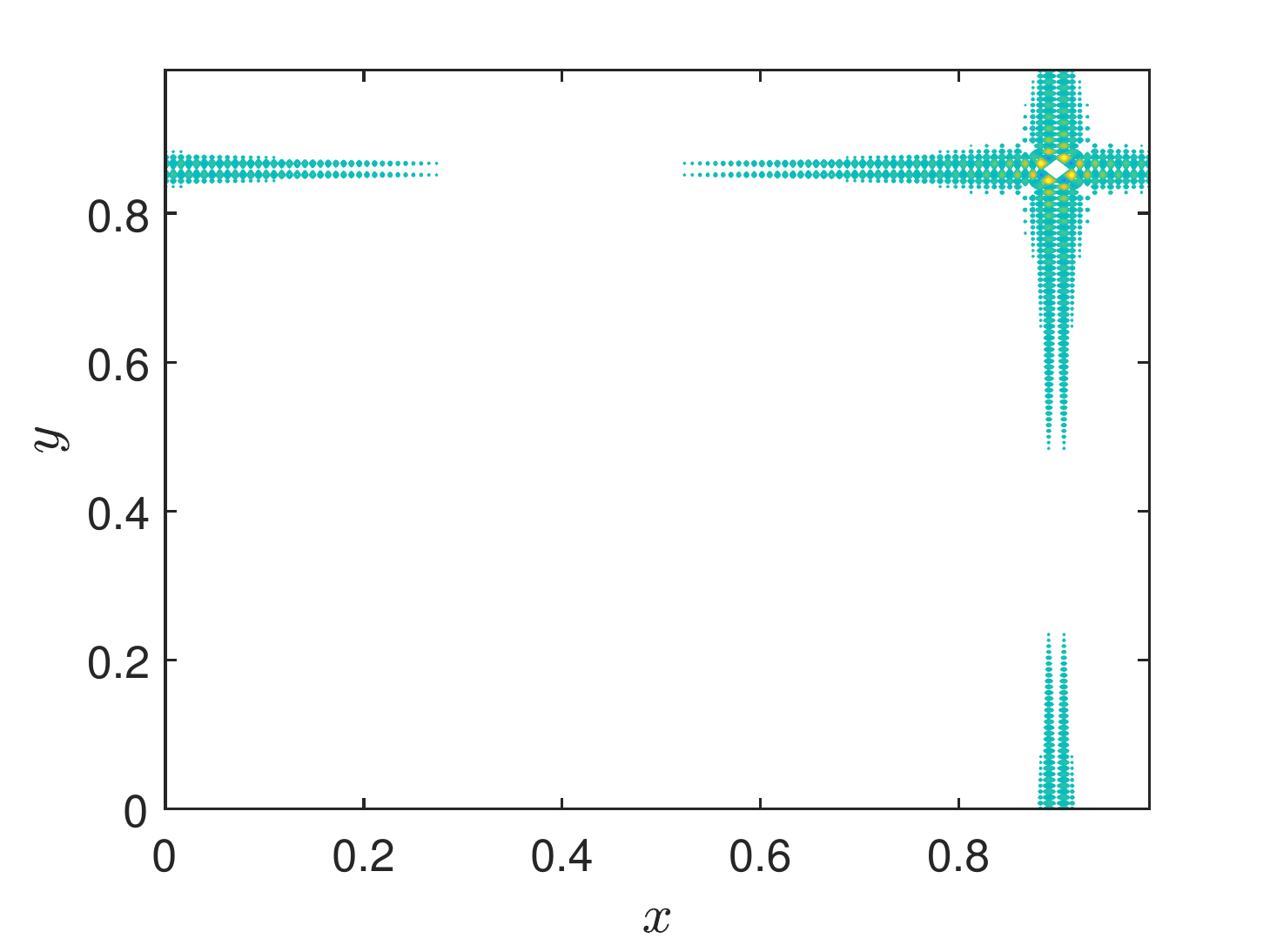}
	\end{minipage}	
	\begin{minipage}{0.4\textwidth}
		\centering
		\includegraphics[width=5.3cm]{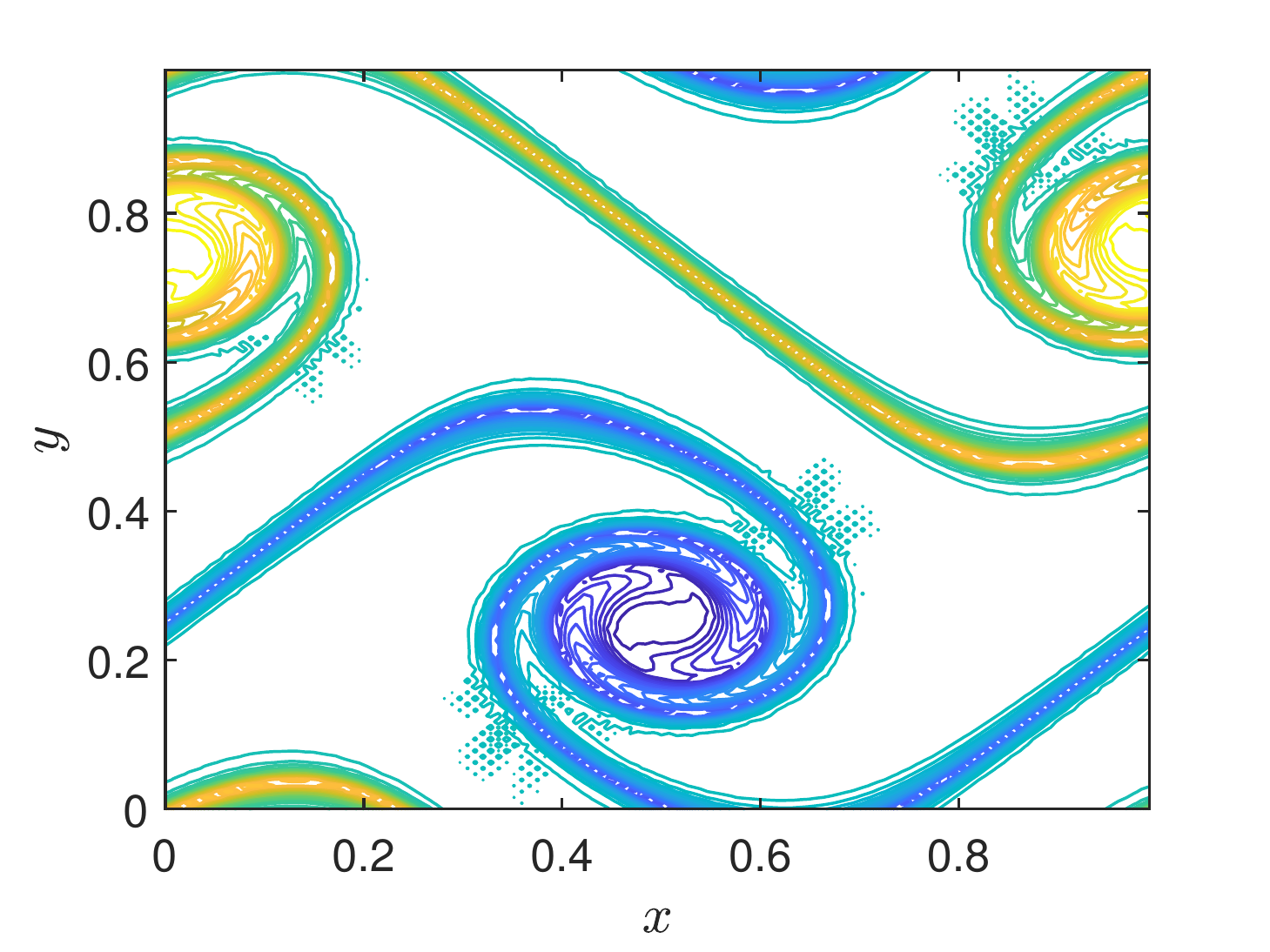}
	\end{minipage}
	\begin{minipage}{0.4\textwidth}
		\centering
		\includegraphics[width=5.3cm]{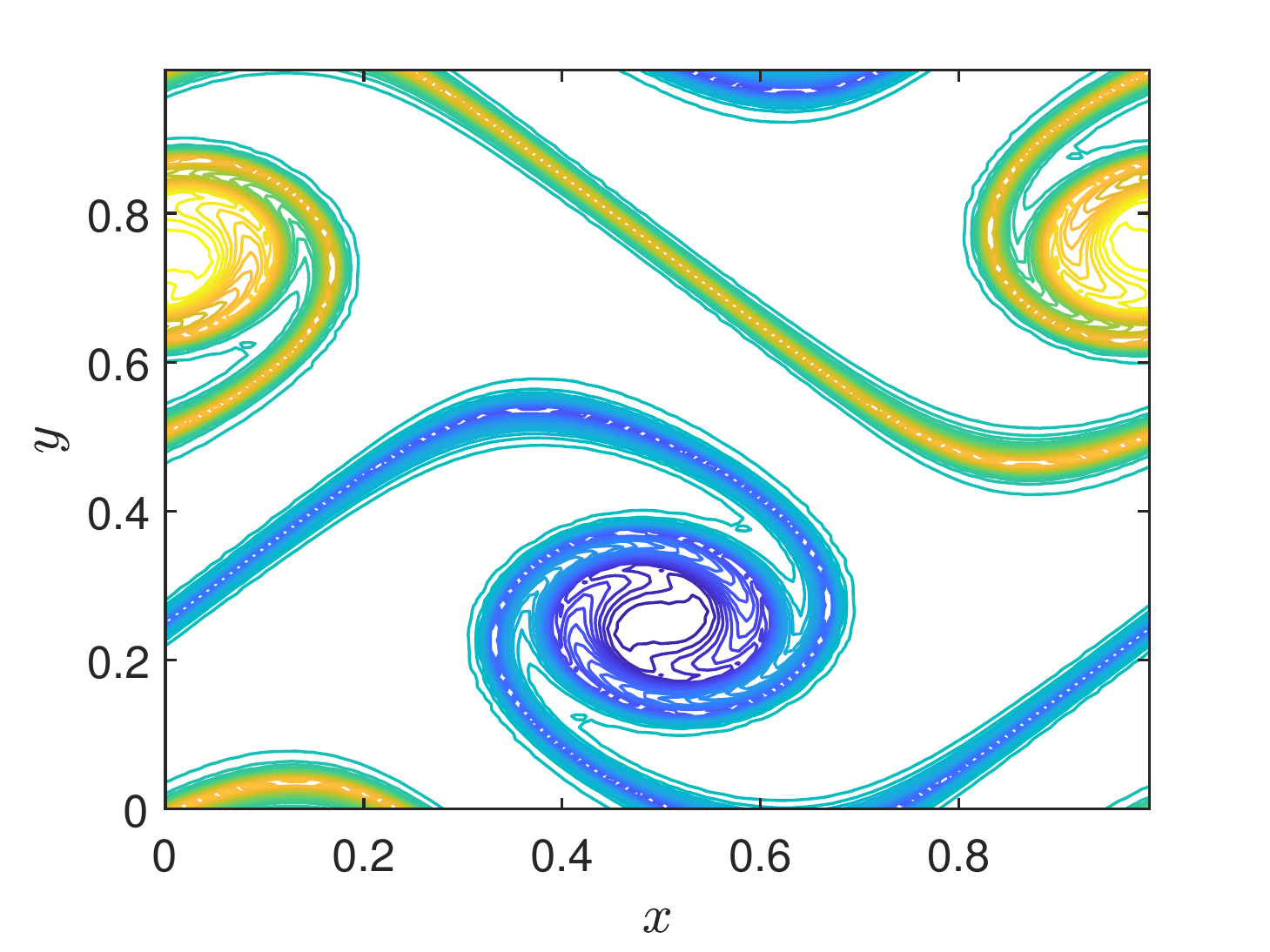}
	\end{minipage}
	\begin{minipage}{0.4\textwidth}
		\centering
		\includegraphics[width=5.3cm]{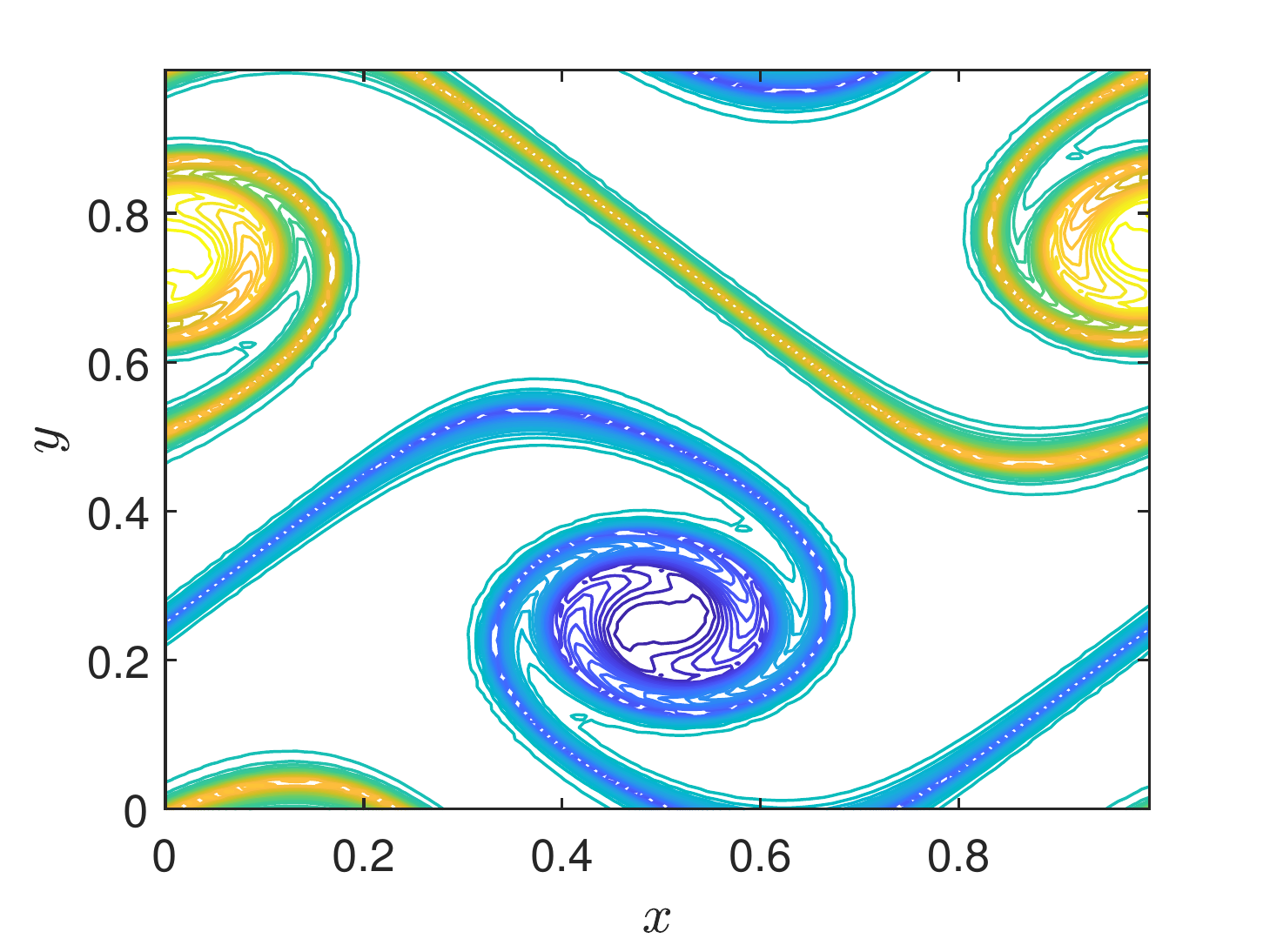}
	\end{minipage}
	\caption{Example\,\ref{ex:NS}(Case B). The evolution of vorticity of Navier-Stokes equation with $\rho=30, \nu=5e-5$ using EOP-GSAV/BDF$k$, ($k=1, 2, 3, 4$) schemes.}
	\label{Fig:NS-double-shear-comparison}
\end{figure}

\end{example}
\section*{Acknowledgement}
No potential conflict of interest was reported by the author. We would like to acknowledge the assistance of volunteers in putting together this example manuscript and supplement.
\bibliographystyle{siamplain}
\bibliography{Reference}

\begin{thebibliography}{10}

\bibitem{antoine2021scalar}
{\sc X.~Antoine, J.~Shen, and Q.~Tang}, {\em Scalar auxiliary
  variable/{L}agrange multiplier based pseudospectral schemes for the dynamics
  of nonlinear {S}chr{\"o}dinger/{G}ross-{P}itaevskii equations}, Journal of
  Computational Physics, 437 (2021), p.~110328.

\bibitem{baskaran2013convergence}
{\sc A.~Baskaran, J.~S. Lowengrub, C.~Wang, and S.~M. Wise}, {\em Convergence
  analysis of a second order convex splitting scheme for the modified phase
  field crystal equation}, SIAM Journal on Numerical Analysis, 51 (2013),
  pp.~2851--2873.

\bibitem{chen1998applications}
{\sc L.~Q. Chen and J.~Shen}, {\em Applications of semi-implicit
  {F}ourier-spectral method to phase field equations}, Computer Physics
  Communications, 108 (1998), pp.~147--158.

\bibitem{cheng2018multiple}
{\sc Q.~Cheng and J.~Shen}, {\em Multiple scalar auxiliary variable ({MSAV})
  approach and its application to the phase-field vesicle membrane model}, SIAM
  Journal on Scientific Computing, 40 (2018), pp.~A3982--A4006.

\bibitem{du2019maximum}
{\sc Q.~Du, L.~Ju, X.~Li, and Z.~Qiao}, {\em Maximum principle preserving
  exponential time differencing schemes for the nonlocal {A}llen--{C}ahn
  equation}, SIAM Journal on numerical analysis, 57 (2019), pp.~875--898.

\bibitem{du2021maximum}
{\sc Q.~Du, L.~Ju, X.~Li, and Z.~Qiao}, {\em Maximum bound principles for a
  class of semilinear parabolic equations and exponential time-differencing
  schemes}, SIAM Review, 63 (2021), pp.~317--359.

\bibitem{eyre1998unconditionally}
{\sc D.~J. Eyre}, {\em Unconditionally gradient stable time marching the
  {C}ahn-{H}illiard equation}, MRS Online Proceedings Library (OPL), 529
  (1998), p.~39.

\bibitem{hou2021robust}
{\sc D.~Hou and C.~Xu}, {\em Robust and stable schemes for time fractional
  molecular beam epitaxial growth model using {SAV} approach}, Journal of
  Computational Physics, 445 (2021), p.~110628.

\bibitem{jiang2022improving}
{\sc M.~Jiang, Z.~Zhang, and J.~Zhao}, {\em Improving the accuracy and
  consistency of the scalar auxiliary variable ({SAV}) method with relaxation},
  Journal of Computational Physics, 456 (2022), p.~110954.

\bibitem{ju2022stabilized}
{\sc L.~Ju, X.~Li, and Z.~Qiao}, {\em Stabilized exponential-{SAV} schemes
  preserving energy dissipation law and maximum bound principle for the
  {A}llen--{C}ahn type equations}, Journal of Scientific Computing, 92 (2022),
  p.~66.

\bibitem{ju2018energy}
{\sc L.~Ju, X.~Li, Z.~Qiao, and H.~Zhang}, {\em Energy stability and error
  estimates of exponential time differencing schemes for the epitaxial growth
  model without slope selection}, Mathematics of Computation, 87 (2018),
  pp.~1859--1885.

\bibitem{li2019efficient}
{\sc Q.~Li, L.~Mei, X.~Yang, and Y.~Li}, {\em Efficient numerical schemes with
  unconditional energy stabilities for the modified phase field crystal
  equation}, Advances in Computational Mathematics, 45 (2019), pp.~1551--1580.

\bibitem{li2020stability}
{\sc X.~Li and J.~Shen}, {\em Stability and error estimates of the {SAV}
  fourier-spectral method for the phase field crystal equation}, Adv Comput
  Math, 46 (2020), p.~48.

\bibitem{li2022stability}
{\sc X.~Li, W.~Wang, and J.~Shen}, {\em Stability and error analysis of {IMEX}
  {SAV} schemes for the magneto-hydrodynamic equations}, SIAM Journal on
  Numerical Analysis, 60 (2022), pp.~1026--1054.

\bibitem{lin2019numerical}
{\sc L.~Lin, Z.~Yang, and S.~Dong}, {\em Numerical approximation of
  incompressible {N}avier-{S}tokes equations based on an auxiliary energy
  variable}, Journal of Computational Physics, 388 (2019), pp.~1--22.

\bibitem{liu2020exponential}
{\sc Z.~Liu and X.~Li}, {\em The exponential scalar auxiliary variable
  ({E-SAV}) approach for phase field models and its explicit computing}, SIAM
  Journal on Scientific Computing, 42 (2020), pp.~B630--B655.

\bibitem{shen2018scalar}
{\sc J.~Shen, J.~Xu, and J.~Yang}, {\em The scalar auxiliary variable ({SAV})
  approach for gradient flows}, Journal of Computational Physics, 353 (2018),
  pp.~407--416.

\bibitem{shen2019new}
{\sc J.~Shen, J.~Xu, and J.~Yang}, {\em A new class of efficient and robust
  energy stable schemes for gradient flows}, SIAM Review, 61 (2019),
  pp.~474--506.

\bibitem{shen2010numerical}
{\sc J.~Shen and X.~Yang}, {\em Numerical approximations of {A}llen-{C}ahn and
  {C}ahn-{H}illiard equations}, Discrete Contin. Dyn. Syst, 28 (2010),
  pp.~1669--1691.

\bibitem{xu2006stability}
{\sc C.~Xu and T.~Tang}, {\em Stability analysis of large time-stepping methods
  for epitaxial growth models}, SIAM Journal on Numerical Analysis, 44 (2006),
  pp.~1759--1779.

\bibitem{yang2017linearly}
{\sc X.~Yang and D.~Han}, {\em Linearly first-and second-order, unconditionally
  energy stable schemes for the phase field crystal model}, Journal of
  Computational Physics, 330 (2017), pp.~1116--1134.

\bibitem{yang2018linear}
{\sc X.~Yang, J.~Zhao, and X.~He}, {\em Linear, second order and
  unconditionally energy stable schemes for the viscous {C}ahn--{H}illiard
  equation with hyperbolic relaxation using the invariant energy quadratization
  method}, Journal of Computational and Applied Mathematics, 343 (2018),
  pp.~80--97.

\bibitem{yang2017numerical}
{\sc X.~Yang, J.~Zhao, and Q.~Wang}, {\em Numerical approximations for the
  molecular beam epitaxial growth model based on the invariant energy
  quadratization method}, Journal of Computational Physics, 333 (2017),
  pp.~104--127.

\bibitem{yang2020roadmap}
{\sc Z.~Yang and S.~Dong}, {\em A roadmap for discretely energy-stable schemes
  for dissipative systems based on a generalized auxiliary variable with
  guaranteed positivity}, Journal of Computational Physics, 404 (2020),
  p.~109121.

\bibitem{zhang2022generalized}
{\sc Y.~Zhang and J.~Shen}, {\em A generalized {SAV} approach with relaxation
  for dissipative systems}, Journal of Computational Physics,  (2022),
  p.~111311.

\bibitem{zhao2017numerical}
{\sc J.~Zhao, Q.~Wang, and X.~Yang}, {\em Numerical approximations for a phase
  field dendritic crystal growth model based on the invariant energy
  quadratization approach}, International Journal for Numerical Methods in
  Engineering, 110 (2017), pp.~279--300.

\end{thebibliography}

\end{document}